\documentclass[11pt]{amsart}
\usepackage{amssymb}
\usepackage{amsmath}
\usepackage{bbm}
\usepackage{graphicx} 
\usepackage[english]{babel} 
\usepackage{enumitem}

\usepackage{float}
\usepackage{mathtools}
\usepackage{hyperref}
\usepackage[noadjust]{cite}
\usepackage[margin=3.2cm]{geometry}
\hypersetup{
    colorlinks=true,       % false: boxed links; true: colored links
    linkcolor=blue,          % color of internal links
    citecolor=blue,        % color of links to bibliography
    filecolor=blue,      % color of file links
    urlcolor=blue           % color of external links
}

\usepackage[pdf]{pstricks} % Compile using XeLatex, View PDF
\usepackage{epsfig}
\usepackage{pst-grad} % For gradients
\usepackage{pst-plot} % For axes
\usepackage[space]{grffile} % For spaces in paths
\usepackage{etoolbox} % For spaces in paths
\makeatletter % For spaces in paths
\patchcmd\Gread@eps{\@inputcheck#1 }{\@inputcheck"#1"\relax}{}{}
\makeatother

\usepackage[all]{xy}
\usepackage{mathrsfs}
\usepackage{graphics}
\usepackage{amsthm}
\theoremstyle{plain}\newtheorem{theorem}{Theorem}[section]\newtheorem{proposition}[theorem]{Proposition}\newtheorem{lemma}[theorem]{Lemma}\newtheorem{corollary}[theorem]{Corollary}
\def\eq{\coloneqq}

\theoremstyle{definition}\newtheorem{example}[theorem]{Example}\newtheorem{remark}[theorem]{Remark}%\newtheorem{construction}[subsection]{Construction}

\def\C{\mathbb{C}}\def\Z{\mathbb{Z}}\def\R{\mathbb{R}}\def\DD{\mathbb{D}}

\def\ZZ2{\mathbb{\Z/ 2\Z}}\def\Zp{\mathbb{\Z/ p\Z}}
\def\sb{\subset}\def\su{\subset}
\def\lb{\langle}\def\rb{\rangle}\def\ot{\otimes}\def\t{\times}\def\sm{\setminus}
\def\c{\gamma}\def\v{\varphi}

\def\a{\alpha}\def\b{\beta}\def\d{\delta}\def\e{\epsilon}\def\s{\sigma}\def\De{\Delta}\def\D{\Delta}\def\La{\Lambda}\def\la{\lambda}\def\la{\lambda}\def\p{\partial}\def\S{\Sigma}
\def\ma{\mathcal}\def\ov{\overline}\def\wh{\widehat}\def\wt{\widetilde}
\def\D{\mathbb{D}}\def\DD{\mathbb{D}^2}

\def\ad{\text{ad}}\def\gl{\mathfrak{gl}}
\def\sl2{\mathfrak{sl}_2}
\def\su2{\mathfrak{su}(2)}
\def\tr{\text{tr}\,}\def\Aut{\text{Aut}\,}
\def\id{\text{id}}\def\Ker{\text{Ker}\,}
\def\Hom{\text{Hom}\,}
\def\Vect{\text{Vect}\,}
\def\Spin{\text{Spin}}
\def\Ham{\text{Ham}}

\def\Spin{\text{Spin}}

\def\TH3{\Theta_3^{H}}

\def\Usl2{U_q(\sl2)}\def\usl2{\wt{U}_q(\sl2)}\def\Uqgl11{U_q(\mathfrak{gl}(1|1))}\def\Uqsl11{U_q(\mathfrak{sl}(1|1))}

\def\Rep{\text{Rep}}

\def\II1{\text{II}_1}

\def\Vect{\text{Vect}}

\def\mod2{\ (mod \ 2)}

\def\kk{\mathbb{K}}\def\Vect{\text{Vect}_{\kk}}\def\gl{\mathfrak{gl}}\def\gl11{\mathfrak{gl}(1|1)}

\def\gg{\mathfrak{g}}

\def\Ha{A_{\a}}\def\Ham{A_{\a^{-1}}}\def\Hb{A_{\b}}\def\Sa{S_{\a}}\def\Sam{S_{\a^{-1}}}\def\uH{\underline{A}} \def\uHH{\uH=\{\Ha\}_{\a\in G}}
\def\Hg{A_g}\def\Hab{A_{\a\b}}
\def\Hab{A_{\a\b}}

\def\Deab{\De_{\a,\b}}

\def\Aut{\text{Aut}}

\def\rhoc{\rho\ot h}

\def\rhoc{\rho\ot h}

\def\rhoc{\rho\ot h}

\def\int{\boldsymbol{\mu}}

\def\rhoc{\rho\ot h}

\def\am{\a^{-1}}
\def\ma{m_{\a}}\def\oa{1_{\a}}\def\ga{g_{\a}}

\def\CC{\mathcal{C}}

\def\va{v_{\a}}

\def\vaa{\v_{\a}}

\def\Vect{\text{Vect}}

\def\Sa{S_{\a}}

\def\uDH{\underline{D(H)}}\def\uDHH{\underline{D(H)}=\{D(H)_{\a}\}_{\a\in\Aut(H)}}

\def\aa{\bold{\a}}\def\bb{\bold{\b}}
\def\CC{\mathcal{C}}

\def\rH{r_{H}}

\def\gg{\bold{g}}

\def\rt{Z}
\def\zz{\boldsymbol{\zeta}}

\def\half{\frac{1}{2}}

\begin{document}

% OTHER POSSIBILITIES" TWISTING rt INVARIANTS OF DRINFELD DOUBLES VIA FOX CALCULUS

\title[Twisted Kuperberg invariants of knots via twisted Drinfeld doubles]{Twisted Kuperberg invariants of knots and Reidemeister torsion via twisted Drinfeld doubles}

%\title[Reshetikhin-Turaev invariants of relative centers of crossed products]{Reshetikhin-Turaev invariants of relative Drinfeld centers of crossed products}%{Reshetikhin-Turaev invariants of knots with flat connections from graded Drinfeld doubles}
\author{Daniel L\'opez Neumann}
\email{dlopezne@indiana.edu}

\maketitle

\begin{abstract}
In this paper, we consider the Reshetikhin-Turaev invariants of knots in the three-sphere obtained from a twisted Drinfeld double of a Hopf algebra, or equivalently, the relative Drinfeld center of the crossed product $\Rep(H)\rtimes\Aut(H)$. These are quantum invariants of knots endowed with a homomorphism of the knot group to $\Aut(H)$. We show that, at least for knots in the three-sphere, these invariants provide a non-involutory generalization of the Fox-calculus-twisted Kuperberg invariants of sutured manifolds introduced previously by the author, which are only defined for involutory Hopf algebras. In particular, we describe the $SL(n,\C)$-twisted Reidemeister torsion of a knot complement as a Reshetikhin-Turaev invariant.%out of an arbitrary $\Z$-graded doubly-balanced Hopf algebra of finite dimension, we obtain knot polynomials that generalize twisted Alexander polynomials. %As an application we show that, when $H$ is $\Z$-graded, our invariants can be upgraded to polynomial invariants that give lower bounds to the knot genus, generalizing work of Friedl and Kim.

%When $H$ is $\N$-graded, this invariant lifts to a polynomial invariant of knots, this is non-trivial when $H$ is non-semisimple. When $H$ is an exterior algebra, we show that our construction specializes to the twisted Alexander polynomials of knots. 
\end{abstract}

\section[Introduction]{Introduction}
%Note that the present work is only about  knot invariants so semisimplicity does not seems to play a big role. It actually does, since we want an infinite group as target of pi1. So even if we evaluate a twisted invariant on a ss rep, it cannot be recovered from a semisimple theory since the target is infinite!!!

% Maybe: Talk first about q-invs, mark difference between ss/non-ss.
%Then introduce equivariant version. Now it is more natural to mark a difference non-ss/ss.

% 1st paragraph: braided cats and rt invariants. Survey style.

The theory of braided and ribbon monoidal categories is a rich subject at the intersection of various fields, such as representation theory, physics and low dimensional topology. %I could even give def of braided monoidal.
Their relevance in topology stems from the fact that a ribbon monoidal category produces an entire family of topological invariants of links and tangles in the three-sphere, known as Reshetikhin-Turaev invariants \cite{RT1}. One of the best known invariants in this family is the Jones polynomial, which is obtained from
the category of modules over the quantum group of $\sl2$. These link invariants can be extended to invariants of closed 3-manifolds as well and even to a powerful object called a topological quantum field theory (TQFT) \cite{Turaev:BOOK1}. Even thirty years after their introduction, many fundamental questions about these invariants remain open, such as their topological and geometrical content, categorification, etc. %One sentence about their top content, it is mysterious or something alike.

% A Quantum invariant is more general than Reshetikhin-Turaev: e.g. Witten's path integral, Vassiliev invariants are quantum, but not rt??
\medskip

% MENTION MORE WORKS: INVARIANTS FROM HOPF ALGEBRAS (NON-SS), UNIVERSAL INVARIANTS (OHTSUKI?).

% can we call it EQUIVARIANT??

% 2nd paragraph: Equivariant version of the above. Examples of previous work. Survey style.
\def\DD{\mathcal{D}}

The Reshetikhin-Turaev invariants of links and 3-manifolds admit a refinement when the ribbon categories involved are graded by a group $G$. Here, the algebraic input is a $G$-crossed ribbon category, that is, a category $\CC$ which is a disjoint union of subcategories $\CC=\coprod_{g\in G}\CC_g$ admitting a monoidal structure, $G$-action and ribbon structure compatible with the grading. %The theory of $G$-extensions $\CC$ of a monoidal category $\CC_e$ is by now well-understood 
The output is a topological invariant of $G$-tangles, that is, tangles $T\sb \R^2\t[0,1]$ endowed with a flat $G$-connection on a principal $G$-bundle of the complement $X_T$, or equivalently, a homomorphism $\pi_1(X_T)\to G$. This extension was introduced by Turaev in \cite{Turaev:homotopy}. %Special instances (and modifications) of this construction related to quantum groups (and with abelian group) have been studied by Akutsu-Deguchi-Ohtsuki \cite{ADO} and Costantino-Geer-Patureau \cite{CGP:non-semisimple}. A modification of Turaev's construction that allows to obtain invariants with a non-abelian flat connection out of quantum groups was proposed by Kashaev and Reshetikhin \cite{KR:tangles-flat-connections}, and further studied in the works \cite{BGPR:holonomy}, \cite{McPhail-Snyder:holonomy}. 
Special instances (and modifications) of this construction related to quantum groups have been studied in various works \cite{ADO, CGP:non-semisimple, KR:tangles-flat-connections, BGPR:holonomy, McPhail-Snyder:holonomy}. When the $G$-crossed ribbon categories are semisimple (or more precisely, $G$-modular), Turaev and Virelizier extended this construction to a ``homotopy quantum field theory" (HQFT) \cite{TV:HQFT-II}. It has to be noted that the theory of $G$-extensions of a given monoidal category $\DD$ (i.e. $G$-graded $\CC$ with $\CC_1=\DD$), including extensions of quantum groups at roots of unity, is by now well-understood \cite{ENO:fusion-homotopy, DEN:autoequivalences, DN:braided-Picard, DN:Picard}. %The $G$-braided extensions of a braided category $\CC$ can be classified through a group (or rather a groupoid) called the Picard group of $\CC$ \cite{DEN:autoequivalences}, \cite{DN:Picard} .

%The G-category theory is also a rich subject in itself: refer to fusion and homotopy EGNO, 3-categorical perspective or other stuff? Maybe I can say: the classification of G-extensions of a category is by know well-understood, cite ENO + DN. Maybe mention it has to do with Aut(C)?

\medskip

%COULD STArt AS WELL: A first step to understand equiv theory is to understand R-torsion as part of it.

%3rd paragraph: The study of invariants of pairs (K,\rho) seem natural. For instance, Reidemeister torsion is such a thing. But it will be more difficult: it has to be non-semisimple.

The study of quantum invariants of knots and 3-manifolds endowed with a representation of their fundamental groups seems very natural and (possibly) powerful from a topological point of view. On the one hand, certain classes of 3-manifolds come equipped with a canonical representation $\pi_1(M)\to G$, such as hyperbolic 3-manifolds (with $G=PSL(2,\C)$) or knot complements in the three sphere (with $G=\Z=H_1(M)$). On the other hand, classical invariants such as twisted Reidemeister torsion are by definition invariants of pairs $(M,\rho)$ where $M$ is a 3-manifold and $\rho:\pi_1(M)\to GL(n,\C)$ is an homomorphism \cite{Turaev:BOOK2}. The fact that this invariant is defined for any $\rho$ leads to interesting applications in geometric topology (see the surveys \cite{FV:survey} for topological applications and \cite{Porti:survey} for applications in hyperbolic geometry). %For instance, the Alexander polynomial of a knot in the three-sphere is essentially the Reidemeister torsion of the complement at the abelian representation $\rho:\pi_1(M)\to \Z=H_1(M)$ (also called Milnor torsion). 
These considerations lead to the question of whether quantum invariants from $G$-crossed ribbon categories may have applications to geometric topology generalizing those of Reidemeister torsion. 

%NOT QUITE FOR NON-SS: Friedl-Vidussi only use FINITE GROUPS as target!Note that in all these instances, the target group $G$ is an infinite group. Since semisimple $G$-crossed ribbon categories exist only for finite groups $G$ \cite{ENO:fusion}, 

\medskip

%WHERE TO SAY MORE ABOUT G-CROSSED: ENO CLASSIFICATION THMS ETC, DN IN NON-SS CASE.

%IS IT REALLY TRUE THAT ADO AND CGP ARE G-rt??

%I COULD ALSO STArt HERE DESCRIBING MY WORK. I ALREADY MENTIONED ALL THESE WORKS SO IT IS OK.

A natural first step to study the above question is to realize twisted Reidemeister torsion as a special case of the Reshetikhin-Turaev invariant of $G$-tangles of \cite{Turaev:homotopy}. Some of the aforementioned works do that in certain cases: \cite{ADO} for multivariable Alexander, \cite{BCGP} for abelian Reidemeister torsion. However, to the knowledge of the author, the works that find (special cases of) non-abelian Reidemeister torsion as a quantum invariant do not rely directly on $G$-crossed ribbon categories. For instance, in \cite{McPhail-Snyder:holonomy} McPhail-Snyder finds the $SL(2,\C)$-twisted Reidemeister torsion of a link complement using the ``holonomy braidings" of Kashaev-Reshetikhin \cite{KR:tangles-flat-connections, KR:braiding-roots-of-unity}. %In \cite{McPhail-Snyder:holonomy} McPhail-Snyder obtained the $SL(2,\C)$-twisted torsion of a link complement, though using the ``holonomy braidings" of Kashaev-Reshetikhin \cite{KR:tangles-flat-connections, KR:braiding-roots-of-unity} instead of $G$-crossed ribbon categories. 
Another example is our previous work \cite{LN:twisted}, where, for involutory Hopf algebras $H$, we extended Kuperberg's invariant \cite{Kup1} to an invariant $I_H^{\rho}(M,\c)$ of balanced sutured manifolds $(M,\c)$ endowed with a homomorphism $\rho:\pi_1(M)\to\Aut(H)$, and we showed that the $SL(n,\C)$-twisted Reidemeister torsion was the special case when $H$ is an exterior algebra $\La(\C^n).$

%showed that the $SL(n,\C)$-twisted Reidemeister torsion of sutured manifolds was a special case of an equivariant version of (involutory) Kuperberg invariants \cite{Kup1}. 

%owever, the construction of \cite{LN:twisted} only applies to a very restrictive class of Hopf algebras and its relation to $G$-crossed ribbon categories is not clear.
\medskip

 %The definition of this invariant relied on the semidirect product $H\rtimes\kk[\Aut(H)]$ or, dually and categorically speaking, the crossed product category $\wt{\CC}=\Rep(H)\rtimes \Aut(H)$. %However, since our construction only applied to involutory Hopf algebras, it didn't applied to quantum groups.

%There, starting with a finite dimensional involutory Hopf (super)algebra $H$ over a field $\kk$, we produced a quantum invariant $$I_H^{\rho}(M,\c,\ss,\o)\in\kk$$ of balanced sutured manifolds $(M,\c)$ endowed with a homomorphism $\pi_1(M)\to\Aut(H)$ (and some additional structure), which is computed via Fox calculus in a certain sense. When $H$ is $\N$-graded and $(M,\c)$ is a (sutured) knot complement, $\rho$ extends to an homomorphism $\rhoc$ into $\Aut(H')$, where $H'=H\ot\kk[t^{\pm 1}]$, so one gets a knot polynomial $$I_H^{\rhoc}(M,\c,\ss,\o)\in\kk[t^{\pm 1}].$$When $H=\La(\C^n)$ is an exterior algebra, we showed that $I_H^{\rhoc}$ specializes to the $SL(n,\C)$-twisted Reidemeister torsion of a sutured manifold, in particular, it contains the twisted Alexander polynomials \cite{Lin:representations}, \cite{Wada:twisted}, \cite{KL:twisted} of a knot.

%SOME MORE ABOUT OUR I?? It is flexible, it leads to twisted knot polynomials for non-ss Hopf algebras.
\def\DD{\mathcal{D}}\def\ZZ{\mathcal{Z}}\def\wCC{\wt{\CC}}\def\ZZCC{\ZZ_{\CC}(\wCC)}

The purpose of the present work is to find a general class of $G$-crossed ribbon categories for which the associated invariants contain $SL(n,\C)$-twisted Reidemeister torsion, in particular twisted Alexander polynomials \cite{Lin:representations, Wada:twisted, KL:twisted}, as a special case. We achieve this by finding the Reshetikhin-Turaev version of the twisted Kuperberg invariants $I_H^{\rho}$ of our previous work \cite{LN:twisted}, in particular extending that work to non-involutory Hopf algebras (at least for knots in the three-sphere). To do this, we need to take an appropriate Drinfeld double of the graded Hopf object considered in \cite{LN:twisted}. This is motivated by various works that relate Kuperberg or Turaev-Viro invariants to Hennings or Reshetikhin-Turaev invariants of the corresponding Drinfeld double, see \cite{CC:ontwoinvariants, TV:HQFT-III}. Particularly relevant to us is the work of Turaev and Virelizier in the semisimple $G$-crossed setting \cite{TV:HQFT-III}, where it is shown that the Turaev-Viro HQFT and the Reshetikhin-Turaev HQFT are related via the relative Drinfeld center construction of Gelaki-Naidu-Nikshych \cite{GNN:centers}. Since our previous work depended (implicitly) on a crossed product category $\CC=\Rep(H)\rtimes\Aut(H)$ (see Remark \ref{remark: conventions on IH}), we consider here its relative Drinfeld center $\ZZ_{\Rep(H)}(\CC)$. This is a braided $\Aut(H)$-crossed category and is equivalent to the category of modules over the ``twisted Drinfeld double" introduced by Virelizier in \cite{Virelizier:Graded-QG} (see Proposition \ref{prop: TDD is REL DRINFELD}), here denoted $\uDH$, which is a quasi-triangular Hopf group-coalgebra in the sense of \cite{Turaev:homotopy, Virelizier:Hopfgroup}. 
%WE ACHIEVE THIS BY FINDING THE RESHETIKHIN-TURAEV VERSION OF OUR PREVIOUS WORK. TO motive the results, we recall the construction of LN.

\medskip

\medskip
In light of the above, we consider the Reshetikhin-Turaev invariant of a knot obtained from a twisted Drinfeld double. We will work in the setting of universal quantum invariants. Recall that in the untwisted setting, if $H$ is a finite dimensional Hopf algebra for which $D(H)$ is ribbon and if $K$ is a framed knot presented as the closure of a long knot $T$, then the Reshetikhin-Turaev invariant of $T$ colored by the regular $D(H)$-module is an invariant of $K$. This invariant is left multiplication by an element $\rt_{D(H)}(K)$ of the center of $D(H)$ and recovers all the Reshetikhin-Turaev invariants of $K$ via the quantum trace (or a modified trace if the color is projective), hence called the universal invariant. In the twisted setting, if $G\sb\Aut(H)$ is a subgroup for which the twisted Drinfeld double is $G$-ribbon, a similar procedure leads to a ``twisted" universal quantum invariant $$\rt_{\uDH}^{\rho}(T)\in D(H)$$ where $\rho:\pi_1(X_T)\to G$ is an homomorphism. 
%To give a down-to-earth construction of the quantum invariants obtained from $\ZZCC$ and to show that they contain our previous work, we first define a graded Hopf object $\uDH$ (a quasi-triangular Hopf group-coalgebra in the sense of \cite{Turaev:homotopy, Virelizier:Hopfgroup}) that represents this category, that is, $\ZZCC=\Rep(\uDH)$. This Hopf object had already been defined by Virelizier in \cite{Virelizier:Graded-QG} and called a ``twisted Drinfeld double". 
%In particular, we can evaluate over the counit and get an invariant $\e_{D(H)}(rt_{\uDH}^{\rho}(K))\in\kk$. This invariant is trivial at $\rho\equiv 1$, so it cannot be read from non-equivariant quantum invariants. %somehow measures non-triviality of the extension $\CC\sb \CC\rtimes\Aut(H)$. 
If $H$ is $\Z$-graded, one can use the abelianization $h:\pi_1(X_T)\to\Z$ to extend $\rho$ to a representation $\rhoc$ with coefficients in $\kk[t^{\pm 1}]$. This leads to a polynomial invariant $$\rt_{\uDH}^{\rhoc}(T)\in D(H)[t^{\pm \half}].$$
These are not necessarily invariants of the closure $(K,\rho)$ (except for abelian $\rho$), but they are after an appropriate evaluation. For instance, $\e_{D(H)}(\rt_{\uDH}^{\rhoc}(T))\in\kk[t^{\pm \half}]$ is a polynomial invariant of $(K,\rho)$ and after a framing normalization, one obtains a polynomial invariant $$P_H^{\rho}(K,t)\in\kk[t^{\pm 1}]$$ of unframed $G$-knots. Our main theorem (Theorem \ref{Thm: main theorem}) is that this evaluation recovers the ``Fox-calculus twisted" Kuperberg invariant of our previous work \cite{LN:twisted} (which is defined for involutory $H$) as follows: $$P^{\rho}_H(K,t)\dot{=}I^{\rhoc}_H(M,\c).$$
%Here $\dot{=}$ is equality up to multiplication by $\pm\rH(\rho(g))^{\frac{k}{2}}t^{\frac{k|\La|}{2}}$ for some $g\in\pi_1(M),k\in\Z$ (here $|\La|$ is the $\Z$-degree of the cointegral of $H$) and $(M,\c)$ is the sutured manifold associated to $(S^3,K)$. 
Here $\dot{=}$ is equality up to multiplication by $\pm\rH(\rho(g))t^{k|\La|}$ for some $g\in\pi_1(X_K), k\in\Z$ ($\rH$ is defined in (\ref{eq: rH}) below and $|\La|$ is the $\Z$-degree of the cointegral of $H$) and $(M,\c)$ is the sutured manifold associated to $(S^3,K)$. This theorem is stated for involutory $H$ so that the right hand side is a well-defined invariant, but one should still think that, for non-involutory $H$, $\rho$ is twisting via Fox calculus the tensors of the universal invariant $Z_{D(H)}(K)$. 

\medskip
As a corollary of our main theorem and our previous work \cite{LN:twisted} we get that the twisted Reidemeister torsion $\tau^{\rhoc}$ is a Reshetikhin-Turaev invariant of a twisted Drinfeld double of an exterior algebra $\La(\C^n)$ (Corollary \ref{cor: twisted Alex is twisted RT}): $$P^{\rho}_{\La(\C^n)}(K,t)\dot{=}\tau^{\rhoc}(S^3\sm K,m)$$
where $m\sb \p(S^3\sm K)$ is a meridian and $\rho:\pi_1(S^3\sm K)\to SL(n,\C)$. Thus, the twisted polynomial invariant $P_{H}^{\rho}(K,t)$ generalizes twisted Alexander polynomials to arbitrary finite-dimensional $\Z$-graded Hopf algebras (with ribbon double). %, such as Taft algebras at odd roots of unity (or even roots of unity if an appropriate extension is taken).

%Note that the left hand side is defined for non-involutory Hopf algebras as well. The argument of the proof of the above theorem applies to any Hopf algebra, but the right hand side is only defined for involutory Hopf algebras, that is why we restrict to this setting. Thus, one should still think that the homomorphism $\rho$ is simply ``twisting via Fox calculus" the formula for the untwisted universal invariant $\rt_{D(H)}(K)$.
\medskip

%we recover the twisted Alexander polynomial \cite{Lin:representations}, \cite{Wada:twisted}, \cite{KL:twisted} (see also \cite{FV:survey}) as follows:

%One could as well bosonize $\La(\C^n)$ and get a Hopf algebra (not super) whose twisted Drinfeld double recovers twisted Reidemeister torsion. When $\dim(V)=1$ the bosonization of $\La(\C)$ is the Borel part of $\Uq$ at $q=i$. When $\dim(V)=2$, the bosonization of $\La(\C^2)$ is the quotient of $\Uq$ at $q=i$ by the relation $K^2=1$.

%Equivalently, we could have used the bosonization $H$ of $\La(\C^n)$ which, for $\dim(V)=1$ is a Taft algebra at $q=i$ (or a quotient of the Borel of $U_i(\sl2)$). %Therefore we get an ``Alexander poly at roots of unity" but is it interesting??

%Therefore, it is not at all unnatural to take Drinfeld doubles of entire quantum groups in order to get invariants containing $SL(2,\C)$ informatio extenn.

\medskip

The plan of the paper is the following. In Section \ref{section: Hopf algebras} we recall the notions from Hopf algebra theory that we need and we study the twisted Drinfeld double of a Hopf algebra. In Section \ref{section: quantum invariants of G-tangles} we recall the construction of invariants of $G$-tangles of \cite{Turaev:homotopy}, though in the universal setting, so the tangles are not colored by modules. Here we define the invariants $Z^{\rho}_{\uDH}(T), Z^{\rhoc}_{\uDH}(T)$ and the knot polynomial $P_H^{\rho}(K,t)$ mentioned above. In Section \ref{section: Fox calculus from TDD} we state and prove our main theorem.

%\subsection{Future directions} Use the work \cite{DN:braided-Picard}.

\subsection{Acknowledgments} The author would like to thank Dylan Thurston and Roland van der Veen for various helpful conversations.

%Moreover, for any $\N$-graded Hopf algebra $H$, the degree of this polynomial invariant gives a lower bound to the knot genus: $$\deg \ p_K(t)\leq 2g(K)n$$where $n$ is the top degree of $H$ (= degree of the cointegral). This can be shown by using Roland's idea (put the Seifert surface in ``standard" form).\medskip

\section{The twisted Drinfeld double of a Hopf algebra}
\label{section: Hopf algebras}

We begin by recalling some basic notions and notation from Hopf algebra theory and Hopf group-coalgebras, following mainly \cite{Radford:BOOK, Virelizier:Hopfgroup}. The twisted Drinfeld double $\uDH$ is defined in Subsection \ref{subs: Hopf twisted DH}. In Subsection \ref{subs: Hopf ribbon element KR thm} we discuss under which conditions $\uDH$ is ribbon. In Subsection \ref{subs: Hopf Rep theory} we relate $\uDH$ to the relative Drinfeld center of \cite{GNN:centers}. 
\medskip

\def\Vect{\text{Vect}}
\def\SVect{\text{SVect}}

In what follows we denote by $\Vect$ the category of vector spaces over a field $\kk$ and by $\SVect$ that of super-vector spaces and degree-preserving linear maps. 

\subsection{Tensor network notation} We will use tensor network notation for tensors of the form $T:V_1\ot\dots\ot V_n\to W_1\ot\dots\ot W_m$ where the $V_i$'s and $W_j$'s are vector spaces. Such a tensor is denoted by a diagram with $n$ incoming arrows, or inputs, and $m$ outcoming arrows, or outputs. The diagram is read from left to right, and the arrows are ordered from top to bottom. Thus, a tensor $T:V_1\ot V_2\to W$ is denoted by 
\begin{figure}[H]
\centering
\begin{pspicture}(0,-0.30805147)(1.4684552,0.30805147)
\psline[linecolor=black, linewidth=0.026, arrowsize=0.05291667cm 2.0,arrowlength=1.4,arrowinset=0.0]{->}(0.0040216064,-0.3)(0.40402162,-0.1)
\psline[linecolor=black, linewidth=0.026, arrowsize=0.05291667cm 2.0,arrowlength=1.4,arrowinset=0.0]{->}(0.0040216064,0.3)(0.40402162,0.1)
\rput[bl](0.6040216,-0.1){$T$}
\psline[linecolor=black, linewidth=0.026, arrowsize=0.05291667cm 2.0,arrowlength=1.4,arrowinset=0.0]{->}(1.1040215,0.0)(1.5040216,0.0)
\end{pspicture}
\end{figure}
\noindent where the top leftmost arrow corresponds to $V_1$ and the bottom one to $V_2$. In this paper we will only consider tensors where all the arrows correspond to the same vector space $V$, so we do not include $V$ in the notation. The tensor product $T_1\ot T_2$ is obtained by stacking $T_1$ on top of $T_2$ and compositions are drawn by joining the outputs of a tensor with the corresponding inputs in another tensor. 
When an input/output is the dual of a vector space, we will denote it by reversing the orientation of that arrow. For instance, if $V$ is a finite dimensional vector space, a tensor $T:V\ot V^*\to V^*$, the pairing $V^*\ot V\to \kk$, the copairing $\kk\to V\ot V^*$ and the trace of a map $f:V\to V$ are denoted by

\begin{figure}[H]
\begin{pspicture}(0,-0.44)(8.458938,0.44)
\psbezier[linecolor=black, linewidth=0.026, arrowsize=0.05291667cm 2.0,arrowlength=1.4,arrowinset=0.0]{<-}(3.0089374,0.38)(3.6089373,0.38)(3.6089373,-0.42)(3.0089374,-0.42)
\psbezier[linecolor=black, linewidth=0.026, arrowsize=0.05291667cm 2.0,arrowlength=1.4,arrowinset=0.0]{<-}(5.4089375,0.38)(4.8089375,0.38)(4.8089375,-0.42)(5.4089375,-0.42)
\rput[bl](3.7089374,-0.42){,}
\rput[bl](0.6589374,-0.12){$T$}
\psline[linecolor=black, linewidth=0.026, arrowsize=0.05291667cm 2.0,arrowlength=1.4,arrowinset=0.0]{->}(0.0089373775,0.38)(0.40893736,0.18)
\psline[linecolor=black, linewidth=0.026, arrowsize=0.05291667cm 2.0,arrowlength=1.4,arrowinset=0.0]{<-}(0.0089373775,-0.42)(0.40893736,-0.22)
\psline[linecolor=black, linewidth=0.026, arrowsize=0.05291667cm 2.0,arrowlength=1.4,arrowinset=0.0]{<-}(1.2089374,-0.02)(1.6089374,-0.02)
\rput[bl](1.9089373,-0.42){,}
\rput[bl](8.408937,-0.42){.}
\rput[bl](7.458937,0.13){$f$}
\psbezier[linecolor=black, linewidth=0.026](7.8089375,0.28)(8.408937,0.28)(8.408937,-0.32)(7.6089373,-0.32)
\psbezier[linecolor=black, linewidth=0.026, arrowsize=0.05291667cm 2.0,arrowlength=1.4,arrowinset=0.0]{<-}(7.3089375,0.28)(6.708937,0.28)(6.708937,-0.32)(7.5089374,-0.32)
\psline[linecolor=black, linewidth=0.026](7.5089374,-0.32)(7.6089373,-0.32)
\rput[bl](5.708937,-0.42){,}
\end{pspicture}
\end{figure}

The above notation is also valid for super-vector spaces. Here we add the convention that a crossing pair of arrows stands for the symmetry $\tau(x\ot y)=(-1)^{|x||y|}y\ot x$ of super-vector spaces:
\begin{figure}[H]
\centering
\begin{pspicture}(0,-0.42)(2.3569586,0.42)
\rput[bl](0.0,0.0){$\tau$}
\rput[bl](0.5,0.05){=}
\psbezier[linecolor=black, linewidth=0.026, arrowsize=0.05291667cm 2.0,arrowlength=1.4,arrowinset=0.0]{->}(1.25,0.4)(1.85,0.4)(1.85,-0.4)(2.45,-0.4)
\psbezier[linecolor=black, linewidth=0.026, arrowsize=0.05291667cm 2.0,arrowlength=1.4,arrowinset=0.0]{->}(1.25,-0.4)(1.85,-0.4)(1.85,0.4)(2.45,0.4)
\end{pspicture}
\end{figure}

\noindent If $V$ is a super-vector space, the left pairing/copairing $V^*\ot V\to \kk$ and $\kk\to V\ot V^*$ are the usual ones of vector spaces. We will suppose that the category of super-vector spaces has the canonical pivotal structure, that is, the pairing $V\ot V^*\to \kk$ and the copairing $\kk\to V^*\ot V$ are defined by

\begin{figure}[H]
\centering

\begin{pspicture}(0,-0.42)(8.2,0.42)
\psbezier[linecolor=black, linewidth=0.026](1.8,0.4)(2.4,0.4)(2.4,-0.4)(3.0,-0.4)
\psbezier[linecolor=black, linewidth=0.026](1.8,-0.4)(2.4,-0.4)(2.4,0.4)(3.0,0.4)
\psbezier[linecolor=black, linewidth=0.026, arrowsize=0.05291667cm 2.0,arrowlength=1.4,arrowinset=0.0]{<-}(2.97,0.4)(3.57,0.4)(3.6,-0.4)(3.0,-0.4)
\psbezier[linecolor=black, linewidth=0.026, arrowsize=0.05291667cm 2.0,arrowlength=1.4,arrowinset=0.0]{->}(0.0,0.4)(0.6,0.4)(0.6,-0.4)(0.0,-0.4)
\rput[bl](1.1,0.0){=}
\psbezier[linecolor=black, linewidth=0.026](8.2,0.4)(7.6,0.4)(7.6,-0.4)(7.0,-0.4)
\psbezier[linecolor=black, linewidth=0.026](8.2,-0.4)(7.6,-0.4)(7.6,0.4)(7.0,0.4)
\psbezier[linecolor=black, linewidth=0.026, arrowsize=0.05291667cm 2.0,arrowlength=1.4,arrowinset=0.0]{<-}(7.03,0.4)(6.43,0.4)(6.4,-0.4)(7.0,-0.4)
\psbezier[linecolor=black, linewidth=0.026, arrowsize=0.05291667cm 2.0,arrowlength=1.4,arrowinset=0.0]{->}(5.4,0.4)(4.8,0.4)(4.8,-0.4)(5.4,-0.4)
\rput[bl](5.9,0.0){=}
\rput[bl](3.9,-0.4){,}
\end{pspicture}

\end{figure}

\noindent Note that if $V=V_0\oplus V_1$ is a finite-dimensional super-vector space, the above trace of $f:V\to V$ is the supertrace $\tr(f|_{V_0})-\tr(f|_{V_1})$.

%Of course, we can trace any output with another input provided they correspond to the same vector space and the above notation extend to this case. 

\subsection{Hopf algebras} A Hopf algebra over $\kk$ is a vector space $H$ endowed with tensors
\begin{figure}[H]
\centering
\begin{pspicture}(0,-0.31162986)(10.105809,0.31162986)
\psline[linecolor=black, linewidth=0.026, arrowsize=0.05291667cm 2.0,arrowlength=1.4,arrowinset=0.0]{->}(4.3058095,0.0)(4.705809,0.0)
\rput[bl](4.8658094,-0.09999993){$\Delta$}
\psline[linecolor=black, linewidth=0.026, arrowsize=0.05291667cm 2.0,arrowlength=1.4,arrowinset=0.0]{->}(5.3058095,0.100000076)(5.705809,0.30000007)
\psline[linecolor=black, linewidth=0.026, arrowsize=0.05291667cm 2.0,arrowlength=1.4,arrowinset=0.0]{->}(5.3058095,-0.09999993)(5.705809,-0.29999992)
\psline[linecolor=black, linewidth=0.026, arrowsize=0.05291667cm 2.0,arrowlength=1.4,arrowinset=0.0]{->}(0.005809326,-0.29999992)(0.4058093,-0.09999993)
\psline[linecolor=black, linewidth=0.026, arrowsize=0.05291667cm 2.0,arrowlength=1.4,arrowinset=0.0]{->}(0.005809326,0.30000007)(0.4058093,0.100000076)
\rput[bl](0.5658093,-0.049999923){$m$}
\psline[linecolor=black, linewidth=0.026, arrowsize=0.05291667cm 2.0,arrowlength=1.4,arrowinset=0.0]{->}(1.0058093,0.0)(1.4058093,0.0)
\rput[bl](2.4458094,-0.09999993){$1$}
\psline[linecolor=black, linewidth=0.026, arrowsize=0.05291667cm 2.0,arrowlength=1.4,arrowinset=0.0]{->}(2.8058093,0.0)(3.2058094,0.0)
\psline[linecolor=black, linewidth=0.026, arrowsize=0.05291667cm 2.0,arrowlength=1.4,arrowinset=0.0]{->}(6.8058095,0.0)(7.205809,0.0)
\rput[bl](7.4058094,-0.059999924){$\epsilon$}
\rput[bl](9.145809,-0.09999993){$S$}
\psline[linecolor=black, linewidth=0.026, arrowsize=0.05291667cm 2.0,arrowlength=1.4,arrowinset=0.0]{->}(9.50581,0.0)(9.905809,0.0)
\psline[linecolor=black, linewidth=0.026, arrowsize=0.05291667cm 2.0,arrowlength=1.4,arrowinset=0.0]{->}(8.605809,0.0)(9.00581,0.0)
\rput[bl](1.5258093,-0.09999993){,}
\rput[bl](3.3258092,-0.09999993){,}
\rput[bl](5.725809,-0.09999993){,}
\rput[bl](7.625809,-0.09999993){,}
\rput[bl](10.055809,-0.09999993){.}
\end{pspicture}
\end{figure}

\noindent such that $(H,m,1)$ is an algebra, $(H,\De,\e)$ is a coalgebra, $\De$ and $\e$ are algebra homomorphism and $S$ satisfies the antipode axiom. Sometimes, we also use Sweedler's notation for the coproduct, that is, we write $\De(h)= h_{(1)}\ot h_{(2)}$ for $h\in H$ where it is understood that this is a sum of various elements $h_{(1)}\ot h_{(2)}\in H^{\ot 2}$. We will also consider Hopf algebras in $\SVect$, the only difference is that the multiplicative axiom for the coproduct involves a sign coming from the symmetry of the category. We denote by $\Aut(H)$ the group of Hopf algebra automorphisms of $H$. In this paper, we assume Hopf algebras are finite dimensional. %For more details see \cite{Radford:BOOK}.
\medskip

The dual $H^*$ is a Hopf algebra if the product and the coproduct (which we simply denote by $\De$ and $m$ respectively in tensor network notation) are defined by
\begin{figure}[H]
\begin{pspicture}(0,-0.570345)(9.9,0.570345)
\psline[linecolor=black, linewidth=0.026, arrowsize=0.05291667cm 2.0,arrowlength=1.4,arrowinset=0.0]{->}(8.1,-0.449655)(8.5,-0.249655)
\psline[linecolor=black, linewidth=0.026, arrowsize=0.05291667cm 2.0,arrowlength=1.4,arrowinset=0.0]{->}(8.1,0.150345)(8.5,-0.049654998)
\rput[bl](8.66,-0.199655){$m$}
\psbezier[linecolor=black, linewidth=0.026](1.2,0.150345)(1.6,0.150345)(1.6,-0.449655)(1.2,-0.4496549987792969)
\psline[linecolor=black, linewidth=0.026, arrowsize=0.05291667cm 2.0,arrowlength=1.4,arrowinset=0.0]{->}(2.5,-0.349655)(2.9,-0.349655)
\rput[bl](3.06,-0.449655){$\Delta$}
\psline[linecolor=black, linewidth=0.026, arrowsize=0.05291667cm 2.0,arrowlength=1.4,arrowinset=0.0]{<-}(0.0,-0.149655)(0.4,0.050345)
\psline[linecolor=black, linewidth=0.026, arrowsize=0.05291667cm 2.0,arrowlength=1.4,arrowinset=0.0]{<-}(0.0,0.450345)(0.4,0.250345)
\psline[linecolor=black, linewidth=0.026, arrowsize=0.05291667cm 2.0,arrowlength=1.4,arrowinset=0.0]{<-}(1.0,0.150345)(1.2,0.150345)
\psline[linecolor=black, linewidth=0.026](1.2,-0.449655)(0.0,-0.449655)
\psline[linecolor=black, linewidth=0.026, arrowsize=0.05291667cm 2.0,arrowlength=1.4,arrowinset=0.0]{->}(3.2,0.150345)(2.5,0.150345)
\psline[linecolor=black, linewidth=0.026, arrowsize=0.05291667cm 2.0,arrowlength=1.4,arrowinset=0.0]{->}(3.2,0.550345)(2.5,0.550345)
\rput[bl](0.56,0.050345){$\Delta$}
\psbezier[linecolor=black, linewidth=0.026](3.5,-0.249655)(3.9,-0.149655)(4.04,-0.069655)(4.02,0.21034500122070313)(4.0,0.490345)(3.6,0.550345)(3.2,0.550345)
\psbezier[linecolor=black, linewidth=0.026](3.5,-0.449655)(3.8,-0.649655)(4.17,-0.439655)(4.13,-0.13965499877929688)(4.09,0.160345)(3.7,0.150345)(3.2,0.150345)
\psbezier[linecolor=black, linewidth=0.026](9.1,0.450345)(9.5,0.450345)(9.5,-0.149655)(9.1,-0.1496549987792969)
\psline[linecolor=black, linewidth=0.026, arrowsize=0.05291667cm 2.0,arrowlength=1.4,arrowinset=0.0]{->}(9.1,0.450345)(8.1,0.450345)
\psline[linecolor=black, linewidth=0.026, arrowsize=0.05291667cm 2.0,arrowlength=1.4,arrowinset=0.0]{<-}(5.5,0.250345)(5.9,0.250345)
\rput[bl](6.06,0.150345){$m$}
\psline[linecolor=black, linewidth=0.026](5.6,-0.149655)(5.9,-0.149655)
\psline[linecolor=black, linewidth=0.026](5.6,-0.449655)(5.9,-0.449655)
\psbezier[linecolor=black, linewidth=0.026, arrowsize=0.05291667cm 2.0,arrowlength=1.4,arrowinset=0.0]{->}(5.9,-0.149655)(6.2,-0.149655)(7.03,-0.209655)(7.05,0.15034500122070313)(7.07,0.510345)(6.8,0.550345)(6.5,0.450345)
\psbezier[linecolor=black, linewidth=0.026, arrowsize=0.05291667cm 2.0,arrowlength=1.4,arrowinset=0.0]{->}(5.9,-0.449655)(6.7,-0.449655)(7.02,-0.419655)(7.02,-0.18965499877929687)(7.02,0.040345002)(6.8,0.050345)(6.5,0.150345)
\rput[bl](1.9,-0.149655){=}
\rput[bl](7.4,-0.149655){=}
\rput[bl](4.45,-0.499655){,}
\rput[bl](9.85,-0.39965498){.}
\end{pspicture}
\end{figure}

%    Integral cointegral dist group-likes Radford's formula and properties.

\medskip

\subsection{Integrals and cointegrals}\label{subs: Hopf integrals} Let $H$ be a finite dimensional Hopf algebra. Then there exists a unique element (up to scalar) $\la_r\in H^*$  satisfying
\begin{figure}[H]
\centering
\begin{pspicture}(0,-0.43439284)(4.926559,0.43439284)
\psline[linecolor=black, linewidth=0.026, arrowsize=0.05291667cm 2.0,arrowlength=1.4,arrowinset=0.0]{->}(0.0,-0.15560715)(0.4,-0.15560715)
\rput[bl](0.56,-0.25560716){$\Delta$}
\psline[linecolor=black, linewidth=0.026, arrowsize=0.05291667cm 2.0,arrowlength=1.4,arrowinset=0.0]{->}(1.0,-0.055607148)(1.4,0.14439285)
\psline[linecolor=black, linewidth=0.026, arrowsize=0.05291667cm 2.0,arrowlength=1.4,arrowinset=0.0]{->}(1.0,-0.25560716)(1.4,-0.45560715)
\rput[bl](4.24,-0.25560716){$1$}
\psline[linecolor=black, linewidth=0.026, arrowsize=0.05291667cm 2.0,arrowlength=1.4,arrowinset=0.0]{->}(4.6,-0.15560715)(5.0,-0.15560715)
\rput[bl](1.6,0.14439285){$\lambda_r$}
\rput[bl](3.5,-0.25560716){$\lambda_r$}
\psline[linecolor=black, linewidth=0.026, arrowsize=0.05291667cm 2.0,arrowlength=1.4,arrowinset=0.0]{->}(2.9,-0.15560715)(3.3,-0.15560715)
\rput[bl](2.25,-0.15560715){=}
\end{pspicture}
\end{figure}
\noindent This is called a {\em right integral} of $H$. Similarly, there exists a unique (up to scalar) {\em left cointegral} $\La_l\in H$ satisfying
\begin{figure}[H]
\centering
\begin{pspicture}(0,-0.4558149)(5.2,0.4558149)
\psline[linecolor=black, linewidth=0.026, arrowsize=0.05291667cm 2.0,arrowlength=1.4,arrowinset=0.0]{->}(0.45,-0.1558149)(0.85,0.044185106)
\psline[linecolor=black, linewidth=0.026, arrowsize=0.05291667cm 2.0,arrowlength=1.4,arrowinset=0.0]{->}(0.45,0.4441851)(0.85,0.2441851)
\psline[linecolor=black, linewidth=0.026, arrowsize=0.05291667cm 2.0,arrowlength=1.4,arrowinset=0.0]{->}(1.55,0.14418511)(1.95,0.14418511)
\rput[bl](2.3,0.14418511){=}
\psline[linecolor=black, linewidth=0.026, arrowsize=0.05291667cm 2.0,arrowlength=1.4,arrowinset=0.0]{->}(2.95,0.14418511)(3.35,0.14418511)
\psline[linecolor=black, linewidth=0.026, arrowsize=0.05291667cm 2.0,arrowlength=1.4,arrowinset=0.0]{->}(4.55,0.14418511)(4.95,0.14418511)
\rput[bl](3.5,0.044185106){$\epsilon$}
\rput[bl](5.15,0.14418511){.}
\rput[bl](0.0,-0.4558149){$\Lambda_l$}
\rput[bl](1.05,0.044185106){$m$}
\rput[bl](4.0,-0.055814896){$\Lambda_l$}
\end{pspicture}
\end{figure}

From the uniqueness of the integrals and cointegrals, it follows that there exist group-likes $\gg\in G(H)$ and $\zz\in G(H^*)$ characterized by the following equations:
\begin{figure}[H]
\centering
\begin{pspicture}(0,-0.645)(11.45,0.645)
\rput[bl](8.55,-0.045){=}
\psline[linecolor=black, linewidth=0.026, arrowsize=0.05291667cm 2.0,arrowlength=1.4,arrowinset=0.0]{->}(9.2,-0.045)(9.6,-0.045)
\rput[bl](9.75,-0.245){$\zz$}
\psline[linecolor=black, linewidth=0.026, arrowsize=0.05291667cm 2.0,arrowlength=1.4,arrowinset=0.0]{->}(6.7,-0.345)(7.1,-0.145)
\psline[linecolor=black, linewidth=0.026, arrowsize=0.05291667cm 2.0,arrowlength=1.4,arrowinset=0.0]{->}(6.7,0.255)(7.1,0.055)
\psline[linecolor=black, linewidth=0.026, arrowsize=0.05291667cm 2.0,arrowlength=1.4,arrowinset=0.0]{->}(7.8,-0.045)(8.2,-0.045)
\rput[bl](6.15,0.355){$\Lambda_l$}
\rput[bl](7.3,-0.145){$m$}
\psline[linecolor=black, linewidth=0.026, arrowsize=0.05291667cm 2.0,arrowlength=1.4,arrowinset=0.0]{->}(10.8,-0.045)(11.2,-0.045)
\rput[bl](11.4,-0.045){.}
\rput[bl](10.25,-0.245){$\Lambda_l$}
\psline[linecolor=black, linewidth=0.026, arrowsize=0.05291667cm 2.0,arrowlength=1.4,arrowinset=0.0]{->}(0.0,-0.045)(0.4,-0.045)
\rput[bl](0.56,-0.145){$\Delta$}
\psline[linecolor=black, linewidth=0.026, arrowsize=0.05291667cm 2.0,arrowlength=1.4,arrowinset=0.0]{->}(1.0,0.055)(1.4,0.255)
\psline[linecolor=black, linewidth=0.026, arrowsize=0.05291667cm 2.0,arrowlength=1.4,arrowinset=0.0]{->}(1.0,-0.145)(1.4,-0.345)
\rput[bl](4.24,-0.145){$\gg$}
\psline[linecolor=black, linewidth=0.026, arrowsize=0.05291667cm 2.0,arrowlength=1.4,arrowinset=0.0]{->}(4.6,-0.045)(5.0,-0.045)
\rput[bl](1.6,-0.645){$\lambda_r$}
\rput[bl](3.5,-0.145){$\lambda_r$}
\psline[linecolor=black, linewidth=0.026, arrowsize=0.05291667cm 2.0,arrowlength=1.4,arrowinset=0.0]{->}(2.9,-0.045)(3.3,-0.045)
\rput[bl](2.25,-0.045){=}
\rput[bl](5.2,-0.145){,}
\end{pspicture}
\end{figure}

\def\Vect{\text{Vect}}
\def\SVect{\text{SVect}}

If $H$ is a Hopf algebra in $\Vect$, the trace of a linear map $f:H\to H$ can be computed from a cointegral/integral pair $\La_l,\la_r$ as above satisfying $\la_r(\La_l)=1$ by the following formula of Radford \cite[Thm. 10.4.1]{Radford:BOOK}:

\begin{figure}[H]
\centering
\begin{pspicture}(0,-0.455)(8.596064,0.455)
\psline[linecolor=black, linewidth=0.026, arrowsize=0.05291667cm 2.0,arrowlength=1.4,arrowinset=0.0]{->}(7.4760637,-0.005)(7.8760633,-0.005)
\rput[bl](3.1260636,-0.105){$\Lambda_l$}
\rput[bl](6.9760637,-0.105){$m$}
\psline[linecolor=black, linewidth=0.026, arrowsize=0.05291667cm 2.0,arrowlength=1.4,arrowinset=0.0]{->}(3.7760634,-0.005)(4.1760635,-0.005)
\rput[bl](4.3360634,-0.105){$\Delta$}
\psline[linecolor=black, linewidth=0.026, arrowsize=0.05291667cm 2.0,arrowlength=1.4,arrowinset=0.0]{->}(4.7760634,0.095)(5.1760635,0.295)
\psline[linecolor=black, linewidth=0.026, arrowsize=0.05291667cm 2.0,arrowlength=1.4,arrowinset=0.0]{->}(4.7760634,-0.105)(5.1760635,-0.305)
\rput[bl](8.276064,-0.105){$\lambda_r$}
\rput[bl](5.4260635,0.145){$f$}
\rput[bl](5.4260635,-0.455){$S$}
\psbezier[linecolor=black, linewidth=0.026, arrowsize=0.05291667cm 2.0,arrowlength=1.4,arrowinset=0.0]{->}(5.8760633,-0.305)(6.2760634,-0.305)(6.1760635,0.595)(6.7760634,0.195)
\psbezier[linecolor=black, linewidth=0.026, arrowsize=0.05291667cm 2.0,arrowlength=1.4,arrowinset=0.0]{->}(5.8760633,0.295)(6.2760634,0.295)(6.1760635,-0.605)(6.7760634,-0.205)
\rput[bl](2.2260635,-0.055){=}
\rput[bl](0.6260635,0.145){$f$}
\psbezier[linecolor=black, linewidth=0.026, arrowsize=0.05291667cm 2.0,arrowlength=1.4,arrowinset=0.0](0.97606355,0.295)(1.5760635,0.295)(1.5760635,-0.305)(0.77606356,-0.305)
\psbezier[linecolor=black, linewidth=0.026, arrowsize=0.05291667cm 2.0,arrowlength=1.4,arrowinset=0.0]{<-}(0.47606355,0.295)(-0.12393646,0.295)(-0.12393646,-0.305)(0.67606354,-0.305)
\psline[linecolor=black, linewidth=0.026, arrowsize=0.05291667cm 2.0,arrowlength=1.4,arrowinset=0.0](0.67606354,-0.305)(0.77606356,-0.305)
\end{pspicture}
\end{figure}

\noindent All of the above holds as well if $H$ is a Hopf algebra in $\SVect$ \cite{Kup2}, except that the right hand side in Radford's trace formula has to be multiplied by $(-1)^{|\La_l|}$ (this follows e.g. from \cite[Lemma 3.4]{Kup2}).
\medskip

Now let $\a\in\Aut(H)$. Then, if $\la_r$ is a right integral of $H$, $\la_r\circ \a$ is also a right integral so by uniqueness, 
\begin{align}
\label{eq: rH}
\la_r\circ \a=\rH(\a)\cdot\la_r
\end{align}
for some scalar $\rH(\a)\in\kk^{\t}$. This defines a group homomorphism $\rH:\Aut(H)\to\kk^{\t}$. Note that if $\la_l$ is a left integral, then also $\la_l\circ\a=\rH(\a)\la_l$ (since any left integral is a multiple of $\la_r\circ S$). Similarly, if $\La$ is any left or right cointegral we have $\a(\La)=\rH(\a)\La$, since $\la(\La)\neq 0$ for any nonzero integral (left or right) and any nonzero cointegral (left or right) \cite[Thm. 10.2.2]{Radford:BOOK}.

%If $H$ is unimodular, the ribbon element of $D(H)$ corresponding to the pivot $g$ as above is $(g^{-1}\ot \e)u$, where $u$ is the Drinfeld element of $D(H)$.\medskip

%The following is taken from \cite{CC:ontwoinvariants}.\begin{proposition}Suppose $D(H)$ is ribbon and $\b\ot g$ is the pivotal element corresponding to the ribbon element $v$. Let $\La_l\in H$ be a left cointegral and $\la_r\in H^*$ and $\la_r(\La)_l$ be a right integral so $\la_R^{D(H)}=\La_l\ot \la_r$ is a right integral of $D(H)$. If $\la_r(\La_l)=1$ then \begin{align*}\la_r^{D(H)}(v)&=\b(b)^5, & \la_r^{D(H)}(v^{-1})&=\b(b)^{-1}.\end{align*}\end{proposition}

\subsection{Hopf $G$-coalgebras}\label{subs: Hopf G-coalgebras} We now recall some definitions from \cite{Virelizier:Hopfgroup}. Let $G$ be a group. A {\em Hopf $G$-coalgebra} over $\kk$ is a family $\uH=\{\Ha\}_{\a\in G}$ where each $\Ha=(\Ha,\ma,\oa)$ is a $\kk$-algebra with multiplication $\ma$ and unit $\oa$, endowed with algebra morphisms $\Deab:\Hab\to\Ha\ot\Hb$ for each $\a,\b\in G$, a counit $\e:A_1\to\kk$ and an antipode $\Sa:\Ha\to\Ham$ for each $\a\in G$ satisfying graded versions of the usual Hopf algebra axioms (see \cite{Virelizier:Hopfgroup} for more details). %For example, the antipode axiom becomes:\begin{align*}\ma(\id_{\Ha}\ot\Sam)\De_{\a,\a^{-1}}=\oa\e=\ma(\Sam\ot\id_{\Ha})\De_{\am,\a}.\end{align*}
%We do not include subscripts on arrows indicating which $\Ha$ it corresponds to, but we include the subscript on the corresponding tensors (so the $\a\in G$ of each arrow can be deduced from this). In tensor network notation, for instance, the multiplicative axiom for $\De$ is
For instance, the antipode axiom for Hopf $G$-coalgebras is

\begin{figure}[H]
\centering
\begin{pspicture}(0,-0.51427)(12.9,0.51427)
\rput[bl](0.57,-0.2){$\Delta_{\alpha^{-1},\alpha}$}
\psline[linecolor=black, linewidth=0.026, arrowsize=0.05291667cm 2.0,arrowlength=1.4,arrowinset=0.0]{->}(1.9,0.1)(2.2,0.3)
\rput[bl](2.4,0.2){$S_{\alpha^{-1}}$}
\psbezier[linecolor=black, linewidth=0.026, arrowsize=0.05291667cm 2.0,arrowlength=1.4,arrowinset=0.0]{->}(1.9,-0.1)(2.3,-0.5)(3.1,-0.5)(3.5,-0.1)
\psline[linecolor=black, linewidth=0.026, arrowsize=0.05291667cm 2.0,arrowlength=1.4,arrowinset=0.0]{->}(3.2,0.3)(3.5,0.1)
\rput[bl](3.64,-0.07){$m_{\alpha}$}
\psline[linecolor=black, linewidth=0.026, arrowsize=0.05291667cm 2.0,arrowlength=1.4,arrowinset=0.0]{->}(0.0,0.0)(0.4,0.0)
\psline[linecolor=black, linewidth=0.026, arrowsize=0.05291667cm 2.0,arrowlength=1.4,arrowinset=0.0]{->}(4.3,0.0)(4.7,0.0)
\rput[bl](8.67,-0.2){$\Delta_{\alpha,\alpha^{-1}}$}
\psline[linecolor=black, linewidth=0.026, arrowsize=0.05291667cm 2.0,arrowlength=1.4,arrowinset=0.0]{->}(11.3,-0.3)(11.6,-0.1)
\rput[bl](10.5,-0.5){$S_{\alpha^{-1}}$}
\psbezier[linecolor=black, linewidth=0.026, arrowsize=0.05291667cm 2.0,arrowlength=1.4,arrowinset=0.0]{->}(10.0,0.1)(10.4,0.5)(11.2,0.5)(11.6,0.1)
\psline[linecolor=black, linewidth=0.026, arrowsize=0.05291667cm 2.0,arrowlength=1.4,arrowinset=0.0]{->}(10.0,-0.1)(10.3,-0.3)
\rput[bl](11.74,-0.07){$m_{\alpha}$}
\psline[linecolor=black, linewidth=0.026, arrowsize=0.05291667cm 2.0,arrowlength=1.4,arrowinset=0.0]{->}(8.1,0.0)(8.5,0.0)
\psline[linecolor=black, linewidth=0.026, arrowsize=0.05291667cm 2.0,arrowlength=1.4,arrowinset=0.0]{->}(12.3,0.0)(12.7,0.0)
\rput[bl](4.98,-0.06){=}
\rput[bl](7.68,-0.06){=}
\rput[bl](6.54,-0.1){$1_{\alpha}$}
\psline[linecolor=black, linewidth=0.026, arrowsize=0.05291667cm 2.0,arrowlength=1.4,arrowinset=0.0]{->}(7.0,0.0)(7.4,0.0)
\psline[linecolor=black, linewidth=0.026, arrowsize=0.05291667cm 2.0,arrowlength=1.4,arrowinset=0.0]{->}(5.5,0.0)(5.9,0.0)
\rput[bl](6.1,-0.06){$\epsilon$}
\rput[bl](12.85,-0.1){.}
\end{pspicture}
\end{figure}

\noindent This definition is also valid in the category of super-vector spaces. When $G=1$ we recover the usual notion of Hopf algebra and we denote the structure tensors simply by $m,1,\De,\e,S$. If $\uHH$ is a Hopf $G$-coalgebra, then $A_1$ is a Hopf algebra in the usual sense. 
The notion of integral and distinguished group-like extend to this setting \cite{Virelizier:Hopfgroup}. %We say that $\uH$ is a {\em $G$-extension} of $A_1$.

\medskip

 %Here, the right integral is a family of functionals $\la_{\a}:\Ha\to\kk$ satisfying $(\la_{\a}\ot\id)\De_{\a\b}=\la_{\a\b}1_{\b}$ and the distinguished group-like $\gg_{\a}\in\Ha$ is a group-like in the sense that $\De_{\a\b}$

\medskip

\def\vaa{\v_{\a}}\def\whva{\wh{\v}_{\a}}\def\whv{\wh{\v}}

\def\Sets{\bold{Sets}}\def\Rab{R_{\a,\b}}

A Hopf $G$-coalgebra $\uHH$ is said to be {\em crossed} if it is endowed with a family of algebra isomorphisms $\v=\{\v_{\b,\a}:\Ha\to A_{\b\a\b^{-1}}\}_{\a,\b\in G}$, called a {\em crossing}, satisfying:
\begin{enumerate}
\item $(\v_{\b,\a}\ot\v_{\b,\c})\De_{\a,\c}=\De_{\b\a\b^{-1},\b\c\b^{-1}}\v_{\b,\a\c}$,
\item $\e\v_{\b,1}=\e$,
\item $\v_{\a,\b\c\b^{-1}}\v_{\b,\c}=\v_{\a\b,\c}$,

\end{enumerate}
for each $\a,\b,\c\in G$. We will omit the second subscript of $\v_{\b,\a}$ and denote $\v_{\b}=\v_{\b,\a}$. If $(\uH,\v)$ is crossed, we will say that a family of elements $x=\{x_{\a}\}_{\a\in G}$, where each $x_{\a}\in A_{\a}$, is {\em $G$-invariant} if $\v_{\b}(x_{\a})=x_{\b\a\b^{-1}}$ for each $\a,\b\in G$.

\medskip

A crossed Hopf $G$-coalgebra $(\uH,\v)$ is {\em quasi-triangular} if it is endowed with a family of elements $R=\{\Rab\}_{\a,\b\in G}$ satisfying the following axioms for each $\a,\b,\c\in G$:
\def\Hc{A_{\c}}
\begin{enumerate}

\item $R_{\a,\b}$ is invertible in $A_{\a}\ot A_{\b}$,
\item $(\v_{\c}\ot \v_{\c})(R_{\a,\b})=R_{\c\a\c^{-1},\c\b\c^{-1}}$,

\item $R_{\a,\b}\cdot\Deab(x)=(\tau_{\b,\a}(\v_{\a^{-1}}\ot \id_{\Ha})\De_{\a\b\a^{-1},\a}(x))\cdot R_{\a,\b}$ for each $x\in\Hab$, where $\tau_{\a,\b}$ denotes permutation of two factors (with signs in the super-case),

\item $(\id_{\Ha}\ot\De_{\b,\c})(R_{\a,\b\c})=(R_{\a,\c})_{1\b 3}\cdot (R_{\a,\b})_{12\c}$,

\item $(\Deab\ot\id_{\Hc})(R_{\a\b,\c})=((\id_{\Ha}\ot\v_{\b^{-1}})(R_{\a,\b\c\b^{-1}}))_{1\b 3}\cdot (R_{\b,\c})_{\a 23}$
\end{enumerate}
The last two equalities hold in $\Ha\ot\Hb\ot\Hc$, where we note $X_{1 2\c}=x\ot y\ot 1_{\c}$ for any $X=\sum x\ot y\in \Ha\ot\Hb$ and similarly for $Y_{\a 23}, Z_{1\b 3}$ ($Y\in\Hb\ot\Hc, Z\in \Ha\ot\Hc$). In the super-case we will suppose that $R$-matrices have degree zero.
\medskip

\def\ga{g_{\a}}\def\ua{u_{\a}}
\noindent The {\em Drinfeld element} of a quasi-triangular Hopf $G$-coalgebra $(\uH,\v,R)$ is $u=\{u_{\a}\}_{\a\in G}$ defined by
\begin{align}
\label{eq: graded Drinfeld element}
\ua=\ma(\Sam\vaa\ot\id_{\Ha})\tau_{\a,\am}(R_{\a,\am})\in A_{\a}.
\end{align}
This is invertible with inverse $u_{\a}^{-1}=m_{\a}(\id_{A_{\a}}\ot S_{\a^{-1}}S_{\a})\tau_{\a,\a}(R_{\a,\a})$ \cite[Lemma 6.5]{Virelizier:Hopfgroup}.
\medskip

A quasi-triangular Hopf $G$-coalgebra $(\uH,\v,R)$ is {\em ribbon} if it is endowed with a family $v=\{v_{\a}\}_{\a\in G}$, where each $\va\in\Ha^{\t}$, such that 
\begin{enumerate}
\item $\v_{\b}(v_{\a})=v_{\b\a\b^{-1}}$,
\item $\Deab(v_{\a\b})=(v_{\a}\ot v_{\b})(\tau_{\b,\a}(\v_{\a^{-1}}\ot \id_{\Ha})R_{\a\b\a^{-1},\a})\cdot R_{\a,\b}$,
\item $\vaa(x)=v_{\a}^{-1}xv_{\a}$ for all $x\in\Ha$,
\item $\Sa(v_{\a})=v_{\am}$.
\end{enumerate}
%Note that this is the graded version of Turaev's definition of ribbon element (which is the inverse of that of \cite{Radford:BOOK}). 

%Q: Is there a condition for $\va^2$? This is indeed a consequence, see \cite[Lemma 6.8, d)]{Virelizier:Hopfgroup}. The argument in the ungraded case is $v^{-2}=ug^{-1}v^{-1}=u S(vu)v^{-1}=uS(u)S(v)v^{-1}=uS(u)$. In the graded case we find 
\medskip

The following lemma is a generalization of \cite[Theorem 12.3.6]{Radford:BOOK}. We say that $g=\{g_{\a}\}_{\a\in G}$ is {\em group-like} if $\De_{\a,\b}(g_{\a\b})=g_{\a}\ot g_{\b}$ for each $\a,\b\in G$ and each $g_{\a}$ is invertible in $A_{\a}$.

\def\ua{u_{\a}}\def\uam{u_{\a^-1{}}}\def\Rab{R_{\a,\b}}\def\RRab{R=\{R_{\a,\b}\}_{\a,\b\in G}}\def\vva{v=\{\va\}_{\a\in G}}

\begin{lemma}
\label{lemma: ribbon bijection group-likes}
Let $(\uH,\v,R)$ be a quasi-triangular Hopf $G$-coalgebra. The map $\{v_{\a}\}\mapsto \{g_{\a}:=u_{\a}v_{\a}\}$ is a bijection between ribbon structures and elements $g=\{g_{\a}\}$ of $\uH$ satisfying
\begin{enumerate}[label={(\arabic*$'$)}]
\item $\v_{\b}(g_{\a})=g_{\b\a\b^{-1}}$,
\item $g$ is group-like,
\item $S_{\a^{-1}}S_{\a}(x)=g_{\a}x g_{\a}^{-1}$,
\item $g_{\a}^2=u_{\a}S_{\a^{-1}}(u_{\a^{-1}}^{-1})$,
\end{enumerate}
for each $\a,\b\in G$ and $x\in A_{\a}$.
\end{lemma}
\begin{proof}
We will show that each of the above properties is equivalent to the corresponding property in the definition of ribbon element. The equivalence of $(1)$ and $(1')$ follows from $\v_{\b}(u_{\a})=u_{\b\a\b^{-1}}$ \cite[Lemma 6.5, d)]{Virelizier:Hopfgroup}. Similarly, the equivalence of $(2)$ and $(2')$ follows immediately from \cite[Lemma 6.5, f)]{Virelizier:Hopfgroup}. For $(3)$ we have
\begin{align*}
\v_{\a}(x)=v_{\a}^{-1}xv_{\a} \Leftrightarrow \v_{\a}(x)=g_{\a}^{-1}u_{\a}xu_{\a}^{-1}g_{\a} \Leftrightarrow \v_{\a}(x)=g_{\a}^{-1}S_{\a^{-1}}S_{\a}(\v_{\a}(x))g_{\a}
\end{align*}
which is equivalent to $(3')$. In the last equality we used \cite[Lemma 6.5, b)]{Virelizier:Hopfgroup}. Finally, 
\begin{align*}
S_{\a^{-1}}(v_{\a^{-1}})=v_{\a} \Leftrightarrow S_{\a^{-1}}(g_{\a^{-1}})S_{\a^{-1}}(u_{\a^{-1}}^{-1})=u_{\a}^{-1}g_{\a} \Leftrightarrow g_{\a}^{-1}S_{\a^{-1}}(u_{\a^{-1}}^{-1})=u_{\a}^{-1}g_{\a} 
\end{align*}
Since $g$ is a $G$-invariant group-like and $S_{\a^{-1}}S_{\a}(\v_{\a}(x))=u_{\a}xu_{\a}^{-1}$, $g_{\a}$ and $u_{\a}$ commute so that the last equality above is equivalent to $(4')$.
\end{proof}

If $(\uH,\v,R,v)$ is ribbon, we call $\{g_{\a}=u_{\a}v_{\a}\}_{\a\in G}$ the {\em pivot} of $\uH$.

%If we set $\ga=\va\ua=\ua\va\in\Ha$ then $\{\ga\}_{\a\in G}$ is a $G$-invariant group-like element of $\uH$ and $\Sam\Sa(x)=\ga x\ga^{-1}$ for all $x\in\Ha$. We call $\{\ga\}_{\a\in G}$ the {\em pivotal element} or just the pivot of $\uH$. See \cite[Lemma 6.8]{Virelizier:Hopfgroup} for more details. %Moreover $$v_{\a}^{-2}=\ua\ga^{-1}\va^{-1}=\ua S_{\a^{-1}}(v_{\a^{-1}}u_{\a^{-1}})\va^{-1}=\ua S_{\a^{-1}}(u_{\a^{-1}}),$$

%\begin{remark}The above set of axioms for a ribbon Hopf $G$-coalgebra $\uHH$ can be simply restated by saying that the category $\CC\eq\coprod_{\a\in G}\CC_{\a}$, where $\CC_{\a}=\Rep(\Ha)$, is a {\em strict} $G$-crossed ribbon category. Here strict means that the action $\a\mapsto (F_{\a}:\CC\to\CC)$ (where $F_{\a}:\CC_{\b}\to\CC_{\a\b\a^{-1}}$ is defined using $\v_{\a}^{-1}:A_{\a\b\a^{-1}}\to A_{\b}$) is strict in the sense that the isomorphisms $F_{\a}(X\ot Y)\to F_{\a}(X)\ot F_{\a}(Y)$ are identities as well as the natural isomorphisms $F_{\a}F_{\b}\cong F_{\a\b}$ for each $\a,\b\in G$. %It would be interesting to consider non-strict actions as well, but we do not do it in this paper.\end{remark}

\subsection{Twisted Drinfeld doubles}\label{subs: Hopf twisted DH} Let $H=(H,m,1,\De,\e,S)$ be a finite dimensional Hopf algebra with automorphism group $\Aut(H)$. Following \cite{Virelizier:Graded-QG} we define a quasi-triangular Hopf $\Aut(H)$-coalgebra $\uDHH$ as follows: for each $\a\in\Aut(H)$ set $D(H)_{\a}\eq H^{*}\ot H$ as vector spaces. We define a product $m_{\a}$ on $D(H)_{\a}$ and a coproduct $\De_{\a,\b}:D(H)_{\a\b}\to D(H)_{\a}\ot D(H)_{\b}$ by

\begin{figure}[H]
\centering
\begin{pspicture}(0,-1.62)(12.75,1.62)
\rput[bl](7.28,-0.2){$\Delta_{\alpha,\beta}=$}
\psline[linecolor=black, linewidth=0.026, arrowsize=0.05291667cm 2.0,arrowlength=1.4,arrowinset=0.0]{<-}(8.8,0.3)(9.4,0.3)
\psline[linecolor=black, linewidth=0.026, arrowsize=0.05291667cm 2.0,arrowlength=1.4,arrowinset=0.0]{->}(8.8,-0.3)(9.4,-0.3)
\psbezier[linecolor=black, linewidth=0.026, arrowsize=0.05291667cm 2.0,arrowlength=1.4,arrowinset=0.0]{->}(12.4,1.1)(11.584,1.1)(10.6,-0.11)(10.13,0.21)
\rput[bl](9.64,0.22){$m$}
\rput[bl](9.64,-0.4){$\Delta$}
\psbezier[linecolor=black, linewidth=0.026, arrowsize=0.05291667cm 2.0,arrowlength=1.4,arrowinset=0.0]{->}(12.4,-0.3)(11.14,-0.2)(10.97,1.35)(10.12,0.52)
\psline[linecolor=black, linewidth=0.026, arrowsize=0.05291667cm 2.0,arrowlength=1.4,arrowinset=0.0]{->}(10.2,-0.18)(12.4,0.7)
\psline[linecolor=black, linewidth=0.026, arrowsize=0.05291667cm 2.0,arrowlength=1.4,arrowinset=0.0]{->}(10.2,-0.4)(11.0,-0.7)
\rput[bl](11.21,-0.84){$\alpha^{-1}$}
\psline[linecolor=black, linewidth=0.026, arrowsize=0.05291667cm 2.0,arrowlength=1.4,arrowinset=0.0]{->}(12.0,-0.7)(12.4,-0.7)
\psline[linecolor=black, linewidth=0.026, arrowsize=0.05291667cm 2.0,arrowlength=1.4,arrowinset=0.0]{->}(1.5,1.0)(1.9,1.0)
\rput[bl](2.06,0.9){$\Delta$}
\psline[linecolor=black, linewidth=0.026, arrowsize=0.05291667cm 2.0,arrowlength=1.4,arrowinset=0.0]{->}(2.5,0.8)(2.8,0.6)
\rput[bl](2.04,-1.05){$m$}
\psline[linecolor=black, linewidth=0.026, arrowsize=0.05291667cm 2.0,arrowlength=1.4,arrowinset=0.0]{->}(1.9,-1.0)(1.5,-1.0)
\rput[bl](2.95,0.35){$S^{-1}$}
\psline[linecolor=black, linewidth=0.026, arrowsize=0.05291667cm 2.0,arrowlength=1.4,arrowinset=0.0]{->}(2.8,-0.6)(2.5,-0.8)
\psbezier[linecolor=black, linewidth=0.026, arrowsize=0.05291667cm 2.0,arrowlength=1.4,arrowinset=0.0]{->}(2.5,1.2)(4.8,1.9)(4.8,-1.9)(2.5,-1.2)
\rput[bl](5.1,-0.35){$m$}
\rput[bl](5.06,0.2){$\Delta$}
\psline[linecolor=black, linewidth=0.026, arrowsize=0.05291667cm 2.0,arrowlength=1.4,arrowinset=0.0]{->}(6.0,0.3)(5.54,0.3)
\psline[linecolor=black, linewidth=0.026, arrowsize=0.05291667cm 2.0,arrowlength=1.4,arrowinset=0.0]{->}(5.6,-0.3)(6.0,-0.3)
\psbezier[linecolor=black, linewidth=0.026, arrowsize=0.05291667cm 2.0,arrowlength=1.4,arrowinset=0.0]{->}(5.0,0.5)(4.6,1.1)(3.9,1.6)(1.5,1.6)
\psbezier[linecolor=black, linewidth=0.026, arrowsize=0.05291667cm 2.0,arrowlength=1.4,arrowinset=0.0]{<-}(5.0,-0.5)(4.575,-1.1136364)(4.05,-1.6)(1.5,-1.6)
\rput[bl](2.95,-0.45){${\alpha}^{-1}$}
\psline[linecolor=black, linewidth=0.026, arrowsize=0.05291667cm 2.0,arrowlength=1.4,arrowinset=0.0]{->}(3.2,0.2)(3.2,-0.1)
\psbezier[linecolor=black, linewidth=0.026, arrowsize=0.05291667cm 2.0,arrowlength=1.4,arrowinset=0.0]{->}(2.5,1.0)(4.1,1.0)(4.3,-0.3)(4.9,-0.3)
\psbezier[linecolor=black, linewidth=0.026, arrowsize=0.05291667cm 2.0,arrowlength=1.4,arrowinset=0.0]{<-}(2.5,-1.0)(4.1,-1.0)(4.3,0.3)(4.9,0.3)
\rput[bl](0.0,-0.1){$m_{\alpha}=$}
\rput[bl](6.4,-0.3){,}
\rput[bl](12.7,-0.3){.}
\end{pspicture}
\end{figure}

\noindent Note that if $H$ is a Hopf algebra in $\SVect$ the above formulas involve various signs coming from the symmetry and the right pairings of $\SVect$. The antipode is defined by the conditions $S_{\a}(h)=\a^{-1}(S(h)), S_{\a}(p)=p\circ S^{-1}$ for $h\in H, p\in H^*$ and the fact that it is an algebra antiautomorphism. This defines a Hopf $\Aut(H)$-coalgebra. This is crossed with
\begin{align*}
\v_{\a}(p\ot h)=p\circ\a^{-1}\ot\a(h)
\end{align*}
and quasi-triangular with $R$-matrix
\begin{align}
\label{eq: twisted R-matrix}
R_{\a,\b}=\sum (\e\ot \a(h_i))\ot(h^i\ot 1)=\sum (\e\ot h_i)\ot(h^i\circ\a\ot 1)
\end{align}
where $(h_i)$ is any vector space basis of $H$ and $(h^i)$ is the dual basis of $H^*$ (for the left pairing). It is easy to see that $R_{\a,\b}$ is invertible with inverse
\begin{align}
\label{eq: twisted inverse R-matrix}
R^{-1}_{\a,\b}=(S_{D(H)}\ot\id_{D(H)})(R_{\a,\b}).
\end{align}
 %In other words, we just transported the $R$-matrix of the usual Drinfeld double $D(H)$ to all gradings. 
We call $\uDHH$ the {\em twisted Drinfeld double} of $H$. If $G$ is a subgroup of $\Aut(H)$, then the restriction $\{D(H)_{\a}\}_{\a\in G}$ to $G$ will be called the {\em $G$-twisted Drinfeld double}, denoted $\uDH|_G$. When $\a=\id_H$ the quasi-triangular Hopf algebra $D(H)=D(H)_{\id_H}$ is the usual Drinfeld double of $H$.% (see e.g. \cite{Kassel:BOOK}).

\medskip

%For the usual Drinfeld double $D(H)$, the Drinfeld element and its inverse are\begin{figure}[H]\centering\begin{pspicture}(0,-0.51)(7.95,0.51)\psbezier[linecolor=black, linewidth=0.026, arrowsize=0.05291667cm 2.0,arrowlength=1.4,arrowinset=0.0]{->}(3.3,0.49)(0.3,0.49)(0.9,-0.31)(1.7,-0.30999999999999944)\rput[bl](1.9,-0.51){$S^{-1}$}\psline[linecolor=black, linewidth=0.026, arrowsize=0.05291667cm 2.0,arrowlength=1.4,arrowinset=0.0]{->}(2.7,-0.31)(3.3,-0.31)\rput[bl](0.0,0.09){$u$}\psbezier[linecolor=black, linewidth=0.026, arrowsize=0.05291667cm 2.0,arrowlength=1.4,arrowinset=0.0]{->}(7.7,0.49)(4.7,0.49)(5.5,-0.31)(6.3,-0.30999999999999944)\rput[bl](6.5,-0.51){$S^{2}$}\psline[linecolor=black, linewidth=0.026, arrowsize=0.05291667cm 2.0,arrowlength=1.4,arrowinset=0.0]{->}(7.1,-0.31)(7.7,-0.31)\rput[bl](4.3,0.09){$u^{-1}$}\rput[bl](0.5,0.09){=}\rput[bl](5.1,0.09){=}\rput[bl](3.5,-0.31){,}\rput[bl](7.9,-0.31){.}\end{pspicture}\end{figure}

\noindent From (\ref{eq: graded Drinfeld element}) we see that the Drinfeld element of $\uDH$ and its inverse are given by 
\begin{figure}[H]
\begin{pspicture}(0,-0.485)(9.6,0.485)
\psbezier[linecolor=black, linewidth=0.026, arrowsize=0.05291667cm 2.0,arrowlength=1.4,arrowinset=0.0]{->}(3.6,0.465)(1.2,0.465)(1.6,-0.335)(2.2,-0.335)
\rput[bl](0.9,0.065){=}
\rput[bl](0.0,0.065){$u_{\alpha}$}
\rput[bl](2.39,-0.455){$S^{-1}$}
\psline[linecolor=black, linewidth=0.026, arrowsize=0.05291667cm 2.0,arrowlength=1.4,arrowinset=0.0]{->}(3.2,-0.335)(3.6,-0.335)
\psbezier[linecolor=black, linewidth=0.026, arrowsize=0.05291667cm 2.0,arrowlength=1.4,arrowinset=0.0]{->}(9.6,0.465)(5.8,0.465)(6.6,-0.335)(7.2,-0.335)
\rput[bl](5.9,0.065){=}
\rput[bl](4.95,-0.015){$u_{\alpha}^{-1}$}
\rput[bl](7.43,-0.485){$S^{2}$}
\psline[linecolor=black, linewidth=0.026, arrowsize=0.05291667cm 2.0,arrowlength=1.4,arrowinset=0.0]{->}(8.0,-0.335)(8.4,-0.335)
\rput[bl](8.67,-0.435){$\alpha$}
\psline[linecolor=black, linewidth=0.026, arrowsize=0.05291667cm 2.0,arrowlength=1.4,arrowinset=0.0]{->}(9.17,-0.335)(9.57,-0.335)
\rput[bl](4.1,-0.435){,}
\end{pspicture}
\end{figure}
%$$\ua=h^i\ot S^{-1}(h_i), \hspace{1cm} u_{\a}^{-1}=h^i\ot S^2(\a(h_i)).$$
%where $u$ is the Drinfeld element of $D(H)$.

%The following theorem is a version of the Kauffman-Radford theorem for the twisted Drinfeld double defined above.

\def\Hbos{H_{\text{bos}}}\def\SVect{\text{SVect}}\def\ZZ{\mathcal{Z}}

\medskip
%Q: Is there a categorical interpretation of this construction? Is $\Rep(\uDHH)$ some graded Drinfeld center of a graded category associated to $\Rep(H)$? A: First we should take the crossed product $\CC\rtimes \Aut(H)$ where $\CC=\Rep(H)$ and $\Aut(H)$ acts by autoequivalences. This is the categorical version of the Fox calculus Hopf $\Aut(H)$-coalgebra. Now need to take some Drinfeld center of this. Reconstruction should give the above graded Hopf algebra.

It is easy to see that the right integral of $\uDH$ is given by $\la_{\a}=\La_l\ot\la_r$ for every $\a$. The distinguished group-like of $\uDH$ is
\begin{equation}
\label{eq: dist group-like of TDD}
\gg_{\a}=\rH(\a)^{-1}\zz\ot \gg
\end{equation}
where $\rH(\a)\in \kk^{\t}$ is defined by (\ref{eq: rH}) and $\zz,\gg$ are the distinguished group-likes of $H$. If $H\in \Vect$, this follows from the following computation
\begin{align*}
(\id_{D(H)_a}\ot \la_{\b})\circ \De_{\a\b}(p\ot h)&=p_{(2)}\ot h_{(1)}\ot p_{(1)}(\La_l)\la_r(\a^{-1}(h_{(2)}))\\
&=\rH(\a)^{-1}p_{(2)}p_{(1)}(\La_l)\ot h_{(1)}\la_r(h_{(2)})\\
&=\rH(\a)^{-1}p(\La_l)\zz\ot \la_r(h)\gg\\
&=\la_{\a\b}(p\ot h)\cdot \rH(\a)^{-1}\zz\ot \gg
\end{align*}
where $p\ot h\in D(H)_{\a\b}$. A similar computation holds for $H\in \SVect$.

%\begin{remark}\label{remark: TDD for Super}The Drinfeld double (and its twisted extension) is also defined for Hopf superalgebras, where the product of $D(H)$ involves some signs coming from the symmetry of super-vector spaces. When $H$ is a Hopf superalgebra we will assume that $D(H)$ means this Drinfeld double. This is not the same as the Drinfeld double of the bosonization of $H$. See Remark \ref{remark: categorical interp SUPER double}.\end{remark} 

\subsection{Ribbon elements in the twisted double}\label{subs: Hopf ribbon element KR thm} %A necessary condition for the twisted Drinfeld double to be ribbon is that the usual Drinfeld double $D(H)$ is ribbon. 

We now characterize the ribbon structures on $\uDH|_G$, generalizing a theorem of Kauffman and Radford \cite[Theorem 13.7.3]{Radford:BOOK}. %We begin by characterizing the $G$-invariant group-likes of $\uDH$.

\begin{lemma}
\label{lemma: group-likes of uDH}
The $G$-invariant group-likes of $\uDH|_G$ have the form $\{g_{\a}=p(\a)^{-1}\b \ot b\}_{\a\in G}$ for unique $G$-invariant group-likes $\b\in H^*, b\in H$ and a unique homomorphism $p:G\to\kk^{\t}$.
\end{lemma}
\begin{proof}
If $\{g_{\a}\}$ is a $G$-invariant group-like of $\uDH|_G$, then $g_1$ is a group-like of $D(H)$ so $g_1=\b\ot b$ for unique $\b\in G(H^*), b\in G(H)$. Since $g_1$ is $G$-invariant, $\b,b$ must be $G$-invariant. Now, we have
\begin{align*}
g_{\a}\ot g_{\a^{-1}}=\De_{\a,\a^{-1}}(g_1)=(\b\ot b)\ot(\b\ot b)
\end{align*}
since $\a^{-1}(b)=b$. It follows that $g_{\a}=p_{\a}^{-1}\cdot \b\ot b$ for some scalar $p_{\a}$ (apply $\id\ot f$ where $f:D(H)_{\a^{-1}}\to\kk$ satisfies $f(g_{\a^{-1}})\neq 0$). It is clear that $p_{\a}$ is unique, and the group-like condition implies that $\a\mapsto p_{\a}$ defines an homomorphism $p:G\to \kk^{\t}$.
\end{proof}

\begin{proposition}\label{prop: ribbon of TDD}
The ribbon elements of the $G$-twisted Drinfeld double $\uDH|_G$ are in 1:1 correspondence with triples $(\b,b,p)$ where
\begin{enumerate}
\item $\b\in H^*, b\in H$ are $G$-invariant group-likes such that $S^2_H=\ad_{\b^{-1}}\circ \ad_b$ and $\b^2=\zz, b^2=\gg$ (where $\zz,\gg$ are the distinguished group-likes of $H$).
\item $p:G\to \kk^{\t}$ is an homomorphism such that $p^2=\rH$.
\end{enumerate} 
The ribbon element and pivot of $\uDH|_G$ corresponding to such triple are given by $$\va=\rH(\a)^{-\frac{1}{2}}(\b\ot b)\cdot u_{\a}^{-1}, \ g_{\a}=\rH(\a)^{-\frac{1}{2}}\b\ot b.$$
for each $\a\in G$, where we denote $p(\a)=\rH(\a)^{\frac{1}{2}}$ and $u_{\a}$ is the Drinfeld element.
\end{proposition}

\begin{proof}
By \cite[Theorem 6.9, $b)$]{Virelizier:Hopfgroup} the distinguished group-like $\gg_{\a}$ of $\uDH$ satisfies $\gg_{\a}=u_{\a}S_{\a^{-1}}(u_{\am}^{-1})$ (this also uses that $D(H)$ is unimodular and that the character of \cite[Lemma 6.2]{Virelizier:Hopfgroup} is trivial). This holds in the super-case as well. Thus, by Lemma \ref{lemma: ribbon bijection group-likes} and (\ref{eq: dist group-like of TDD}), ribbon elements of $\uDH|_G$ are in bijection with $G$-invariant group-likes $\{g_{\a}\}$ satisfying $(3')$ and
\begin{align*}
g_{\a}^2=\gg_{\a}=\rH(\a)^{-1}\zz\ot\gg.
\end{align*}
If $(\b,b,p)$ corresponds to $\{g_{\a}\}$ by Lemma \ref{lemma: group-likes of uDH} then $(4')$ is equivalent to
\begin{align*}
\b^2=\zz, \ b^2=\gg, \ p(\a)^2=\rH(\a).
\end{align*}
It is easy to see that $(3')$ is equivalent to $S^2_H=\ad_{\b^{-1}}\circ\ad_b$. This proves the first part of the proposition. The ribbon element associated to such $\{g_{\a}\}$ is then given by 
\begin{align*}
v_{\a}=g_{\a}u_{\a}^{-1}=p(\a)^{-1}(\b\ot b)u_{\a}^{-1}
\end{align*}
as desired.
\end{proof}

%EXPRESS THIS BETTER BY USING THAT RIBBON IS A SQUARE ROOT OF CASIMIR

\def\dgl{\gg_{\a}}

\medskip

\def\Autbb{\Aut_{(b,\b)}}\def\rH{r_H}

%Let $\b,b$ be as before. We denote by $\Autbb(H)$ the subgroup of $\Aut(H)$ of automorphisms that fix $b$ and $\b$. Note that in many case (e.g. Taft algebras at odd roots of unity) the ribbon element of the double is unique, so that $\Autbb(H)=\Aut(H)$. %Then it is easy to see that setting $\va=v$ for all $\a$, we get a ribbon Hopf $\Autbb(H)$-coalgebra.

\subsection{Representation theoretic interpretation}\label{subs: Hopf Rep theory} We now give a more familiar description of the representation categories of the above twisted Drinfeld doubles. Since this is not essential for this paper, we just briefly recall the objects involved. We refer to \cite{Turaev:homotopy} for definitions of braided $G$-crossed categories and \cite{GNN:centers} for relative Drinfeld centers. In what follows, $G$ denotes a group (not necessarily finite) with neutral element $1$ and $\Rep(A)$ denotes the category of finite dimensional representations of a $\kk$-algebra $A$.
%\def\wCC{\widetilde{\CC}}\def\wot{\widetilde{\ot}}
%\def\DD{\mathcal{D}}
%\def\ZZ{\mathcal{Z}}\def\ZZDCt{\ZZ_{\DD}(\wCC)}\def\ZZDC{\ZZ_{\DD}(\CC)}\def\ZZCCt{\ZZ_{\CC}(\wCC)}\def\ZZa{\ZZ_{\a}}\def\ZZb{\ZZ_{\b}}\def\wtc{\widetilde{c}}
%\def\Zot{\ot_{\ZZ}}

%\begin{remark}The above set of axioms for a ribbon Hopf $G$-coalgebra $\uHH$ can be simply restated by saying that the category $\CC\eq\coprod_{\a\in G}\CC_{\a}$, where $\CC_{\a}=\Rep(\Ha)$, is a {\em strict} $G$-crossed ribbon category. Here strict means that the action $\a\mapsto (F_{\a}:\CC\to\CC)$ (where $F_{\a}:\CC_{\b}\to\CC_{\a\b\a^{-1}}$ is defined using $\v_{\a}^{-1}:A_{\a\b\a^{-1}}\to A_{\b}$) is strict in the sense that the isomorphisms $F_{\a}(X\ot Y)\to F_{\a}(X)\ot F_{\a}(Y)$ are identities as well as the natural isomorphisms $F_{\a}F_{\b}\cong F_{\a\b}$ for each $\a,\b\in G$. %It would be interesting to consider non-strict actions as well, but we do not do it in this paper.\end{remark}

\def\wCC{\CC}\def\wot{\ot}
\def\DD{\mathcal{D}}
\def\ZZ{\mathcal{Z}}\def\ZZDCt{\ZZ_{\DD}(\CC)}\def\ZZDC{\ZZ_{\DD}(\CC)}\def\ZZCCt{\ZZ_{\DD}(\CC)}\def\ZZa{\ZZ_{\a}}\def\ZZb{\ZZ_{\b}}\def\wtc{c}
\def\Zot{\ot}\def\ZZCC{\ZZ_{\DD}(\CC)}
\medskip

Let $H$ be a finite dimensional Hopf algebra (in $\Vect$, for simplicity) and $G=\Aut(H)$. Since $\uDH$ is a quasi-triangular Hopf $G$-coalgebra, $\Rep\ \uDH:=\coprod_{\a\in G}\Rep(D(H)_{\a})$ is a braided $G$-crossed category in a natural way \cite{Turaev:homotopy}. This category can be defined entirely in terms of the monoidal category $\DD=\Rep(H)$ as follows. For each $X\in\DD$ and morphism $f$ of $\DD$, let $T_{\a}(X)$ be the vector space $X$ with $H$-action given by $h\cdot_{\a}x:=\a^{-1}(h)x$ and $T_{\a}(f)=f$. Then $\a\mapsto T_{\a}$ defines a monoidal action of $G$ on $\DD$. Let $\wCC=\DD\rtimes G$, that is, $\wCC$ is the category of pairs $(X,\a)$ where $X\in\DD,\a\in G$ and $\Hom_{\wCC}((X,\a),(Y,\b))$ equals $\Hom_{\DD}(X,Y)$ if $\a=\b$ and zero otherwise. This is a $G$-graded monoidal category with tensor product on objects given by $$(X,\a)\wot(Y,\b)=(X\ot T_{\a}(Y),\a\b)$$ and on morphisms $(f,\a)\wot(g,\b)=(f\ot T_{\a}(g),\a\b)$.
The neutral component is $\DD$ as a monoidal category.  % and $T_{\b}(\wCC_{\a})\sb \wCC_{\b\a\b^{-1}}$ so $\wCC$ is $G$-crossed. 
We consider the relative Drinfeld center $\ZZCCt$ of this category. An object of $\ZZCC$ is an object $(Y,\a)$ of $\CC$ endowed with a half-braiding $\wtc_{(X,1),(Y,\a)}:(X,1)\wot(Y,\a)\to (Y,\a)\wot (X,1)$, which is the same as a $\DD$-morphism $$c_{X,Y}:X\ot Y\to Y\ot T_{\a}(X)$$ satisfying $c_{X\ot X',Y}=(c_{X,Y}\ot \id_{T_{\a}X'})(\id_X\ot c_{X',Y})$ for each $X,X'\in\DD$. %Note that this equality should also involve the isomorphisms $T_{\a}(X\ot X')\cong T_{\a}(X)\ot T_{\a}(X')$, but we do not include them since they are all identities in the present situation. 
As shown in \cite[Example 3.4]{GNN:centers}, $\ZZCC$ is a braided $G$-crossed category in the following way. First note that the group $G$ acts on $\wCC$ by $T_{\b}(Y,\a)=(T_{\b}(Y),\b\a\b^{-1})$. This action extends to $\ZZCCt$ if the half-braiding of $T_{\b}(Y,\a)$ is defined, for $X\in\DD$, by 
\begin{align}
\label{eq: G-action on rel Z}
\wtc_{(X,1),T_{\b}(Y,\a)}=T_{\b}(\wtc_{(T_{\b^{-1}}(X),1),(Y,\a)}).
\end{align}
This is the same as the $\DD$-morphism $T_{\b}(c_{T_{\b^{-1}}(X),Y}).$ The $G$-braiding is defined by the following isomorphisms 
\begin{align*}
(X,\a)\ot (Y,\b)=(X\ot T_{\a}Y,\a\b)&\xrightarrow{c_{X,T_{\a}(Y)}} (T_{\a}Y\ot T_{\a\b\a^{-1}}X,\a\b)\\
&=(T_{\a}Y,\a\b\a^{-1})\ot (X,\a)\\
&=T_{\a}(Y,\b)\ot (X,\a).
\end{align*}

\begin{proposition}
\label{prop: TDD is REL DRINFELD}
Let $H$ be a finite dimensional Hopf algebra and $G=\Aut(H)$. Then, $\Rep \ \uDH$ is equivalent, as a braided $G$-crossed category, to the relative Drinfeld center of the crossed product $\wCC=\Rep(H)\rtimes G$.
\end{proposition}
\begin{proof}
The proof is very similar to that of the $G=1$ case (see e.g. \cite{Kassel:BOOK}) so we only sketch it. Denote $D_{\a}=D(H)_{\a}$ for each $\a$. The idea is that if $(X,\a)$ is an object of $\ZZCC$ then, by naturality, the whole half-braiding is encoded in the half-braiding $c_{H,X}:H\ot X\to X\ot T_{\a}H$. Again by naturality, this map is determined by $c_{H,X}(1\ot -)$ which is a right $H$-comodule structure on $X$, or equivalently, a left $H^*$-module structure. Noting that the $H$-factor in the target of $c_{H,X}$ has a twisted $H$-action (twisted by $\a^{-1}$), the condition that $c_{H,X}$ is an $H$-module map translates into $X$ being a module over $D_{\a}$. This correspondence is clearly an equivalence and it is monoidal: since the tensor product $(X,\a)\ot (Y,\b)=(X\ot T_{\a}Y,\a\b)$ twists by $\a^{-1}$ the $H$-module structure on $Y$, the comultiplication $\De_{\a\b}$ is twisted by $\a^{-1}$ in the second $H$-factor, which is exactly the formula we gave for $\De_{\a\b}$. Since the $G$-action on $\ZZCC$ is $T_{\a}(Y,\b)=(T_{\a}Y,\a\b\a^{-1})$ the $H$-module structure of $T_{\a}(Y,\b)$ is twisted by $\a^{-1}$ so $\v_{\a}:D_{\b}\to D_{\a\b\a^{-1}}$ acts by $\a$ on $H\sb D_{\a}$. Now, the $H$-comodule structure of $T_{\a}(Y,\b)$ is given by the half-braiding $c_{T_{\a^{-1}}H,Y}$ (see (\ref{eq: G-action on rel Z})). Since $\a:H\to T_{\a^{-1}}H$ is an $H$-module map, by naturality one finds 
\begin{equation}
\label{eq: Z rel proof}
c_{T_{\a^{-1}}H,Y}(1\ot y)=(\id\ot \a)c_{H,Y}(1\ot y).
\end{equation}
Thus, the $H$-comodule structure of $T_{\a}Y$ is twisted by $\a$, so that $\v_{\a}$ acts by $\v_{\a}(h^*)=h^*\circ\a$ on $H^*$. Taken together, this shows the equivalence is a $G$-crossed equivalence. Finally, since $\Rep (\uDH)=\ZZCC$ as $G$-categories, the braiding of $\ZZCC$ must be represented by an $R$-matrix $R'_{\a,\b}$ of $\uDH$. By naturality and (\ref{eq: Z rel proof}) above, this is given by
\begin{align*}
P(R'_{\a,\b})&=c_{T_{\a^{-1}}D_{\a},D_{\b}}(1_{\a}\ot 1_{\b})\\
&=(\id\ot\a)c_{H,D_{\b}}(1_{\a}\ot 1_{\b})
\end{align*}
where $P$ is the switch map $P(x\ot y)=y\ot x$ and on the bottom we think of $H$ as a subalgebra of $T_{\a^{-1}}D_{\a}$. But, $c_{H,D_{\b}}(1_{\a}\ot 1_{\b})=h^i\ot h_i$ as in the untwisted case, where $h_i$ is a basis of $H$ and $h^i$ is its dual basis. Therefore, $R'_{\a,\b}=P((\id\ot\a)(h^i\ot h_i))=R_{\a,\b}$ where $R_{\a,\b}$ is as in (\ref{eq: twisted R-matrix}).

\end{proof}

%In particular, it follows that the above relative Drinfeld center is a $G$-crossed braided category. 

\begin{remark}
\label{remark: categorical interp SUPER double}
Suppose $H$ is a Hopf algebra in $\SVect$. Then, it is not true that $\Rep_S(D(H))$ is the Drinfeld center of $\Rep_S(H)$ (the $S$ subscript indicates that modules are super-vector spaces). Indeed, $\Rep_S(H)$ is equivalent to $\Rep(\Hbos)$ where $\Hbos$ is the {\em bosonization} of $H$. As an algebra, $\Hbos=\kk[\Z/2\Z]\ltimes H$ and so $D(\Hbos)$ has two additional group-likes. This double is itself the bosonization of a Hopf algebra $D$ in $\SVect$ such that $\Rep_S(D)=\ZZ(\Rep_S(H))$, and this is an extension of $D(H)$ by a single group-like. The category $\Rep_S(D(H))$ turns out to be equivalent to the full (braided) subcategory of $\ZZ(\Rep_S(H))$ consisting of objects $V\in\Rep_S(H)$ with a half-braiding $\{\s_{X,V}:X\ot V\to V\ot X\}_{X\in\Rep_S(H)}$ for which $\s_{I,V}$ coincides with the symmetry $c_{I,V}:I\ot V\to V\ot I$ of $\SVect$. Here $I$ denotes the unique non-trivial invertible object of $\SVect$, this is an $H$-module via $\e$.
\end{remark}

%\subsection{Automorphism groups} We end with some remarks on automorphism groups of Hopf algebras.Therefore, automorphism groups are generally not too big, except in special cases. A more general approach would be to consider groups acting monoidally on categories of modules instead. This is much more general, for instance, the group $SL(2,\C)$ acts on the category of quantum $\sl2$-modules at a root of unity, but only as ``twisted automorphisms".

\section{Universal Reshetikhin-Turaev invariants of $G$-tangles}\label{section: quantum invariants of G-tangles}

\def\pmX{\p_-X}\def\ppX{\p_+X}\def\TT{\mathcal{T}}\def\pvX{\p_vX}

In this section we define the universal twisted invariant of a framed, oriented, $G$-tangle with no closed components out of a ribbon Hopf $G$-coalgebra. In Subsection \ref{subs: inv of G-knots} we explain how to get invariants of framed and unframed $G$-knots as closures of long $G$-knots. In Subsection \ref{subs: invs from TDDs} we specialize to twisted Drinfeld doubles and in Subsection \ref{subs: invs in Z-GRADED case} we explain how to lift our invariants to polynomials invariants whenever the Hopf algebras are $\Z$-graded.

\subsection{$G$-tangles and (long) $G$-knots} Let $X=\R\t[-1,\infty)\t [0,1]$ and let $\pmX=\R\t[-1,\infty)\t \{0\}$, $\ppX= \R\t[-1,\infty)\t\{1\}$, $\pvX=\R\t\{-1\}\t[0,1]$. We think of the $x$-coordinate as an horizontal line in the plane of the page oriented from left to right, the $y$-coordinate as a line transversal to the page and oriented towards it (so $(0,\infty)$ lies behind the page), and the $z$-coordinate as a vertical line in the plane of the page oriented upwards. By a {\em $(p,q)-$tangle} we mean a framed, oriented tangle $T\sb X$ such that $T\cap \pmX=\{(i,0,0)\}_{i=1}^p$, $T\cap \ppX=\{(i,0,1)\}_{i=1}^q$, $T\cap \p X=T\cap (\pmX\cup\ppX)=\p T$ and the intersection is transversal. Moreover, the framing of $T$ near a point $(i,0,0)$ (resp. $(i,0,1)$) is $(i-\d,0,0)$ (resp. $(i-\d,0,1)$) for some small $\d>0$. Let $X_T=X\sm T$ and $z\in \pvX$ a basepoint. A {\em $(p,q)$-$G$-tangle} is a $(p,q)$-tangle endowed with a group homomorphism $\rho:\pi_1(X_T,z)\to G$. Note that since $\pvX$ is contractible, the choice of basepoint $z$ is irrelevant. Two $G$-tangles $(T,\rho)$ and $(T',\rho')$ are isotopic if there is an isotopy $d_t:X\to X$ rel. $\p X$ from $T$ to $T'$ such that $\rho'\circ (d_1)_*=\rho$. %It is clear that isotopy classes of $G$-tangles in $X$ are the same as isotopy classes of $G$-tangles in $\R^2\t[0,1]$ or $\DD\t[0,1]$ (NOTE: $\DD\t[0,1]$ needs additional structure!), but we find $X$ more convenient. 
Isotopy classes of $G$-tangles form the morphisms of a $G$-crossed ribbon category, see \cite{Turaev:homotopy} for details.

\medskip

A $(p,q)$-$G$-tangle $(T,\rho)$ with a single component will be called a {\em $G$-knot} if $p=q=0$ and a {\em long $G$-knot} if $p=q=1$.

%One can organize $G$-tangles into a (strict) monoidal category $\TT_G$ as follows: the objects are finite ordered sequences of pairs $(g,\e)$ with $g\in G$ and $\e\in\{\pm 1\}$ and morphisms are isotopy classes of $G$-tangles. More precisely, if $(T,\rho)$ is a $(p,q)$-$G$-tangle, then its source is the sequence $\{(\rho(\c_i),\e_i)\}_{i=1}^p$ where $\c_i\sb\p X$ is the unique homotopy class of a loop that goes linearly from the basepoint to a point near $(i,0,0)$, encircles the $i$-th input of $T$ once with linking number $-1$ and then goes back again by a linear path (see Figure \ref{figure: tangle with loops}), while $\e_i=+1$ (resp. $\e_i=1$) if at this input $T$ is oriented upwards (resp. downwards). The target object is defined similarly, using loops $\c_i\in\pi_1(X_T)$ encircling the $i$-th output once with linking number -1 as well. Composition of $G$-tangles is defined by stacking one tangle over another, the representation $\rho$ of the composition is defined by Van Kampen's theorem. The tensor product of two objects is obtained by juxtaposing the second factor to the right of the first factor, and the tensor of two morphisms (= tangles) is obtained by stacking the second factor to the right of the first one along the $x$-axis of $X$. Again, this defines a $G$-tangle by Van Kampen's theorem. The category of $G$-tangles has the structure of a crossed ribbon $G$-category, see \cite{Turaev:homotopy}.
\medskip

%\subsection{Colored $G$-tangles} Now let $\CC$ be a crossed ribbon $G$-category. By \cite{Galindo:coherence} we can suppose $\CC$ is strict, that is, the $G$-action consists of functors $\v_g:\CC\to\CC$ satisfying $\v_g(X\ot Y)=\v_g(X)\ot\v_g(Y)$ for all $X,Y\in\CC$ and $\v_g\v_h=\v_{gh}$ for all $g,h\in G$. Then, we define a {\em $\CC$-coloring} of a $G$-tangle $(T,\rho)$ as a function that to any path $\c\sb X\sm T$ from the basepoint $z$ to $\wt{T}$ associates an object $U_{\c}\in\CC$ with the property that for any $\b\in\pi_1(X_T,z)$ we have $U_{\b\c}=\v_{\rho(\b)}(U_{\c})$.

\subsection{Invariants of $G$-tangles}
\label{subs: invs of G-tangles}

\def\Hg{A_{\a}}\def\za{z_{\a}}\def\Hc{A_{\c}}
Let $(\uHH,\v,R,v)$ be a ribbon Hopf $G$-coalgebra with pivot $\{\ga\}_{\a\in G}$. Let $(T=T_1\cup\dots\cup T_m, \rho)$ be an ordered, oriented, framed $G$-tangle with no closed components. %Eventually we will only consider twisted Drinfeld doubles $A_{\a}=D(H)_{\a}$ so the reader can restrict to this case.
\medskip

%USE WORD "EDGES" FOR ARCS OF DIAGRAMS
%USE "ARCS" ONLY FOR BRIDGE PRESENTATIONS

Let $D$ be a planar diagram of $T$. We draw $D$ in the plane and assume it comes with the blackboard framing. We also suppose $D$ is oriented upwards at each crossing, so it is built of the following local pieces:
\begin{figure}[H]
\begin{pspicture}(0,-0.59241927)(9.047577,0.59241927)
\psline[linecolor=black, linewidth=0.026](2.0275772,-0.56484205)(2.827577,0.23515792)
\psline[linecolor=white, linewidth=0.026, doubleline=true, doublesep=0.026, doublecolor=black](2.4275773,0.23515792)(3.2275772,-0.56484205)
\psline[linecolor=black, linewidth=0.026](1.2275772,-0.56484205)(0.4275772,0.23515792)
\psline[linecolor=white, linewidth=0.026, doubleline=true, doublesep=0.026, doublecolor=black](0.82757723,0.23515792)(0.02757721,-0.56484205)
\psline[linecolor=black, linewidth=0.026, arrowsize=0.05291667cm 2.0,arrowlength=1.4,arrowinset=0.0]{<-}(0.02757721,0.63515794)(0.4275772,0.23515792)
\psline[linecolor=black, linewidth=0.026, arrowsize=0.05291667cm 2.0,arrowlength=1.4,arrowinset=0.0]{<-}(1.2275772,0.63515794)(0.82757723,0.23515792)
\psline[linecolor=black, linewidth=0.026, arrowsize=0.05291667cm 2.0,arrowlength=1.4,arrowinset=0.0]{<-}(2.0275772,0.63515794)(2.4275773,0.23515792)
\psline[linecolor=black, linewidth=0.026, arrowsize=0.05291667cm 2.0,arrowlength=1.4,arrowinset=0.0]{<-}(3.2275772,0.63515794)(2.827577,0.23515792)
\psbezier[linecolor=black, linewidth=0.026](7.6275773,-0.26484206)(7.6275773,0.03515793)(7.427577,0.33515793)(7.227577,0.3351579284667969)(7.0275774,0.33515793)(6.827577,0.03515793)(6.827577,-0.26484206)
\psline[linecolor=black, linewidth=0.026, arrowsize=0.09cm 2.0,arrowlength=0.6,arrowinset=0.2]{->}(7.207577,0.33515793)(7.294244,0.33515793)
\psbezier[linecolor=black, linewidth=0.026](9.027577,0.33515793)(9.027577,0.03515793)(8.827578,-0.26484206)(8.627577,-0.2648420715332031)(8.427577,-0.26484206)(8.227577,0.03515793)(8.227577,0.33515793)
\psline[linecolor=black, linewidth=0.026, arrowsize=0.09cm 2.0,arrowlength=0.6,arrowinset=0.2]{->}(8.607577,-0.26484206)(8.694243,-0.26484206)
\psbezier[linecolor=black, linewidth=0.026](5.427577,0.33515793)(5.427577,0.03515793)(5.6275773,-0.26484206)(5.827577,-0.2648420715332031)(6.0275774,-0.26484206)(6.227577,0.03515793)(6.227577,0.33515793)
\psline[linecolor=black, linewidth=0.026, arrowsize=0.09cm 2.0,arrowlength=0.6,arrowinset=0.2]{->}(5.847577,-0.26484206)(5.7609105,-0.26484206)
\psbezier[linecolor=black, linewidth=0.026](4.0275774,-0.26484206)(4.0275774,0.03515793)(4.227577,0.33515793)(4.427577,0.3351579284667969)(4.6275773,0.33515793)(4.827577,0.03515793)(4.827577,-0.26484206)
\psline[linecolor=black, linewidth=0.026, arrowsize=0.09cm 2.0,arrowlength=0.6,arrowinset=0.2]{->}(4.447577,0.33515793)(4.3609104,0.33515793)
\end{pspicture}
\end{figure}

\noindent By an edge of $D$ we mean a subarc of the diagram having no overpasses in its interior and ending at underpasses. For each edge $e$ of $D$ we let $\c_e\in\pi_1(X_T,z)$ be the homotopy class of the loop that goes from the basepoint $z\in \p_v X$ to a point close to $e$ along a linear path, then encircles $e$ once with linking number $-1$ and finally comes back  to $z$ with a linear path in the opposite direction. Thus, the edges of $D$ become labeled by elements of $G$, where the label of $e$ is $\rho(\c_e)$:

\begin{figure}[h]
\centering
\begin{pspicture}(0,-1.275)(9.49,1.275)
\psline[linecolor=black, linewidth=0.026, arrowsize=0.05291667cm 2.0,arrowlength=1.4,arrowinset=0.0]{<-}(3.1,0.825)(3.7,0.225)
\psline[linecolor=black, linewidth=0.026, arrowsize=0.05291667cm 2.0,arrowlength=1.4,arrowinset=0.0]{<-}(4.7,0.825)(4.1666665,0.29166666)
\psline[linecolor=white, linewidth=0.026, doubleline=true, doublesep=0.026, doublecolor=black](3.7,0.225)(4.7,-0.775)
\psline[linecolor=white, linewidth=0.026, doubleline=true, doublesep=0.026, doublecolor=black](4.1666665,0.29166666)(3.1,-0.775)
\psbezier[linecolor=white, linewidth=0.026, doubleline=true, doublesep=0.026, doublecolor=black](0.5222222,-0.3824074)(2.4333334,-0.22685185)(3.1592593,-0.49351853)(3.5148149,-0.575)(3.8703704,-0.6564815)(3.8777778,-0.3824074)(3.6925926,-0.2935185)
\psbezier[linecolor=black, linewidth=0.026, arrowsize=0.05291667cm 2.0,arrowlength=1.4,arrowinset=0.0](0.5222222,-0.3824074)(1.462963,0.039814815)(3.1666667,-0.08611111)(3.5296297,-0.23425925925925925)
\psdots[linecolor=black, dotsize=0.14](0.5222222,-0.3824074)
\rput[bl](1.7222222,-0.7675926){$\gamma_e$}
\psline[linecolor=black, linewidth=0.026, arrowsize=0.05291667cm 2.0,arrowlength=1.4,arrowinset=0.0]{->}(1.9051852,-0.09574074)(2.1666667,-0.08092593)
\rput[bl](4.7,0.225){$e$}
\rput[bl](2.79,0.235){$e'$}
\rput[bl](4.7,-0.475){$e''$}
\rput[bl](0.0,-0.575){$z$}
\psline[linecolor=black, linewidth=0.026, arrowsize=0.05291667cm 2.0,arrowlength=1.4,arrowinset=0.0]{<-}(7.5,0.825)(8.1,0.225)
\psline[linecolor=black, linewidth=0.026, arrowsize=0.05291667cm 2.0,arrowlength=1.4,arrowinset=0.0]{<-}(9.1,0.825)(8.566667,0.29166666)
\psline[linecolor=white, linewidth=0.026, doubleline=true, doublesep=0.026, doublecolor=black](8.1,0.225)(9.1,-0.775)
\psline[linecolor=white, linewidth=0.026, doubleline=true, doublesep=0.026, doublecolor=black](8.566667,0.29166666)(7.5,-0.775)
\rput[bl](7.1,-1.175){$\alpha$}
\rput[bl](9.3,-1.275){$\beta$}
\rput[bl](6.7,0.925){$\alpha\beta\alpha^{-1}$}
\end{pspicture}
\caption{A diagram of a crossing with three edges $e,e',e''$. The paths $\c_e,\c_{e'},\c_{e''}$ satisfy $\c_{e'}=\c_e\c_{e''}\c_e^{-1}$.  On the right we draw the correponding $G$-colored diagram, where $\a=\rho(\c_e),\beta=\rho(\c_{e''})$.}
\label{figure: tangle with loops}
\end{figure}
 
To define an invariant of $(T,\rho)$ out of $(\uH,\v,R,v)$ we associate black and white beads to each of the local pieces of the $G$-colored diagram as follows:

\begin{figure}[H]
\centering
\begin{pspicture}(0,-1.055)(14.093333,1.055)
\psline[linecolor=black, linewidth=0.026](3.5333333,-0.5587037)(4.3333335,0.24129629)
\psline[linecolor=white, linewidth=0.026, doubleline=true, doublesep=0.026, doublecolor=black](3.9333334,0.24129629)(4.733333,-0.5587037)
\psline[linecolor=black, linewidth=0.026](1.5333333,-0.5587037)(0.73333335,0.24129629)
\psline[linecolor=white, linewidth=0.026, doubleline=true, doublesep=0.026, doublecolor=black](1.1333333,0.24129629)(0.33333334,-0.5587037)
\psline[linecolor=black, linewidth=0.026, arrowsize=0.05291667cm 2.0,arrowlength=1.4,arrowinset=0.0]{<-}(0.33333334,0.64129627)(0.73333335,0.24129629)
\psline[linecolor=black, linewidth=0.026, arrowsize=0.05291667cm 2.0,arrowlength=1.4,arrowinset=0.0]{<-}(1.5333333,0.64129627)(1.1333333,0.24129629)
\psline[linecolor=black, linewidth=0.026, arrowsize=0.05291667cm 2.0,arrowlength=1.4,arrowinset=0.0]{<-}(3.5333333,0.64129627)(3.9333334,0.24129629)
\psline[linecolor=black, linewidth=0.026, arrowsize=0.05291667cm 2.0,arrowlength=1.4,arrowinset=0.0]{<-}(4.733333,0.64129627)(4.3333335,0.24129629)
\psdots[linecolor=black, dotsize=0.16](0.73333335,-0.1587037)
\psdots[linecolor=black, dotsize=0.16](1.1333333,-0.1587037)
\psdots[linecolor=black, dotsize=0.16](3.9333334,0.24129629)
\psdots[linecolor=black, dotsize=0.16](4.3333335,0.24129629)
\rput[bl](0.13333334,-0.2587037){$s_{\alpha}$}
\rput[bl](1.4333333,-0.3587037){$t_{\beta}$}
\rput[bl](3.1333334,0.1412963){$\overline{s_{\alpha}}$}
\rput[bl](4.733333,0.041296296){$\overline{t_{\beta}}$}
\psdots[linecolor=black, fillstyle=solid, dotstyle=o, dotsize=0.14, fillcolor=white](0.67777777,0.28203705)
\psdots[linecolor=black, fillstyle=solid, dotstyle=o, dotsize=0.14, fillcolor=white](3.8333333,-0.28833333)
\rput[bl](0.0,0.20425926){$\varphi_{\alpha}$}
\rput[bl](2.9407408,-0.3475926){$\varphi_{\alpha^{-1}}$}
\rput[bl](0.085185185,-0.955){$\alpha$}
\rput[bl](1.5185186,-1.055){$\beta$}
%\rput[bl](2.8851852,-1.055){$\alpha\beta\alpha^{-1}$}
\rput[bl](3.3185184,0.845){$\alpha$}
\rput[bl](4.8185186,0.745){$\beta$}
\psbezier[linecolor=black, linewidth=0.026](11.833333,-0.1587037)(11.833333,0.1412963)(11.633333,0.4412963)(11.433333,0.4412962962962962)(11.233334,0.4412963)(11.033334,0.1412963)(11.033334,-0.1587037)
\psline[linecolor=black, linewidth=0.026, arrowsize=0.09cm 2.0,arrowlength=0.6,arrowinset=0.2]{->}(11.413333,0.4412963)(11.5,0.4412963)
\psbezier[linecolor=black, linewidth=0.026](13.333333,0.4412963)(13.333333,0.1412963)(13.133333,-0.1587037)(12.933333,-0.15870370370370382)(12.733334,-0.1587037)(12.533334,0.1412963)(12.533334,0.4412963)
\psline[linecolor=black, linewidth=0.026, arrowsize=0.09cm 2.0,arrowlength=0.6,arrowinset=0.2]{->}(12.913333,-0.1587037)(13.0,-0.1587037)
\psdots[linecolor=black, dotsize=0.14](13.3133335,0.24129629)
\psdots[linecolor=black, dotsize=0.14](11.0633335,0.041296296)
\rput[bl](13.573334,0.121296294){$g_{\alpha}^{-1}$}
\rput[bl](10.463333,-0.038703702){$g_{\alpha}$}
\psbezier[linecolor=black, linewidth=0.026](8.533334,0.4412963)(8.533334,0.1412963)(8.733334,-0.1587037)(8.933333,-0.15870370370370382)(9.133333,-0.1587037)(9.333333,0.1412963)(9.333333,0.4412963)
\psline[linecolor=black, linewidth=0.026, arrowsize=0.09cm 2.0,arrowlength=0.6,arrowinset=0.2]{->}(8.953333,-0.1587037)(8.866667,-0.1587037)
\psbezier[linecolor=black, linewidth=0.026](6.8333335,-0.1587037)(6.8333335,0.1412963)(7.0333333,0.4412963)(7.233333,0.4412962962962962)(7.4333334,0.4412963)(7.633333,0.1412963)(7.633333,-0.1587037)
\psline[linecolor=black, linewidth=0.026, arrowsize=0.09cm 2.0,arrowlength=0.6,arrowinset=0.2]{->}(7.2533336,0.4412963)(7.1666665,0.4412963)
\psdots[linecolor=black, dotsize=0.14](6.863333,0.041296296)
\psdots[linecolor=black, dotsize=0.14](9.3133335,0.2312963)
\rput[bl](6.233333,-0.058703702){$1_{\alpha}$}
\rput[bl](9.653334,0.1412963){$1_{\alpha}$}
\rput[bl](7.5333333,-0.5087037){$\alpha$}
\rput[bl](10.933333,-0.5187037){$\alpha$}
\rput[bl](8.443334,0.64129627){$\alpha$}
\rput[bl](13.233334,0.6512963){$\alpha$}
\end{pspicture}
\end{figure}

\noindent For positive (resp. negative) crossings, the black beads represent $R_{\a,\b}=\sum s_{\a}\ot t_{\b}$ (resp. $R_{\a,\b}^{-1}=\sum \ov{s_{\a}}\ot\ov{t_{\b}}$) where $\a,\b$ are the $G$-labels of the corresponding edges and the white bead represents $\v_{\a}$ (resp. $\v_{\a^{-1}}$).  %Here, one only needs to recall the two leftmost pictures: the other pictures are obtained by applying $S$ whenever an arrow is oriented downwards (this is because orienting downwards means taking the dual and this implies that the antipode has been applied). 
Now, for each cap or cup of the diagram labeled by $\a\in G$, we put a black bead representing either $1_{\a}$ (left caps/cups) or the pivot $g_{\a}^{\pm 1}$ (right caps/cups). Note that this is where we use the ribbon element $v_{\a}$ of $\Ha$ as $g_{\a}=u_{\a}v_{\a}$.

\medskip

\noindent With these conventions, all the black beads lying over an edge $e$ of $D$ labeled by $\b$ represent an element of $\Hb$. Now follow the orientation of $e$ from its starting point and multiply the black beads encountered from right to left. This results in an element $x_e\in A_{\b}$. If an edge $e'$ follows $e$ in the orientation of $D$, then there is a white bead $\v_{\a^{\pm 1}}$ in between $e$ and $e'$, where $\a$ is the label of the overpass. We ``multiply" $x_{e'}$ with $x_e$ by taking $x_{e'}\v_{\a^{\pm 1}}(x_e)$ and we think of this as the new bead of $e'$. One can think that $x_e$ is being slided through the crossing and that when it passes through the white bead labeled with $\v_{\a^{\pm 1}}$ then it must be evaluated on it. If we do this for all edges of all components of $T$, this results in an element $$z_D\in A_{\c_1}\ot\dots\ot A_{\c_m}$$
where $\c_i$ is the label of the endpoint of $T_i$ for each $i=1,\dots,m.$ This procedure also makes sense for Hopf $G$-coalgebras in $\SVect$. Here one has to suppose the set of crossings is totally ordered (the order turns out to be irrelevant since the $R$-matrices have degree zero). Then, when following the orientation of $T$, the black beads are reordered and this introduces various signs into our formula.
  %Let $[T]\in\pi_1(X_T)$ be the homotopy class of a path that goes linearly from the basepoint to the input of $T$, then follows the framing of $T$ and finally goes back from the output to the basepoint by a linear path. Note that this path commutes with the canonical path at the input, so $\v_{\rho([T])}:\Hc\to\Hc$.

\def\col{\Gamma}

\def\ha{h_{\a}}\def\hc{h_{\c}}

\begin{example}

Consider the diagram of an open right trefoil $T$ as in Figure \ref{fig: open right trefoil with dots} and let $\rho:\pi_1(X_T)\to G$ be an homomorphism. Let $e_1,e_2,e_3$ be the edges of this diagram, starting from the second edge from top to bottom. We denote by $b_1,b_2,b_3\in G$ the associated labels via $\rho$, that is $b_i=\rho(\c_{e_i})$. From top to bottom, the pairs of black beads on the crossings represent the $R$-matrices $R_{b_2,b_1}, R_{b_1,b_3}, R_{b_3,b_2}$ respectively. If we denote $R_{\a,\b}=\sum s_{\a}\ot t_{\b}$ for each $\a,\b\in G$ then we get%Then if $i,j,k$ are the indices of the $R$-matrices ordered from top to bottom as we follow the orientation of $T$, we get 
\begin{align*}
z_D=\sum \v_{b_2}(t_{b_1}s_{b_1}\v_{b_3}(t_{b_2}g_{b_2}s_{b_2}\v_{b_1}(t_{b_3}s_{b_3})))
\end{align*}
If $\uH$ is a Hopf $G$-coalgebra in $\SVect$, each term in the above sum has to be multiplied by the sign $(-1)^{|s_{b_3}||t_{b_2}|+|t_{b_3}||t_{b_2}|+|s_{b_2}||t_{b_2}|+|s_{b_2}||s_{b_1}|+|s_{b_2}||t_{b_2}|}$.
\begin{figure}[h]
\centering

\begin{pspicture}(0,-2.203479)(2.72,2.203479)
\psline[linecolor=black, linewidth=0.026, arrowsize=0.05291667cm 2.0,arrowlength=1.4,arrowinset=0.0]{->}(0.06,1.4965209)(0.06,2.296521)
\psline[linecolor=black, linewidth=0.026, arrowsize=0.05291667cm 2.0,arrowlength=1.4,arrowinset=0.0](0.06,-1.5034791)(0.06,-2.203479)
\psbezier[linecolor=black, linewidth=0.026, arrowsize=0.05291667cm 2.0,arrowlength=1.4,arrowinset=0.0](0.06,1.4965209)(0.06,0.89652085)(1.06,1.0965209)(1.06,0.4965208435058594)
\psbezier[linecolor=white, linewidth=0.026, doubleline=true, doublesep=0.026, doublecolor=black](1.06,1.4965209)(1.06,0.89652085)(0.06,1.0965209)(0.06,0.4965208435058594)
\psbezier[linecolor=black, linewidth=0.026, arrowsize=0.05291667cm 2.0,arrowlength=1.4,arrowinset=0.0](1.06,1.4965209)(1.06,2.4965208)(2.66,2.4965208)(2.66,0.09652084350585938)
\psbezier[linecolor=black, linewidth=0.026, arrowsize=0.05291667cm 2.0,arrowlength=1.4,arrowinset=0.0](1.06,-1.5034791)(1.06,-2.5034792)(2.66,-2.5034792)(2.66,-0.10347915649414062)
\psline[linecolor=black, linewidth=0.026, arrowsize=0.05291667cm 2.0,arrowlength=1.4,arrowinset=0.0](2.66,0.09652084)(2.66,-0.103479154)
\psdots[linecolor=black, dotsize=0.14](0.34,0.89652085)
\psdots[linecolor=black, fillstyle=solid, dotstyle=o, dotsize=0.14, fillcolor=white](0.33,1.0965209)
\psdots[linecolor=black, dotsize=0.14](0.79,0.88652086)
\psbezier[linecolor=black, linewidth=0.026, arrowsize=0.05291667cm 2.0,arrowlength=1.4,arrowinset=0.0](0.06,0.49652085)(0.06,-0.103479154)(1.06,0.09652084)(1.06,-0.5034791564941407)
\psbezier[linecolor=white, linewidth=0.026, doubleline=true, doublesep=0.026, doublecolor=black](1.06,0.49652085)(1.06,-0.103479154)(0.06,0.09652084)(0.06,-0.5034791564941407)
\psdots[linecolor=black, dotsize=0.14](0.34,-0.103479154)
\psdots[linecolor=black, dotstyle=o, dotsize=0.14, fillcolor=white](0.33,0.09652084)
\psdots[linecolor=black, dotsize=0.14](0.79,-0.11347916)
\psbezier[linecolor=black, linewidth=0.026, arrowsize=0.05291667cm 2.0,arrowlength=1.4,arrowinset=0.0](0.06,-0.5034792)(0.06,-1.1034791)(1.06,-0.90347916)(1.06,-1.5034791564941405)
\psbezier[linecolor=white, linewidth=0.026, doubleline=true, doublesep=0.026, doublecolor=black](1.06,-0.5034792)(1.06,-1.1034791)(0.06,-0.90347916)(0.06,-1.5034791564941405)
\psdots[linecolor=black, dotsize=0.14](0.34,-1.1034791)
\psdots[linecolor=black, dotstyle=o, dotsize=0.14, fillcolor=white](0.33,-0.90347916)
\psdots[linecolor=black, dotsize=0.14](0.79,-1.1134791)
\psdots[linecolor=black, dotsize=0.14](1.25,1.9765209)
\rput[bl](1.26,0.49652085){$b_1$}
\rput[bl](2.46,1.6965208){$b_2$}
\rput[bl](1.26,-0.5034792){$b_3$}
\end{pspicture}
\caption{Beads on a diagram of a long $G$-knot whose closure is a right trefoil.}
\label{fig: open right trefoil with dots}
\end{figure}

\end{example}

\begin{proposition}
%The Reshetikhin-Turaev invariant $F(T,\rho):\Hc\to \Hc$ at the regular representation coincides with left multiplication by the element $z_D$ described above: $$F(T,\rho)(h)=z_D\cdot h$$for all $h\in \Hc$.
The element $z_D$ defined above is an isotopy invariant of the (framed, oriented) $G$-tangle $(T,\rho)$.% and satisfies that $\v_{\rho([T])}^{-1}(h_{\c})z_D=z_Dh_{\c}$ for each $\hc\in\Hc$.  %If $U\in \Rep(\Hc)$ is the source of a $\CC$-coloring $\col$ of $(T,\rho)$, then $F(T,\rho,\col):U\to U$ coincides with the action of $z_D$ on $U$. %If $K_V$ denotes the closure of $T$ colored by $V\in \Rep(\Hg)$, then the invariant of $K_V$ is $$F(K_V)=\tr_q^V(z_D)=\tr_V(g\cdot z_D).$$
\end{proposition}
\begin{proof}
This is because the usual Reidemeister moves between tangle diagrams can be extended to colored Reidemeister moves between colored diagrams of $G$-tangles and then it holds that two $G$-tangles are isotopic if and only if their colored diagrams are related by colored Reidemeister moves. Since we used upward oriented diagrams, the Reidemeister moves can be taken as in \cite[Section 3.2]{Ohtsuki:BOOK}. That $z_D$ is invariant under colored Reidemeister moves follows directly from the axioms of a ribbon Hopf $G$-coalgebra and therefore it is an isotopy invariant. For more details, see \cite{Turaev:BOOK-HQFT}. %The second assertion follows from property ?? of the graded $R$-matrix.
\end{proof}

We will denote this invariant by $$z_D=\rt_{\uH}^{\rho}(T).$$ We could also have started from the endpoint of a component of $T$, follow its opposite orientation and multiply the black beads from left to right, applying $\v_{\a}^{-1}$ each time we encounter a $\v_{\a}$. This would result in an element 
\begin{align}
\label{eq: z'D}
z'_D\in A_{\d_1}\ot\dots \ot A_{\d_m}
\end{align}
where $\d_i$ is the label of the starting point of $T_i$ for each $i=1,\dots,m$. These two elements are related by $\otimes_{i=1}^m\v_{[T_i]}(z'_D)=z_D$, where $[T_i]\in\pi_1(X_T)$ is the homotopy class of a path that goes from the basepoint linearly to the endpoint of $T_i$, then follows the framing of $T_i$ along the opposite orientation and finally goes back from the starting point of $T_i$ linearly to the basepoint.

%In the present case $z_D$ may not be an element of the center of $\Hg$. Indeed, if the source is an object $U_{\a}$ then the target is $\v_{\rho([T])}(U_{\a})$ where $[T]\in\pi_1(X_T)$ is the homotopy class of a path that goes from the basepoint linearly to the input, then follows the framing of $T$ and finally goes back from the output linearly to the basepoint. Note that this path commutes with the canonical path at the input, so $\v_{\rho([T])}:\Ha\to\Ha$. Therefore, the element $z_D$ satisfies $\v_{\rho([T])}^{-1}(h_{\a})z_D=z_Dh_{\a}$ for each $\ha\in\Ha$.

\subsection{Invariants of $G$-knots}\label{subs: inv of G-knots} %NOTE: SEE MORE LITERATURE (e.g. ADO), THERE SHOULD BE SOMETHING FOR LINKS (at least recover multivariable Alexander!). 
It is well-known that the isotopy class of the closure of a long knot determines that of the long knot itself. Therefore, any invariant of long knots defines a knot invariant. We now generalize this to the case of $G$-knots.
\medskip

Let $(T,\rho)$ be a long $G$-knot and $K$ the closure of $T$. We suppose $K\sb X=\R\t[-1,\infty)\t[0,1]$. It is clear that the $G$-label of the starting point of $T$ coincides with that of the endpoint (indeed, the paths $\c_{e_0},\c_{e_t}$ are homotopic, where $e_0$ (resp. $e_t$) is the first (resp. final) edge of a diagram of $T$ along its orientation). Thus, $\rho$ extends to an homomorphism $\pi_1(X_K,z)\to G$ that we still denote by $\rho$. We say that $(K,\rho)$ is the {\em $G$-closure} of $(T,\rho)$.

\medskip

In what follows, $(K,\rho)$ is a (framed, oriented) $G$-knot and $(T,\rho)$ is a long $G$-knot whose $G$-closure is isotopic to $(K,\rho)$. We denote by $\c$ the $G$-label of the endpoint of $T$.

\begin{lemma}
\label{lemma: closure of $G$-tangle}
Let $(T',\rho')$ be another long $G$-knot whose closure is isotopic to $(K,\rho)$. Then, there is an element $\b\in G$ in the image of $\rho'$ such that $(T,\rho)$ is isotopic to $(T',\rho'_{\b})$ as a long $G$-knot, where $\rho'_{\b}(\d)=\b\rho'(\d)\b^{-1}$ for all $\d\in\pi_1(X_{T'})$. 
\end{lemma}
\begin{proof}
Let $(K',\rho')$ be the closure $(T',\rho')$. For simplicity, let's suppose that the origin lies in the closure-strand of both $K$ and $K'$. Suppose there is an isotopy $d_t:X\to X$ such that $d_1(K)=K'$ and $\rho'\circ (d_1)_*=\rho$ over $\pi_1(X_K,z)$ where $z\in\pvX$ is the basepoint. After an isotopy of the identity of $S^3$ we can arrange for $d_1(0)=0$ (note that $(d_1)_*$ is unchanged). Then we can take $d_t-d_t(0)$ and modify it near $\p X$ so that it becomes an isotopy $f_t$ of $G$-knots such that $f_t(0)=0$ for all $t$. Thus, it is an isotopy from $T_0=K\sm K\cap B$ to $T'_0=K'\sm K'\cap B$ in $S^3\sm B$ where $B$ is a small open ball neighborhood of $0$. However, if we identify $S^3\sm B\cong \D^2\t[0,1]$, the basepoint $z$ becomes a point in the interior of $\D^2\t[0,1]$. If $z_0$ is point in the boundary of $\D^2\t[0,1]$, then we can identify $\pi_1(X_T,z)\cong\pi_1(X_T,z_0)$ only up to conjugation. Thus $(T_0,\rho)$ is equivalent to a conjugate of $(T,\rho)$, and similarly for $(T'_0,\rho')$, which proves the lemma.

 %Let $\rho_0,\rho'_0$ be the $G$-structures based at $z_0$. Then, for $\c\in\pi_1(X_T,z)$, we have \begin{align*}\rho(\c)=\rho_0(\zz \c \ov{\zz})=\rho_0'((f_1\zz) (f_1\c) \ov{(f_1\zz)})\end{align*}
\end{proof}

Therefore any invariant of $(T,\rho)$ that depends on $\rho$ only up to conjugation is an invariant of the closure $(K,\rho)$. The following are particular instances of this.

\begin{corollary}
\label{corollary: G-knot inv abelian case}
If the image of $\rho$ is abelian, then $\rt_{\uH}^{\rho}(T)\in \Hc$ is an invariant of $(K,\rho)$.
\end{corollary}
\begin{proof}
If $\rho$ is abelian, the above lemma implies that $(T,\rho)$ is $G$-isotopic to $(T',\rho')$ so any invariant of $(T,\rho)$ is an invariant of $(K,\rho)$.
\end{proof}

\begin{corollary}\label{corollary: G-knot invariant after ev}Let $\uHH$ be a ribbon Hopf $G$-coalgebra and let $\{f_{\a}:A_{\a}\to\kk\}_{\a\in G}$ be a family of linear functionals with the property that $f_{\b\a\b^{-1}}\circ\v_{\b}=f_{\a}$ for all $\a,\b\in G$. Then the scalar $$f_{\c}(\rt_{\uH}^{\rho}(T))\in\kk$$ is an invariant of $(K,\rho)$. \end{corollary}\begin{proof}By the above lemma, any other long $G$-knot whose closure is $(K,\rho)$ is isotopic to $(T,\rho_{\b})$ for some $\b\in G$. The invariant of $(T,\rho_{\b})$ is $\rt_{\uH}^{\rho_{\b}}(T)=\v_{\b}(\rt_{\uH}^{\rho}(T))\in A_{\b \c\b^{-1}}$. It follows that $$f_{\c}(\rt_{\uH}^{\rho}(T))=f_{\b\c\b^{-1}}(\rt_{\uH}^{\rho_{\b}}(T))$$ so this is an invariant of $(K,\rho).$\end{proof}

%We can strengthen the above corollary when $\rho$ is abelian. Note that in this case $\rho(\c)$ is independent of the edge of the diagram chosen.

%NEED G ABELIAN AND phi's TRIVIAL for z to be well-defined in \Hc.

%\begin{corollary}NOOO CORRECT. If $\rho$ is abelian, then $z(T,\rho)\in \Hc$ is an invariant of the $G$-closure of $(T,\rho)$.\end{corollary}\begin{proof}This follows from the fact that the element that realizes the conjugation in Lemma \ref{lemma: closure of $G$-tangle} lies in the image of $\rho$.\end{proof}

%Let $(T,\rho)$ be a $(1,1)$-$G$-tangle. Then the tangle complement $\R^2\t [0,1]\sm T$ retracts onto $\D^2\t [0,1]\sm T$ and this is homeomorphic to the knot complement $S^3\sm K$ where $K$ is the closure of $T$. Thus, $K$ becomes a $G$-knot via $\rho$, in other words, the closure operation of tangles is well-defined in the $G$-graded case. Now, in the ungraded case, isotopy classes of (1,1)-tangles are the same as isotopy classes of knots. In the $G$-graded case the correct assertion involves a conjugation:
\def\Zp{Z'}

Now, suppose we want to get invariants of unframed $G$-knots. Let $(T,\rho)$ be a framed long $G$-knot whose closure, after forgetting the framing, is an unframed $G$-knot $(K,\rho)$. Let $\rt_{\uH}^{\rho}(T)\in\Hc$ be the invariant of $(T,\rho)$, and let
\begin{align*}
\Zp_{\uH}(T,\rho)= \rt_{\uH}^{\rho}(T)\cdot v_{\c}^{-w(T)}=v_{\c}^{-w(T)}\cdot \v_{\c}^{-w(T)}( \rt_{\uH}^{\rho}(T))
\end{align*}
where the product is in $\Hc$ and $w(T)$ is the writhe of $T$. Then $\Zp_{\uH}(T,\rho)$ is an invariant of the underlying unframed long $G$-knot. Since the ribbon element of $\uH$ is $G$-invariant, it follows that $\v_{\b}(\Zp_{\uH}(T,\rho))=\Zp_{\uH}(T,\rho_{\b})$. Thus, if $\{f_{\a}\}_{\a\in G}$ is as in the above corollary, then 
\begin{align}
\label{eq: unframed inv general}
f_{\c}(\Zp_{\uH}(T,\rho))
\end{align}
\noindent is an invariant of the unframed $G$-knot $(K,\rho).$ 

\subsection{Invariants of $G$-knots from twisted Drinfeld doubles}\label{subs: invs from TDDs} Now let $H$ be a finite dimensional Hopf algebra and $G\sb \Aut(H)$ be a subgroup such that $\uDH|_G$ is $G$-ribbon. Let $\{v_{\a}\}_{\a\in G}$ be a ribbon structure determined by a triple $(\b,b,p)$ as in Proposition \ref{prop: ribbon of TDD}. From now on, we denote $p(\a)=\rH(\a)^{\frac{1}{2}}$. Note that since $D(H)_{\a}=H^*\ot H$ as vector spaces, the functional $\e_{D(H)}(f\ot h)=f(1)\e(h)$ is defined over $D(H)_{\a}$ for every $\a$ (though it is not an algebra morphism if $\a\neq \id_H$).
\medskip

 \begin{corollary}
 \label{cor: invariant PH}
Let $(K,\rho)$ be an (unframed) oriented $G$-knot. Then
 \begin{align*}
P_H^{\rho}(K)\eq\rH(\rho(m))^{\frac{1}{2}w(T)}\e_{D(H)}(\rt_{\uDH}^{\rho}(T))\in\kk
\end{align*}
is an invariant of $(K,\rho)$. Here $(T,\rho)$ is a framed, oriented long $G$-knot whose closure is $(K,\rho)$, $w(T)$ is the writhe of $T$ and $m$ is an oriented meridian of $K$.
 \end{corollary}
 \begin{proof}
 Setting $f_{\a}=\e_{D(H)}$ for every $\a$, one easily sees that $\{f_{\a}\}$ satisfies the hypothesis of Corollary \ref{corollary: G-knot invariant after ev}. By (\ref{eq: unframed inv general}), 
 \begin{align*}
 \e_{D(H)}(v_{\c}^{-w(T)}\cdot \v_{\c}^{-w(T)}(Z_{\uDH}^{\rho}(T)))
 \end{align*}
 is an invariant of $(K,\rho)$, where $\c$ is the label of the endpoint of $T$. We now show that this coincides with $P_H^{\rho}(K)$. Since $v_{\c}=\rH(\c)^{-\frac{1}{2}}\b\ot b\cdot u_{\c}^{-1}$, we only need to show that $\e_{D(H)}(\b\ot b\cdot x)=\e_{D(H)}(x)$ and $\e_{D(H)}(u_{\c}^{-1}x)=\e_{D(H)}(x)$ for all $x\in D(H)_{\c}$. Recall that $u_{\c}^{-1}=h^i\ot S^2(\c(h_i))$. If $x=f\ot h$, we can compute $\e_{D(H)}(u_{\c}^{-1}x)$ as follows:
 \begin{figure}[H]
 
 \psscalebox{1.0 1.0} % Change this value to rescale the drawing.
{
\begin{pspicture}(0,-1.66)(12.901257,1.66)
\psline[linecolor=black, linewidth=0.026, arrowsize=0.05291667cm 2.0,arrowlength=1.4,arrowinset=0.0]{->}(2.2900622,1.04)(2.5900621,1.04)
\rput[bl](2.7500622,0.94){$\Delta$}
\psline[linecolor=black, linewidth=0.026, arrowsize=0.05291667cm 2.0,arrowlength=1.4,arrowinset=0.0]{->}(3.1900623,0.84)(3.4900622,0.64)
\rput[bl](2.7300622,-1.01){$m$}
\psline[linecolor=black, linewidth=0.026, arrowsize=0.05291667cm 2.0,arrowlength=1.4,arrowinset=0.0]{->}(2.5900621,-0.96)(2.1900623,-0.96)
\rput[bl](3.6400623,0.39){$S^{-1}$}
\psline[linecolor=black, linewidth=0.026, arrowsize=0.05291667cm 2.0,arrowlength=1.4,arrowinset=0.0]{->}(3.4900622,-0.56)(3.1900623,-0.76)
\psbezier[linecolor=black, linewidth=0.026, arrowsize=0.05291667cm 2.0,arrowlength=1.4,arrowinset=0.0]{->}(3.1900623,1.24)(5.490062,1.94)(5.490062,-1.86)(3.1900623,-1.16)
\rput[bl](5.7900624,-0.31){$m$}
\rput[bl](5.7500625,0.24){$\Delta$}
\psline[linecolor=black, linewidth=0.026, arrowsize=0.05291667cm 2.0,arrowlength=1.4,arrowinset=0.0]{->}(6.690062,0.34)(6.2300625,0.34)
\psline[linecolor=black, linewidth=0.026, arrowsize=0.05291667cm 2.0,arrowlength=1.4,arrowinset=0.0]{->}(6.2900624,-0.26)(6.690062,-0.26)
\psbezier[linecolor=black, linewidth=0.026, arrowsize=0.05291667cm 2.0,arrowlength=1.4,arrowinset=0.0]{->}(5.690062,0.54)(5.2900624,1.14)(5.2900624,1.64)(1.2900623,1.64)
\psbezier[linecolor=black, linewidth=0.026, arrowsize=0.05291667cm 2.0,arrowlength=1.4,arrowinset=0.0]{<-}(5.690062,-0.46)(5.2650623,-1.0736364)(4.740062,-1.56)(2.1900623,-1.56)
\rput[bl](3.6400623,-0.41){${\gamma}^{-1}$}
\psline[linecolor=black, linewidth=0.026, arrowsize=0.05291667cm 2.0,arrowlength=1.4,arrowinset=0.0]{->}(3.8900623,0.24)(3.8900623,-0.06)
\psbezier[linecolor=black, linewidth=0.026, arrowsize=0.05291667cm 2.0,arrowlength=1.4,arrowinset=0.0]{->}(3.1900623,1.04)(4.7900624,1.04)(4.990062,-0.26)(5.590062,-0.26)
\psbezier[linecolor=black, linewidth=0.026, arrowsize=0.05291667cm 2.0,arrowlength=1.4,arrowinset=0.0]{<-}(3.1900623,-0.96)(4.7900624,-0.96)(4.990062,0.34)(5.590062,0.34)
\psline[linecolor=black, linewidth=0.026, arrowsize=0.05291667cm 2.0,arrowlength=1.4,arrowinset=0.0]{->}(1.3900622,1.04)(1.6900623,1.04)
\psbezier[linecolor=black, linewidth=0.026, arrowsize=0.05291667cm 2.0,arrowlength=1.4,arrowinset=0.0]{->}(1.3900622,1.64)(-0.20993775,1.64)(-0.20993775,1.04)(0.59006226,1.04)
\rput[bl](1.8900622,0.89){$\gamma$}
\rput[bl](6.940062,-0.36){$\epsilon$}
\rput[bl](6.940062,0.24){$1$}
\rput[bl](7.6400623,-0.01){=}
\rput[bl](1.7400622,-1.11){$f$}
\rput[bl](1.7900623,-1.66){$h$}
\rput[bl](0.8233956,0.94){$S^2$}
\psline[linecolor=black, linewidth=0.026, arrowsize=0.05291667cm 2.0,arrowlength=1.4,arrowinset=0.0]{->}(10.6900625,1.04)(10.990063,1.04)
\rput[bl](11.150063,0.94){$\Delta$}
\psline[linecolor=black, linewidth=0.026, arrowsize=0.05291667cm 2.0,arrowlength=1.4,arrowinset=0.0]{->}(11.590062,0.84)(11.890062,0.64)
\rput[bl](11.130062,-1.01){$m$}
\psline[linecolor=black, linewidth=0.026, arrowsize=0.05291667cm 2.0,arrowlength=1.4,arrowinset=0.0]{->}(10.990063,-0.96)(10.590062,-0.96)
\rput[bl](12.040062,0.39){$S^{-1}$}
\psline[linecolor=black, linewidth=0.026, arrowsize=0.05291667cm 2.0,arrowlength=1.4,arrowinset=0.0]{->}(11.890062,-0.56)(11.590062,-0.76)
\psbezier[linecolor=black, linewidth=0.026, arrowsize=0.05291667cm 2.0,arrowlength=1.4,arrowinset=0.0]{->}(11.590062,1.24)(13.290062,1.94)(13.290062,-1.86)(11.590062,-1.16)
\psline[linecolor=black, linewidth=0.026, arrowsize=0.05291667cm 2.0,arrowlength=1.4,arrowinset=0.0]{->}(10.590062,-1.56)(10.990063,-1.56)
\rput[bl](12.040062,-0.41){${\gamma}^{-1}$}
\psline[linecolor=black, linewidth=0.026, arrowsize=0.05291667cm 2.0,arrowlength=1.4,arrowinset=0.0]{->}(12.290062,0.24)(12.290062,-0.06)
\psline[linecolor=black, linewidth=0.026, arrowsize=0.05291667cm 2.0,arrowlength=1.4,arrowinset=0.0]{->}(9.790062,1.04)(10.090062,1.04)
\rput[bl](10.290062,0.89){$\gamma$}
\rput[bl](11.240063,-1.66){$\epsilon$}
\rput[bl](10.140062,-1.11){$f$}
\rput[bl](10.1900625,-1.66){$h$}
\rput[bl](9.223395,0.94){$S^2$}
\psline[linecolor=black, linewidth=0.026, arrowsize=0.05291667cm 2.0,arrowlength=1.4,arrowinset=0.0]{->}(8.6900625,1.04)(8.990063,1.04)
\rput[bl](8.283396,0.96){1}
\end{pspicture}
}

 \end{figure}

Using that $\c(S^2(1))=1$ it is easily seen the right hand side is $\e_{D(H)}(x)$. The proof that $\e_{D(H)}(\b\ot b\cdot x)=\e_{D(H)}(x)$ is similar (this requires that $b$ is $G$-invariant). 

 \end{proof}

\subsection{Lifting to polynomials}\label{subs: invs in Z-GRADED case} Suppose in addition to the above that $H$ is $\Z$-graded. %Then $H^*$ is $\Z$-graded with $(H^*)_i=(H_{-i})^*$ for each $i\in\Z$.  
We now show that the invariant $P_H^{\rho}$ of Corollary \ref{cor: invariant PH} lifts to a polynomial invariant of $K$. This follows the ideas of \cite{LN:twisted} (which itself mimics an idea from Reidemeister torsion theory), but now $H$ is a possibly non-involutory Hopf algebra.
\medskip

Let $H'=H\ot_{\kk}\kk[t^{\pm \frac{1}{2}}]$ where $t$ is a variable, then $H'$ is a $\Z$-graded Hopf algebra over $\kk[t^{\pm \frac{1}{2}}]$ and $H\sb H'$ as a $\kk$-linear Hopf subalgebra. We will identify $D(H')$ with $D(H)[t^{\pm \frac{1}{2}}]$. For any $\a\in \Aut(H)$ and $n\in\Z$ we can define a $\kk[t^{\pm 1}]$-linear Hopf automorphism $\a\ot n\in\Aut(H')$ by
\begin{align*}
(\a\ot n)(x)\eq t^{n|x|}\a(x)
\end{align*}
for any homogeneous element $x\in H$. This defines an homomorphism $$\theta:\Aut(H)\t\Z\to \Aut(H'), (\a,n)\mapsto  \a\ot n.$$ If $G\sb\Aut(H)$ is a subgroup, we let $G'$ be the image of $G\t\Z$ under this homomorphism. Note that, if $\La_l$ is a (nonzero) left cointegral of $H$ (and also of $H'$), then $$\a\ot n(\La_l)=t^{n|\La_l|}\rH(\a)\La_l$$
so that $r_{H'}(\a\ot n)=t^{n|\La_l|}\rH(\a)$. Thus, if $\rH$ has a square root over $G$, then $r_{H'}$ has a square root over $G'$, namely, $$r_{H'}(\a\ot n)^{\frac{1}{2}}=r_H(\a)^{\frac{1}{2}}\cdot t^{\frac{n|\La_l|}{2}}$$ so $\underline{D(H')}$ has an induced ribbon $G'$-coalgebra structure by Proposition \ref{prop: ribbon of TDD}.
\medskip

Now let $(K,\rho)$ be an oriented, unframed $G$-knot and let $(T,\rho)$ be a framed long $G$-knot with closure $(K,\rho)$ as unframed $G$-knots. Let $h:\pi_1(X_T)\to \Z$ be the map sending an oriented meridian of $T$ to 1. We will denote $\rhoc=\theta\circ (\rho\t h):\pi_1(X_T)\to G'$, that is,
\begin{align*}
\rhoc(\d)=\rho(\d)\ot h(\d)
\end{align*}
for $\d\in\pi_1(X_T)$. Since $\underline{D(H')}$ is $G'$-ribbon, we get an invariant $$\rt_{\underline{D(H')}}^{\rhoc}(T)\in D(H')=D(H)[t^{\pm \frac{1}{2}}]$$
\medskip
which we will denote simply as $\rt_{\underline{D(H)}}^{\rhoc}(T)$. By Corollary \ref{cor: invariant PH}
\begin{align}
\label{eq: PH polynomial}
P_H^{\rho}(K,t)\eq\rH(\rho(m))^{\frac{1}{2}w(T)}t^{\frac{|\La_l|w(T)}{2}}\e_{D(H')}(\rt_{\uDH}^{\rhoc}(T))\in \kk[t^{\pm 1}]
\end{align}
is a polynomial invariant of $(K,\rho)$, where $m$ is an oriented meridian of $K$. 
\medskip

In principle, the above belongs to $\kk[t^{\pm \frac{1}{2}}]$ but one can show that $P_H^{\rho}(K,t)\in \kk[t^{\pm 1}]$ as follows: first note that the fractional powers of $t$ come from the beads of $\rt_{\uDH}^{\rhoc}(T)$ at the caps and cups (because of the formula for the pivot of Proposition \ref{prop: ribbon of TDD}) and the normalization factor above. The caps/cups contribute a $t^{-\frac{r|\La_l|}{2}}$ where $r$ is the rotation number of the diagram, so the total fractional power is $t^{\frac{|\La_l|(w(T)-r)}{2}}$. But it is easy to see that for a long knot $w(T)-r$ is even: the parity of $w(T)-r$ is an invariant of framed tangles and it is unchanged under crossing changes, thus it coincides with the parity of the trivial tangle which is zero. 
\medskip

Note that for trivial $\rho$, $Z^h_{\uDH}(T)$ is an invariant of the knot $K$ (Corollary \ref{corollary: G-knot inv abelian case}). Thus, the whole universal invariant of $K$ can be ``deformed" to a polynomial invariant, provided $H$ is $\Z$-graded.
\medskip

\section{Twisted Kuperberg invariants via twisted Drinfeld doubles}\label{section: Fox calculus from TDD}

%\subsection{Bridge presentations of knots} In what follows, we will always assume the underarcs $a_i$ are in standard position and oriented upwards.

In this section we state and prove our main theorem (Theorem \ref{Thm: main theorem}). We begin by recalling the construction of a (sutured) Heegaard diagram from a bridge presentation of a knot. In Subsection \ref{subs: Tensors from HD's} we briefly recall the (dual of the) construction of \cite{LN:twisted}. Our main theorem is shown in Subsection \ref{subs: Main thm}, along with its corollary about twisted Reidemeister torsion.
\medskip

In all that follows, we let $H$ be a finite dimensional involutory Hopf algebra over a field $\kk$ (in $\Vect$ or $\SVect$) with a two-sided integral $\la$ and a two-sided cointegral $\La$ such that $\la(\La)=1$. We also let $K\sb S^3$ be a knot and $T$ a long knot whose closure is $K$.% and $\rho:\pi_1(X_T)\cong\pi_1(X_K)\to G\sb\Aut(H)$ be an homomorphism.

\subsection{From planar diagrams to Heegaard diagrams}\label{subs: HDs} 
\def\S{\Sigma}\def\aa{\boldsymbol{\a}}\def\bb{\boldsymbol{\b}}
Let $D$ be an oriented planar diagram of the (framed, oriented) long knot $T$. In all that follows we suppose $D$ is oriented from $\R^2\t\{0\}$ to $\R^2\t\{1\}$ and that at each crossing both strands are oriented upwards as in Subsection \ref{subs: invs of G-tangles}. For simplicity, we will suppose the last crossing of $D$ (when following its orientation) is an underpass. Suppose $D$ has $g$ crossings. Then $D$ determines a bridge presentation of the knot, where the underarcs correspond to the underpasses of the crossings and the overarcs correspond to the edges of the diagram. We will suppose the underarcs are in the plane $\R^2\t\{0\}$ and the overarcs lie above it. We will denote by $a_i$ (resp. $b_i$) the underarcs (resp. overarcs), numbered so that when following the opposite orientation of $D$ we encounter $a_0,b_1,a_1,b_2,\dots,a_{g-1},b_{g}$ respectively, see Figure \ref{fig: Kuperberg marked HD TREFOIL} (left). We denote by $b_{\pi(i)}$ the overarc above $a_i$. %After an isotopy, we may suppose all the underarcs $a_i$ are oriented upwards. 

\medskip
A bridge presentation determines a Heegaard diagram of the knot complement (as a sutured manifold with two sutures in its boundary). The diagram consists of the following: consider the sphere $\R^2\t \{0\}\cup\{\infty\}$ over which the underarcs lie. For each overarc $b_i$ with $1\leq i\leq g-1$ attach a 2-dimensional one-handle to $\R^2\t \{0\}$ along the feet of $b_i$. To do all this inside $\R^3$, we may think of this handle as ``following $b_i$" above the plane $\R^2\t\{0\}$. Now, delete two disks to the surface obtained, one at the endpoint of $b_g$ and one at the endpoint of $a_0$. Let $\S$ be the resulting oriented surface with boundary. For each $1\leq i\leq g-1$ let $\a_i\sb\S$ be the boundary curve of a disk neighborhood of $a_i$ that contains the feet of the 1-handles at the endpoints of $a_i$. We will assume $\a_i$ is oriented as the negative boundary of this disk. For each $1\leq i\leq g-1$ let $\b_i\sb\S$ be a circle which is the union of two arcs: the arc $b_i$ and an arc $b'_i$ which is parallel to $b_i$ but runs through the corresponding 1-handle. We will assume that $\b_i$ is oriented by extending the orientation of the arc $b_i$. Then $(\S,\aa,\bb)$ is a sutured Heegaard diagram of $S^3\sm K$, that is, the manifold obtained from $\S\t[0,1]$ by attaching 2-handles along $\a_i\t\{0\}$ and $\b_j\t \{1\}$ for each $i,j$ is homeomorphic to $S^3\sm K$, and the boundary $\p\S\t \{\frac{1}{2}\}$ correspond to two meridians in the boundary of $S^3\sm K$. We will also consider the arc $\b_g$ in $\S\sm\bb$ that joins the two punctures of $\S$ that is obtained by extending the (oriented) arc $b_g$ to the top puncture of the surface and we denote $\bb^{e}=\bb\cup\{\b_g\}$. We will also suppose that for each $i=1,\dots,g-1$, $\a_i$ and $\b_i$ have a basepoint lying right before (with respect to the orientation of $\a_i,\b_i$) the ``obvious" intersection point between $\a_i$ and $\b_i$ (which comes from the intersection $a_i\cap b_i$ at their endpoints). See Figure \ref{fig: Kuperberg marked HD TREFOIL} (right). Thus, $(\S,\aa,\bb^e)$ is an oriented, based, extended Heegaard diagram as in \cite{LN:twisted}.

\begin{figure}[h]
\centering
\psscalebox{0.8 0.8} % Change this value to rescale the drawing.
{
\begin{pspicture}(0,-4.3610287)(13.27,4.3610287)
\definecolor{colour0}{rgb}{0.0,0.8,0.2}
\definecolor{colour1}{rgb}{1.0,0.6,0.0}
\definecolor{colour2}{rgb}{0.2,0.0,1.0}
\rput{-20.0}(0.20675054,2.3349414){\pscircle[linecolor=black, linewidth=0.026, arrowsize=0.05291667cm 2.0,arrowlength=1.4,arrowinset=0.0, dimen=outer](6.7244306,0.58120036){0.4}}
\rput{-20.0}(0.69296545,3.1770902){\pscircle[linecolor=black, linewidth=0.026, arrowsize=0.05291667cm 2.0,arrowlength=1.4,arrowinset=0.0, dimen=outer](9.35557,-0.37645608){0.4}}
\rput{-20.0}(0.449858,2.7560158){\psellipse[linecolor=red, linewidth=0.026, dimen=outer](8.04,0.10237213)(2.2,0.8)}
\rput{-20.0}(-0.6140978,2.479679){\pscircle[linecolor=black, linewidth=0.026, arrowsize=0.05291667cm 2.0,arrowlength=1.4,arrowinset=0.0, dimen=outer](6.7244306,2.9812002){0.4}}
\rput[bl](6.6,0.50237215){$B$}
\rput{-20.0}(-0.11527847,3.308624){\pscircle[linecolor=black, linewidth=0.026, arrowsize=0.05291667cm 2.0,arrowlength=1.4,arrowinset=0.0, dimen=outer](9.32443,1.9812003){0.4}}
\rput[bl](9.2,1.9023721){$A$}
\rput[bl](5.84,-0.29762787){$\alpha_2$}
\rput{-20.0}(1.0275989,2.1902037){\pscircle[linecolor=black, linewidth=0.026, arrowsize=0.05291667cm 2.0,arrowlength=1.4,arrowinset=0.0, dimen=outer](6.7244306,-1.8187996){0.4}}
\rput{-20.0}(1.5138137,3.0323524){\pscircle[linecolor=black, linewidth=0.026, arrowsize=0.05291667cm 2.0,arrowlength=1.4,arrowinset=0.0, dimen=outer](9.35557,-2.776456){0.4}}
\rput{-20.0}(1.2707063,2.611278){\psellipse[linecolor=red, linewidth=0.026, dimen=outer](8.04,-2.297628)(2.2,0.8)}
\rput[bl](6.6,-1.8976278){$A$}
\rput[bl](5.84,-2.6776278){$\alpha_1$}
\psline[linecolor=red, linewidth=0.026, arrowsize=0.05291667cm 2.0,arrowlength=1.4,arrowinset=0.0]{->}(8.23,-1.5276278)(8.516666,-1.6420723)
\rput[bl](9.22,-2.8776278){$B$}
\psbezier[linecolor=colour0, linewidth=0.026](6.9,0.9023721)(7.154102,1.3516679)(7.6698694,2.1644862)(8.24,2.6223721313476562)(8.810131,3.0802581)(9.456574,3.1511166)(10.06,2.922372)(10.663425,2.6936276)(11.247312,2.1314623)(11.22,1.0823722)(11.192689,0.033281952)(10.398684,-1.3163257)(9.59,-2.4776278)
\psbezier[linecolor=colour1, linewidth=0.026](9.15,-0.69405645)(8.536302,-1.6125625)(8.295588,-2.2098827)(8.19,-2.723342154366619)(8.084413,-3.2368016)(8.077017,-4.062441)(9.05,-4.2476277)(10.022984,-4.4328146)(10.860962,-3.337895)(11.32,-2.612628)(11.779037,-1.8873607)(12.387896,-0.58616096)(12.54,0.8630864)(12.692104,2.3123338)(12.059729,3.3183901)(10.97,3.8723722)(9.880272,4.426354)(8.168719,4.638123)(7.02,3.210515)
\psline[linecolor=red, linewidth=0.026, arrowsize=0.05291667cm 2.0,arrowlength=1.4,arrowinset=0.0]{->}(8.23,0.87237215)(8.516666,0.7579277)
\psline[linecolor=blue, linewidth=0.026](6.96,-1.4976279)(9.04,1.7023721)
\rput[bl](8.94,0.9223721){$\beta_1$}
\rput[bl](10.59,0.9123721){$\beta_2$}
\rput[bl](12.96,0.9023721){$\beta_3$}
\psdots[linecolor=black, dotsize=0.2](6.71,1.1523721)
\psdots[linecolor=black, dotsize=0.2](6.91,0.9223721)
\psdots[linecolor=black, dotsize=0.2](6.95,-1.5176278)
\psdots[linecolor=black, dotsize=0.2](6.8,-1.2576278)
\psline[linecolor=colour2, linewidth=0.026, arrowsize=0.05291667cm 2.0,arrowlength=1.4,arrowinset=0.0]{->}(8.01,0.12237213)(8.2,0.38237214)
\psline[linecolor=colour0, linewidth=0.026, arrowsize=0.05291667cm 2.0,arrowlength=1.4,arrowinset=0.0]{->}(7.45,1.7523721)(7.63,1.9823722)
\psline[linecolor=colour1, linewidth=0.026, arrowsize=0.05291667cm 2.0,arrowlength=1.4,arrowinset=0.0]{->}(12.3,-0.37762788)(12.22,-0.6676279)
\psline[linecolor=black, linewidth=0.026, arrowsize=0.05291667cm 2.0,arrowlength=1.4,arrowinset=0.0](0.4,-3.297628)(2.2,-0.29762787)
\psbezier[linecolor=black, linewidth=0.026, arrowsize=0.05291667cm 2.0,arrowlength=1.4,arrowinset=0.0](2.2,-0.29762787)(2.4,0.10237213)(2.0666666,0.10237213)(2.0,0.10237213134765626)
\psline[linecolor=black, linewidth=0.026, arrowsize=0.05291667cm 2.0,arrowlength=1.4,arrowinset=0.0](0.4,-1.0976279)(0.4,-1.4976279)
\psline[linecolor=black, linewidth=0.026, arrowsize=0.05291667cm 2.0,arrowlength=1.4,arrowinset=0.0](0.4,-1.4976279)(1.0,-1.8976278)
\psline[linecolor=black, linewidth=0.026, arrowsize=0.05291667cm 2.0,arrowlength=1.4,arrowinset=0.0](1.4,-2.0976279)(2.0,-2.497628)
\psline[linecolor=black, linewidth=0.026, arrowsize=0.05291667cm 2.0,arrowlength=1.4,arrowinset=0.0](1.4,0.10237213)(2.0,-0.29762787)
\psline[linecolor=black, linewidth=0.026, arrowsize=0.05291667cm 2.0,arrowlength=1.4,arrowinset=0.0](2.0,0.10237213)(2.0,-0.29762787)
\psline[linecolor=black, linewidth=0.026, arrowsize=0.05291667cm 2.0,arrowlength=1.4,arrowinset=0.0](0.4,-1.0976279)(2.2,1.9023721)
\psbezier[linecolor=black, linewidth=0.026, arrowsize=0.05291667cm 2.0,arrowlength=1.4,arrowinset=0.0](2.2,1.9023721)(2.4,2.3023722)(2.0666666,2.3023722)(2.0,2.3023721313476564)
\psline[linecolor=black, linewidth=0.026, arrowsize=0.05291667cm 2.0,arrowlength=1.4,arrowinset=0.0](1.4,2.3023722)(2.0,1.9023721)
\psline[linecolor=black, linewidth=0.026, arrowsize=0.05291667cm 2.0,arrowlength=1.4,arrowinset=0.0](2.0,2.3023722)(2.0,1.9023721)
\psline[linecolor=black, linewidth=0.026, arrowsize=0.05291667cm 2.0,arrowlength=1.4,arrowinset=0.0](2.0,-2.497628)(2.0,-2.0976279)
\psline[linecolor=black, linewidth=0.026, arrowsize=0.05291667cm 2.0,arrowlength=1.4,arrowinset=0.0](0.4,1.1023722)(0.4,0.70237213)
\psline[linecolor=black, linewidth=0.026, arrowsize=0.05291667cm 2.0,arrowlength=1.4,arrowinset=0.0](0.4,0.70237213)(1.0,0.30237213)
\psbezier[linecolor=black, linewidth=0.026, arrowsize=0.05291667cm 2.0,arrowlength=1.4,arrowinset=0.0](0.4,1.1023722)(1.6,3.3023722)(2.6,3.702372)(3.2,3.3023721313476564)(3.8,2.9023721)(3.4,-0.09762787)(2.0,-2.0976279)
\psline[linecolor=black, linewidth=0.026, arrowsize=0.05291667cm 2.0,arrowlength=1.4,arrowinset=0.0]{<-}(0.0,3.1023722)(1.0,2.502372)
\rput[bl](0.0,2.502372){$a_0$}
\rput[bl](0.2,-2.0976279){$a_1$}
\rput[bl](2.2,1.1023722){$b_1$}
\rput[bl](3.8,2.702372){$b_2$}
\rput[bl](1.0,-3.0976279){$b_3$}
\rput[bl](0.2,0.10237213){$a_2$}
\end{pspicture}
}
\caption{A diagram of a long knot $T$ whose closure is a right trefoil, drawn as a bridge presentation. On the right, the sutured Heegaard diagram associated to this diagram. The $A$-labels and $B$-labels indicate where the one-handles have to be attached. The blue and green arcs indicate the portion of the circles $\b_1,\b_2$ that lie over $\R^2\t\{0\}$. The orange arc is the arc $b_3$ extended to the top puncture. The circles without letters are the boundary components of $\S$, so $\S$ is a genus two surface with two holes.}
\label{fig: Kuperberg marked HD TREFOIL}
\end{figure}

\subsection{Fox calculus} A Heegaard diagram $(\S,\{\a_i\}_{i=1}^{g-1},\{\b_i\}_{i=1}^{g-1})$ of $S^3\sm K$ together with an arc $\b_g\sb\S\sm (\cup_{i=1}^{g-1}\b_i)$ joining the two punctures of $\S$ induces a presentation of $\pi_1(S^3\sm K)$ that has generators $\b^*_i$ for each $1\leq i\leq g$ and relators $\ov{\a}_i$ for each $i=1,\dots,g-1$. This is because $S^3\sm K$ is homeomorphic to a handlebody of genus $g$ to which we attached $g-1$ 2-cells along the $\a$ curves. The $\b^*_i$'s are the duals of the $\b_1,\dots,\b_g$ in $\S$. Thus, the upper handlebody (corresponding to $(\S,\bb)$) is a thickening of the wedge of circles $\vee_{i=1}^{g}\b_i^*$. The relator $\ov{\a}_i$ is obtained by following $\a_i$ starting from a basepoint and multiplying, from right to left, the $(\b^*_i)^{\e}$'s corresponding to the intersection points of $\a_i$. Here $\e=+1$ if $\b_i\cdot \a=+1$ in $\S$ at the given intersection point and $\e=-1$ otherwise. 

\medskip
Let $F$ be the free group on $\b^*_1,\dots,\b^*_g$. Since each appearance of a generator $\b^*_j$ in the word $\ov{\a}_i$ corresponds to an intersection point between $\a_i$ and $\b_j$, the Fox derivative $\p\ov{\a_i}/\p \b^*_j$ is a sum of words in $F$ indexed by the intersection points $\a_i\cap \b_j$. Thus, we can write 
\begin{align}
\label{eq: wx's Fox derivative}
\frac{\p \ov{\a}_i}{\p \b^*_j}=\sum_{x\in \a_i\cap \b_j}m_x w_x
\end{align}
for some $w_x\in F$ and $m_x=1$ if the intersection at $x$ is positive (that is $\b_j\cdot \a_i=+1$ at this point) and $m_x=-1$ if negative.
\medskip

Now consider the Heegaard diagram coming from a planar projection as above. Then the $\b^*_i$'s can be identified with the loops $\c_{b_i}$'s encircling the arcs $b_i$ once with linking number $-1$. To avoid cumbersome notation, we will simply denote $b_i=\b^*_i$. With the above choice of orientations and basepoints on the curves, we have $$\ov{\a_i}=b_{\pi(i)}b_{i+1}b_{\pi(i)}^{-1}b_i^{-1} \ \text{ or } \ \ov{\a}_i=b_{\pi(i)}^{-1}b_{i+1}b_{\pi(i)}b_i^{-1}$$
depending on whether the crossing at $a_i$ is positive or negative. Thus, the presentation $\pi_1(S^3\sm K)=\lb b_1,\dots,b_g \ | \ \ov{\a}_1, \dots,\ov{\a}_{g-1}\rb$ is simply the Wirtinger presentation. From now on, we think of $\p w/\p b_i$ as an element of $\Z\lb \pi_1\rb$ for each $w\in F$ and $b_i$. %Since the appearances of $\b^*_i$ in $\ov{\a_j}$ correspond to the intersection points between the two curves (or arc-curve if $i=g$), we can write \begin{equation}\label{eq: wx's right Fox}\frac{\p\ov{\a}_j}{\p\b^*_i}=\sum_{x\in \b_i\cap\a_j}w_x\end{equation}where $w_x$ is a word in $F$ associated to the intersection point $x$. On the other hand, the arc $\a_0$ specifies an element $[\a_0]\in\pi_1$. This is given by $[\a_0]=w_0w_1\dots w_{g-1}$.
If the crossing at $a_i$ is positive, the nonzero Fox derivatives of $\ov{\a_i}$ are: 
\begin{align}
\label{eq: Fox of Wirtinger}
\frac{\p \ov{\a_i}}{\p b_{\pi(i)}}=1-b_{\pi(i)}b_{i+1}b_{\pi(i)}^{-1}, \  \frac{\p \ov{\a_i}}{\p b_{i+1}}=b_{\pi(i)}, \ \frac{\p \ov{\a_i}}{\p b_i}=-1.
\end{align}
Thus, if $x_1,x_2,x_3,x_4$ are the points of $\a_i\cap\bb^e$ as one follows the orientation of $\a_i$ starting from its basepoint, then (\ref{eq: Fox of Wirtinger}) shows that 
\begin{equation}
\label{eq: wx's positive case}
w_1=1, w_2=b_{\pi(i)}b_{i+1}b_{\pi(i)}^{-1}, w_3=b_{\pi(i)},w_4=1.
\end{equation}
 If the crossing at $a_i$ is negative, then 
\begin{equation}
\label{eq: wx's negative case}
w_1=1,w_2=b_{\pi(i)}^{-1}b_{i+1}, w_3=b_{\pi(i)}^{-1}, w_4=b_{\pi(i)}^{-1}.
\end{equation}

\subsection{Tensors from Heegaard diagrams}\label{subs: Tensors from HD's} Let $H$ be a Hopf algebra as above and let $\rho:\pi_1(S^3\sm K)\to\Aut(H)$ be an homomorphism. Given a sutured Heegaard diagram $(\S,\aa=\{\a_i\}_{i=1}^{g-1},\bb=\{\b_i\}_{i=1}^{g-1})$ of $S^3\sm K$ with an arc $\b_g$ and with orientations and basepoints as above, the construction of \cite{LN:twisted} assigns the following tensors of $H$ (this is actually the dual construction, see Remark \ref{remark: conventions on IH} below). First, to each curve $\a_i$ we associate the tensor 
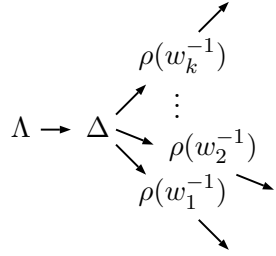
\begin{figure}[H]
\centering
\begin{pspicture}(0,-1.6072613)(3.436638,1.6072613)
\psline[linecolor=black, linewidth=0.026, arrowsize=0.05291667cm 2.0,arrowlength=1.4,arrowinset=0.0]{->}(0.4,-0.049999923)(0.8,-0.049999923)
\psline[linecolor=black, linewidth=0.026, arrowsize=0.05291667cm 2.0,arrowlength=1.4,arrowinset=0.0]{->}(1.4,0.15000008)(1.8,0.5500001)
\psline[linecolor=black, linewidth=0.026, arrowsize=0.05291667cm 2.0,arrowlength=1.4,arrowinset=0.0]{->}(1.4,-0.24999993)(1.8,-0.6499999)
\rput[bl](2.16,0.110000074){$\vdots$}
\psline[linecolor=black, linewidth=0.026, arrowsize=0.05291667cm 2.0,arrowlength=1.4,arrowinset=0.0]{->}(1.4,-0.049999923)(1.9,-0.24999993)
\rput[bl](0.0,-0.14999992){$\Lambda$}
\rput[bl](1.0,-0.14999992){$\Delta$}
\rput[bl](1.7,0.70000005){$\rho(w_k^{-1})$}
\rput[bl](2.1,-0.4999999){$\rho(w_2^{-1})$}
\rput[bl](1.7,-1.0999999){$\rho(w_1^{-1})$}
\psline[linecolor=black, linewidth=0.026, arrowsize=0.05291667cm 2.0,arrowlength=1.4,arrowinset=0.0]{->}(2.5,1.2500001)(2.9,1.6500001)
\psline[linecolor=black, linewidth=0.026, arrowsize=0.05291667cm 2.0,arrowlength=1.4,arrowinset=0.0]{->}(3.0,-0.6499999)(3.5,-0.8499999)
\psline[linecolor=black, linewidth=0.026, arrowsize=0.05291667cm 2.0,arrowlength=1.4,arrowinset=0.0]{->}(2.5,-1.2499999)(2.9,-1.65)
\end{pspicture}
\end{figure}
Here, from bottom to top, the outputs correspond to the intersection points, say $x_1,\dots,x_k$, of $\a_i$ with $\bb$ as one follows its orientation starting from its basepoint and we denote $w_i=w_{x_i}$ as in (\ref{eq: wx's Fox derivative}). Note that for the Heegaard diagram coming from a diagram projection, each $\ov{\a}_i$ intersects $\bb^e=\bb\cup \{\b_g\}$ in exactly four points, but intersects $\bb$ in either two, three or four points. Thus, the above tensor associated to $\ov{\a}_i$ may have two, three or four outputs depending on whether $|\a_i\cap\bb|=2,3$ or $4$. Now, to each closed curve $\b$ we associate the tensor
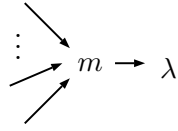
\begin{figure}[H]
\centering
\begin{pspicture}(0,-0.8091925)(2.2751222,0.8091925)
\psline[linecolor=black, linewidth=0.026, arrowsize=0.05291667cm 2.0,arrowlength=1.4,arrowinset=0.0]{->}(0.20512207,-0.8)(0.8051221,-0.2)
\psline[linecolor=black, linewidth=0.026, arrowsize=0.05291667cm 2.0,arrowlength=1.4,arrowinset=0.0]{->}(0.20512207,0.8)(0.8051221,0.2)
\rput[bl](0.06512207,0.06){$\vdots$}
\psline[linecolor=black, linewidth=0.026, arrowsize=0.05291667cm 2.0,arrowlength=1.4,arrowinset=0.0]{->}(1.405122,0.0)(1.805122,0.0)
\psline[linecolor=black, linewidth=0.026, arrowsize=0.05291667cm 2.0,arrowlength=1.4,arrowinset=0.0]{->}(0.00512207,-0.3)(0.70512205,0.0)
\rput[bl](0.90512204,-0.1){$m$}
\rput[bl](2.0051222,-0.2){$\lambda$}
\end{pspicture}
\end{figure}

\def\oo{\omega}
As before, from bottom to top, the inputs of this tensor correspond to the intersection points on $\b$ as one follows its orientation. Finally, to each intersection point $x$ we associate the tensor $S^{\e_x}$ where $\e_x\in\{0,1\}$ is defined by $m_x=(-1)^{\e_x}$ and $m_x$ is the sign in (\ref{eq: wx's Fox derivative}). In this way, each intersection point of the Heegaard diagram corresponds to a unique output of an $\a$-tensor and a unique input of a $\b$-tensor. The contraction of all these tensors is a scalar $$Z_H^{\rho}(\S,\aa,\bb^e)\in\kk.$$ Under our assumptions on $H$, it is shown in \cite{LN:twisted} that the scalar $Z_H^{\rho}(\S,\aa,\bb^e)$ is an invariant of the underlying sutured 3-manifold up to multiplication by an indeterminacy of the form $(\pm 1)^{|\La|}\rH(\rho(\d))$, where $\d\in\pi_1(M)$. We will denote by $I_H^{\rho}(M,\c)$ the invariant up to this indeterminacy. Note that if $\rho:\pi_1(M)\to\Ker(\rH)$ and $|\La|=0$ there is no indeterminacy at all. These indeterminacies can be removed using $\Spin^c$ structures and homology orientations, but we won't need this.
\medskip

%Under the above hypothesis for $H$, the double $D(H)$ is ribbon with $b=1,\b=\e$. However, the twisted Drinfeld double $\uDH$ is ribbon if and only if $\rH:\Aut(H)\to\kk$ has a square root. In this case, the construction of the present paper puts $\sqrt{\rH(\a)}^{\pm 1}$ on the caps and cups of a diagram which do not appear in the construction of \cite{LN:twisted}. The result is that $$\e_{D(H)}(rt_{\uDH}^{\rho}(K))=\sqrt{\rH(\a)}^{r}Z_H^{\rho}(\S,\aa,\bb)$$where $r$ is the rotation number of the diagram.

%           ADD TREFOIL EXAMPLE
\begin{remark}
\label{remark: conventions on IH}
The construction of \cite{LN:twisted} is a special case of \cite{LN:Kup} where we consider the semidirect product $\kk[\Aut(H)]\rtimes H$ as a graded algebra, where the grading group is $\Aut(H)$. As explained in \cite{LN:twisted}, if one sets $H_{\a}=\{\a\cdot h \ | \ h\in H\}\sb \kk[\Aut(H)]\rtimes H$, then $\{H_{\a}\}_{\a\in\Aut(H)}$ is a Hopf group-algebra. The dual of this object is a Hopf group-coalgebra whose representation category is equivalent to $\Rep(H^*)\rtimes\Aut(H^*)$. The above conventions for $I_H^{\rho}$ differ from those of \cite{LN:twisted} in that we use this dual object with $H$ in place of $H^*$. More precisely, the construction of \cite{LN:twisted} used the dual presentation having $\a^*_i$'s as generators and $\ov{\b_i}$'s as relators. The present version is obtained by taking the diagram $(-\S,\bb,\aa)$ and using $H^*$ tensors.  In other words, the present $I_H^{\rho}$ is the $I_{H^*}^{\rho^{-t}}$ of \cite{LN:twisted}, where for each $\a\in\Aut(H)$, $\a^t\in\Aut(H^*)$ is the dual of $\a$ and $\rho^{-t}:\pi_1(M)\to\Aut(H^*)$ is defined by $\rho^{-t}(\d)=(\rho(\d)^t)^{-1}$. 
\end{remark}

\subsection{Main theorem}\label{subs: Main thm} Let $H$ be as in the beginning of this section. By the two-sided cointegral/integral condition, the double $D(H)$ is ribbon with $b=1,\b=\e$. As in Proposition \ref{prop: ribbon of TDD}, we let $G\sb\Aut(H)$ be a subgroup over which $\sqrt{\rH}$ exists, so that $\uDH|_G$ is ribbon and we can define the invariant $P_H^{\rho}(K)$ of Subsection \ref{subs: invs from TDDs}. Then, the main theorem of the present paper is the following:

\begin{theorem}
\label{Thm: main theorem}
Let $(K,\rho)$ be a $G$-knot in $S^3$. Then the Reshetikhin-Turaev invariant of $K$ from the twisted Drinfeld double of $H$ recovers the twisted Kuperberg invariant as follows:
\begin{align*}
P^{\rho}_H(K)\dot{=} I_H^{\rho}(M,\c)
\end{align*}
where $(M,\c)$ is the sutured manifold associated to the complement of the knot $K$ ($M=S^3\sm K$ and $\c$ consists of two meridians in $\p(S^3\sm K)$). Here $\dot{=}$ denotes equality up to multiplication by $(\pm 1)^{|\La|}\rH(\rho(\d))$ for some $\d\in\pi_1(S^3\sm K).$ Similarly, if $H$ is $\Z$-graded then 
\begin{align*}
P_H^{\rho}(K,t)\dot{=} I_H^{\rhoc}(M,\c)
\end{align*}
where $\dot{=}$ denotes equality up to multiplication by $(\pm 1)^{|\La|} \rH(\rho(\d))t^{k|\La|}, k\in\Z, \d\in\pi_1(S^3\sm K)$.
\end{theorem}

\medskip
In the following proof we will use tensor network notation written from bottom to top instead of left to right. The reader may rotate the tensors below clockwise by 90 degrees to bring them into the form of Section \ref{section: Hopf algebras}. %Since we suppose the integral and the cointegral are two-sided, we denote them simply by $\la$ and $\La$. 
Recall that we are assuming that $H$ is involutory, so that $S^{-1}=S$.
\medskip

\begin{proof}

Let $D$ be a diagram of $T$ and take the associated bridge presentation $a_0,b_1,\dots,b_g$ of $K$ as in Subsection \ref{subs: HDs}.
%The idea of the proof consists essentially on bringing all the $H^*$ tensors back to $H$ and replacing traces with (co)integrals using Radford's formula. %To do this properly, for each underarc $a_i$ of the diagram, we will multiply all the elements of that arc and gather the product at the beginning of the arc. 
Recall that we suppose that all the underarcs $a_i$ are oriented upwards. Place the black/white beads on the diagram as specified in Subsection \ref{subs: invs of G-tangles}. For simplicity, we will suppose first that all crossings are positive, so the black beads at the crossings come from $R_{\a,\b}$'s (and not $R_{\a,\b}^{-1}$'s). Then each $a_i$ has one black and one white bead: one is an element of $H^*$, which we will denote by $A_i$, and the other is $\v_{b_{\pi(i)}}$ where $b_{\pi(i)}$ is the overarc at that crossing. Note that by (\ref{eq: twisted R-matrix}), $A_i$ has a $b_{\pi(i)}$. On the other hand, an overarc $b_i$ may have multiple black beads (but no white bead), which are all elements of $H\sb D(H)_{b_i}$. Since we suppose $b=1,\b=\e$, the right caps and cups have a $\sqrt{\rH(\rho(m))}^{-1}\cdot 1_H$ or $\sqrt{\rH(\rho(m))}\cdot 1_H$ bead respectively (Proposition \ref{prop: ribbon of TDD}), where $m$ is an oriented meridian. For simplicity, we will suppose there is no bead at all at the caps and cups, and we multiply the resulting tensor by $\rH(\rho(m))^{-\frac{r}{2}}$ at the end of the proof, where $r$ is the clockwise rotation number of the diagram. We begin by multiplying all the beads of a given overarc $b_i$, this results in a bead $B_i\in H\sb  D(H)_{b_i}$ (which could be the unit of $D(H)_{b_i}$ if $b_i$ has no crossing under it).
%The result in an element $A_i\in D(H)_{b_{i+1}}$ since $b_{i+1}$ is the label at the beginning of $a_i$. Similarly, we will multiply all the beads along $b_i$, this results in an element $B_i\in D(H)_{b_i}$. 
We will actually compute $z'_D=\v_{[T]}^{-1}(Z_{\uDH}^{\rho}(T))$ as in (\ref{eq: z'D}). Thus, we start from the top of $T$ and successively multiply the beads $A_0,B_1,A_1,\dots, A_{g-1},B_g$ in this order, taking care of the evaluations on the crossings $\v_{b_{\pi(i)}}:D_{\b}\to D_{b_{\pi(i)}\b(b_{\pi(i)})^{-1}}$ of $\uDH$. First we take the product $A_0\cdot B_1=A_0\ot B_1$ in $D(H)_{b_1}$ (this product is simply concatenation since $A_0\in H^*$ and $B_1\in H$). Before multiplying with $A_1$, we need to slide $A_0B_1$ through $a_1$, this has the effect of evaluating this product on the (inverse of the) crossing $\v_{b_{\pi(1)}}$ so we get a new bead $\v_{b_{\pi(1)}}^{-1}(A_0\ot B_1)$. We now multiply this with $A_1\ot B_2$ inside $D(H)_{b_2}$. Using the multiplication rule for $D(H)_{b_2}$ and that $\v_b^{-1}(p\ot h)=p\circ b\ot b^{-1}(h)$ for any $p\in H^*, h\in H, b\in G$ the product $\v_{b_{\pi}(1)}^{-1}(A_0\ot B_1)\cdot (A_1\ot B_2)$ can be written as:

%[linecolor=black, linewidth=0.026, arrowsize=0.05291667cm 2.0,arrowlength=1.4,arrowinset=0.0]
%[linecolor=black, linewidth=0.026, arrowsize=0.05291667cm 2.0,arrowlength=1.4,arrowinset=0.0]

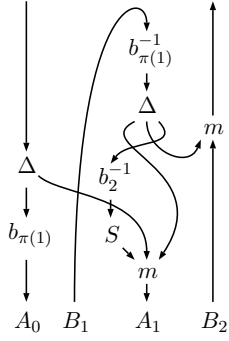
\begin{figure}[H]
\centering
\psscalebox{0.8 0.8} % Change this value to rescale the drawing.
{
\begin{pspicture}(0,-2.75)(3.58,2.75)
\rput[bl](0.1,-2.75){$A_0$}
\rput[bl](0.9,-2.75){$B_1$}
\rput[bl](2.1,-2.75){$A_1$}
\rput[bl](3.2,-2.75){$B_2$}
\psline[linecolor=black, linewidth=0.026, arrowsize=0.05291667cm 2.0,arrowlength=1.4,arrowinset=0.0]{->}(0.3,-1.55)(0.3,-2.25)
\psline[linecolor=black, linewidth=0.026, arrowsize=0.05291667cm 2.0,arrowlength=1.4,arrowinset=0.0]{->}(2.3,-1.95)(2.3,-2.25)
\psline[linecolor=black, linewidth=0.026, arrowsize=0.05291667cm 2.0,arrowlength=1.4,arrowinset=0.0]{<-}(3.4,0.45)(3.4,-2.25)
\rput[bl](1.6,-1.25){$S$}
\rput[bl](0.0,-1.35){$b_{\pi(1)}$}
\psbezier[linecolor=black, linewidth=0.026, arrowsize=0.05291667cm 2.0,arrowlength=1.4,arrowinset=0.0]{->}(1.1,-2.25)(1.1,2.95)(2.3,2.95)(2.3,2.25)
\psline[linecolor=black, linewidth=0.026, arrowsize=0.05291667cm 2.0,arrowlength=1.4,arrowinset=0.0]{->}(0.3,-0.35)(0.3,-0.85)
\rput[bl](2.15,-1.85){$m$}
\rput[bl](2.15,0.9){$\Delta$}
\psline[linecolor=black, linewidth=0.026, arrowsize=0.05291667cm 2.0,arrowlength=1.4,arrowinset=0.0]{->}(2.3,1.55)(2.3,1.25)
\rput[bl](3.25,0.55){$m$}
\psbezier[linecolor=black, linewidth=0.026, arrowsize=0.05291667cm 2.0,arrowlength=1.4,arrowinset=0.0]{->}(2.3,0.75)(2.3,0.25)(3.0,-0.05)(3.2,0.45)
\psline[linecolor=black, linewidth=0.026, arrowsize=0.05291667cm 2.0,arrowlength=1.4,arrowinset=0.0]{->}(0.3,2.75)(0.3,0.25)
\rput[bl](2.0,1.65){$b_{\pi(1)}^{-1}$}
\psline[linecolor=black, linewidth=0.026, arrowsize=0.05291667cm 2.0,arrowlength=1.4,arrowinset=0.0]{<-}(3.4,2.75)(3.4,0.85)
\psline[linecolor=black, linewidth=0.026, arrowsize=0.05291667cm 2.0,arrowlength=1.4,arrowinset=0.0]{->}(1.9,-1.35)(2.1,-1.55)
\rput[bl](1.5,-0.4){$b_2^{-1}$}
\psbezier[linecolor=black, linewidth=0.026, arrowsize=0.05291667cm 2.0,arrowlength=1.4,arrowinset=0.0]{->}(2.5,0.75)(2.9,0.3125)(2.0,0.45)(1.7,0.05)
\psbezier[linecolor=black, linewidth=0.026, arrowsize=0.05291667cm 2.0,arrowlength=1.4,arrowinset=0.0]{->}(2.1,0.75)(1.3,0.35)(3.5,-0.05)(2.5,-1.55)
\psbezier[linecolor=black, linewidth=0.026, arrowsize=0.05291667cm 2.0,arrowlength=1.4,arrowinset=0.0]{->}(0.5,-0.15)(0.81764704,-0.4730769)(2.3,-0.4730769)(2.3,-1.55)
\psline[linecolor=black, linewidth=0.026, arrowsize=0.05291667cm 2.0,arrowlength=1.4,arrowinset=0.0]{->}(1.7,-0.55)(1.7,-0.85)
\rput[bl](0.15,-0.1){$\Delta$}
\end{pspicture}
}

\end{figure}
As before, this is evaluated over $\v_{b_{\pi(2)}}^{-1}$ and then multiplied with $A_2\ot B_3$ inside $D(H)_{b_3}$. We keep doing this until we reach $A_{g-1}\ot B_g$. Thus $z'_D$ can be written as
\begin{figure}[H]
\centering
\psscalebox{0.8 0.8} % Change this value to rescale the drawing.
{
\begin{pspicture}(0,-4.6)(11.59,4.6)
\rput[bl](0.3,-4.6){$A_0$}
\rput[bl](1.1,-4.6){$B_1$}
\rput[bl](2.3,-4.6){$A_1$}
\rput[bl](3.4,-4.6){$B_2$}
\psline[linecolor=black, linewidth=0.026, arrowsize=0.05291667cm 2.0,arrowlength=1.4,arrowinset=0.0]{->}(0.5,-3.4)(0.5,-4.1)
\psline[linecolor=black, linewidth=0.026, arrowsize=0.05291667cm 2.0,arrowlength=1.4,arrowinset=0.0]{->}(2.5,-3.8)(2.5,-4.1)
\psline[linecolor=black, linewidth=0.026, arrowsize=0.05291667cm 2.0,arrowlength=1.4,arrowinset=0.0]{<-}(3.6,-1.4)(3.6,-4.1)
\rput[bl](1.8,-3.1){$S$}
\rput[bl](0.2,-3.2){$b_{\pi(1)}$}
\psbezier[linecolor=black, linewidth=0.026, arrowsize=0.05291667cm 2.0,arrowlength=1.4,arrowinset=0.0]{->}(1.3,-4.1)(1.3,1.1)(2.5,1.1)(2.5,0.4)
\psline[linecolor=black, linewidth=0.026, arrowsize=0.05291667cm 2.0,arrowlength=1.4,arrowinset=0.0]{->}(0.5,-2.2)(0.5,-2.7)
\rput[bl](2.35,-3.7){$m$}
\rput[bl](2.35,-0.95){$\Delta$}
\psline[linecolor=black, linewidth=0.026, arrowsize=0.05291667cm 2.0,arrowlength=1.4,arrowinset=0.0]{->}(2.5,-0.3)(2.5,-0.6)
\rput[bl](3.45,-1.2){$m$}
\psbezier[linecolor=black, linewidth=0.026, arrowsize=0.05291667cm 2.0,arrowlength=1.4,arrowinset=0.0]{->}(2.5,-1.1)(2.5,-1.6)(3.2,-1.9)(3.4,-1.4)
\psline[linecolor=black, linewidth=0.026, arrowsize=0.05291667cm 2.0,arrowlength=1.4,arrowinset=0.0]{->}(0.5,-1.0)(0.5,-1.5)
\rput[bl](2.2,-0.2){$b_{\pi(1)}^{-1}$}
\psline[linecolor=black, linewidth=0.026, arrowsize=0.05291667cm 2.0,arrowlength=1.4,arrowinset=0.0]{->}(2.1,-3.2)(2.3,-3.4)
\rput[bl](1.7,-2.25){$b_2^{-1}$}
\psbezier[linecolor=black, linewidth=0.026, arrowsize=0.05291667cm 2.0,arrowlength=1.4,arrowinset=0.0]{->}(2.7,-1.1)(3.1,-1.5375)(2.2,-1.4)(1.9,-1.8)
\psbezier[linecolor=black, linewidth=0.026, arrowsize=0.05291667cm 2.0,arrowlength=1.4,arrowinset=0.0]{->}(2.3,-1.1)(1.5,-1.5)(3.7,-1.9)(2.7,-3.4)
\psbezier[linecolor=black, linewidth=0.026, arrowsize=0.05291667cm 2.0,arrowlength=1.4,arrowinset=0.0]{->}(0.8,-2.1)(1.117647,-2.4230769)(2.5,-2.323077)(2.5,-3.4)
\psline[linecolor=black, linewidth=0.026, arrowsize=0.05291667cm 2.0,arrowlength=1.4,arrowinset=0.0]{->}(1.9,-2.4)(1.9,-2.7)
\rput[bl](0.35,-1.95){$\Delta$}
\rput[bl](5.2,-4.6){$A_2$}
\rput[bl](6.3,-4.6){$B_3$}
\psline[linecolor=black, linewidth=0.026, arrowsize=0.05291667cm 2.0,arrowlength=1.4,arrowinset=0.0]{->}(5.4,-3.8)(5.4,-4.1)
\psline[linecolor=black, linewidth=0.026, arrowsize=0.05291667cm 2.0,arrowlength=1.4,arrowinset=0.0]{<-}(6.5,-1.4)(6.5,-4.1)
\rput[bl](4.7,-3.1){$S$}
\psbezier[linecolor=black, linewidth=0.026, arrowsize=0.05291667cm 2.0,arrowlength=1.4,arrowinset=0.0]{->}(3.6,-0.9)(3.5,1.9)(5.4,1.1)(5.4,0.4)
\rput[bl](5.25,-3.7){$m$}
\rput[bl](5.25,-0.95){$\Delta$}
\psline[linecolor=black, linewidth=0.026, arrowsize=0.05291667cm 2.0,arrowlength=1.4,arrowinset=0.0]{->}(5.4,-0.3)(5.4,-0.6)
\rput[bl](6.35,-1.2){$m$}
\psbezier[linecolor=black, linewidth=0.026, arrowsize=0.05291667cm 2.0,arrowlength=1.4,arrowinset=0.0]{->}(5.4,-1.1)(5.4,-1.6)(6.1,-1.9)(6.3,-1.4)
\rput[bl](5.1,-0.2){$b_{\pi(2)}^{-1}$}
\psline[linecolor=black, linewidth=0.026, arrowsize=0.05291667cm 2.0,arrowlength=1.4,arrowinset=0.0]{->}(5.0,-3.2)(5.2,-3.4)
\rput[bl](4.6,-2.25){$b_3^{-1}$}
\psbezier[linecolor=black, linewidth=0.026, arrowsize=0.05291667cm 2.0,arrowlength=1.4,arrowinset=0.0]{->}(5.6,-1.1)(6.0,-1.5375)(5.1,-1.4)(4.8,-1.8)
\psbezier[linecolor=black, linewidth=0.026, arrowsize=0.05291667cm 2.0,arrowlength=1.4,arrowinset=0.0]{->}(5.2,-1.1)(4.5,-1.5)(6.6,-1.9)(5.6,-3.4)
\psline[linecolor=black, linewidth=0.026, arrowsize=0.05291667cm 2.0,arrowlength=1.4,arrowinset=0.0]{->}(4.8,-2.4)(4.8,-2.7)
\rput[bl](0.2,-0.8){$b_{\pi(2)}$}
\psline[linecolor=black, linewidth=0.026, arrowsize=0.05291667cm 2.0,arrowlength=1.4,arrowinset=0.0]{->}(0.5,0.2)(0.5,-0.3)
\psline[linecolor=black, linewidth=0.026, arrowsize=0.05291667cm 2.0,arrowlength=1.4,arrowinset=0.0]{->}(0.5,1.3)(0.5,0.9)
\rput[bl](0.35,0.45){$\Delta$}
\psbezier[linecolor=black, linewidth=0.026, arrowsize=0.05291667cm 2.0,arrowlength=1.4,arrowinset=0.0]{->}(0.8,0.3)(1.2,-0.2)(2.0,-0.3)(3.5,-0.6)(5.0,-0.9)(5.4,-1.9)(5.4,-3.4)
\psbezier[linecolor=black, linewidth=0.026, arrowsize=0.05291667cm 2.0,arrowlength=1.4,arrowinset=0.0]{->}(6.5,-0.9)(6.5,-0.8)(6.6,-0.5)(7.0,-0.3)
\rput[bl](10.0,-4.6){$A_{g-1}$}
\rput[bl](11.2,-4.6){$B_g$}
\psline[linecolor=black, linewidth=0.026, arrowsize=0.05291667cm 2.0,arrowlength=1.4,arrowinset=0.0]{->}(10.3,-3.8)(10.3,-4.1)
\psline[linecolor=black, linewidth=0.026, arrowsize=0.05291667cm 2.0,arrowlength=1.4,arrowinset=0.0]{<-}(11.4,-1.4)(11.4,-4.1)
\rput[bl](9.6,-3.1){$S$}
\rput[bl](10.15,-3.7){$m$}
\rput[bl](10.15,-0.95){$\Delta$}
\psline[linecolor=black, linewidth=0.026, arrowsize=0.05291667cm 2.0,arrowlength=1.4,arrowinset=0.0]{->}(10.3,-0.3)(10.3,-0.6)
\rput[bl](11.25,-1.2){$m$}
\psbezier[linecolor=black, linewidth=0.026, arrowsize=0.05291667cm 2.0,arrowlength=1.4,arrowinset=0.0]{->}(10.3,-1.1)(10.3,-1.6)(11.0,-1.9)(11.2,-1.4)
\rput[bl](9.9,-0.2){$b_{\pi(g-1)}^{-1}$}
\psline[linecolor=black, linewidth=0.026, arrowsize=0.05291667cm 2.0,arrowlength=1.4,arrowinset=0.0]{->}(9.9,-3.2)(10.1,-3.4)
\rput[bl](9.5,-2.25){$b_g^{-1}$}
\psbezier[linecolor=black, linewidth=0.026, arrowsize=0.05291667cm 2.0,arrowlength=1.4,arrowinset=0.0]{->}(10.5,-1.1)(10.9,-1.5375)(10.0,-1.4)(9.7,-1.8)
\psbezier[linecolor=black, linewidth=0.026, arrowsize=0.05291667cm 2.0,arrowlength=1.4,arrowinset=0.0]{->}(10.1,-1.1)(9.4,-1.5)(11.5,-1.9)(10.5,-3.4)
\psline[linecolor=black, linewidth=0.026, arrowsize=0.05291667cm 2.0,arrowlength=1.4,arrowinset=0.0]{->}(9.7,-2.4)(9.7,-2.7)
\psbezier[linecolor=black, linewidth=0.026, arrowsize=0.05291667cm 2.0,arrowlength=1.4,arrowinset=0.0]{->}(0.8,3.6)(2.2,2.3)(5.6,1.7)(7.5,0.6)(9.4,-0.5)(10.3,-1.3)(10.3,-3.4)
\psbezier[linecolor=black, linewidth=0.026, arrowsize=0.05291667cm 2.0,arrowlength=1.4,arrowinset=0.0]{->}(9.7,0.7)(9.9,0.8)(10.3,0.9)(10.3,0.4)
\psline[linecolor=black, linewidth=0.026, arrowsize=0.05291667cm 2.0,arrowlength=1.4,arrowinset=0.0]{<-}(11.4,2.5)(11.4,-0.8)
\psline[linecolor=black, linewidth=0.026, arrowsize=0.05291667cm 2.0,arrowlength=1.4,arrowinset=0.0]{->}(0.5,2.4)(0.5,2.0)
\psline[linecolor=black, linewidth=0.026, arrowsize=0.05291667cm 2.0,arrowlength=1.4,arrowinset=0.0]{->}(0.5,3.5)(0.5,3.1)
\psline[linecolor=black, linewidth=0.026, arrowsize=0.05291667cm 2.0,arrowlength=1.4,arrowinset=0.0]{->}(0.5,4.6)(0.5,4.2)
\rput[bl](0.5,1.5){$\vdots$}
\rput[bl](0.0,2.6){$b_{\pi(g-1)}$}
\rput[bl](8.0,-1.4){$\dots$}
\rput[bl](0.35,3.75){$\Delta$}
\end{pspicture}
}
\end{figure}

This product is an element of $D(H)_{b_g}=H^*\ot H$. We now evaluate this product on $\e_{D(H)}=1\ot\e$, note that $\e_{D(H)}(z'_D)=\e_{D(H)}(Z_{\uDH}^{\rho}(T))$. This has the effect of killing all the arrows coming from the leftmost $H^*$-tensor and also kills the rightmost arrow coming from $B_g$. Note that, since $A_0$ is supposed to connect to $B_{\pi(0)}$ in the above picture (because the $R$-matrix is a twisted coevaluation) when $A_0$ is killed, so does the input of $B_{\pi(0)}$ and the tensors immediately above it. Similarly, all the $A_i$ beads connected to $B_g$ disappear. The result is a tensor of the form:
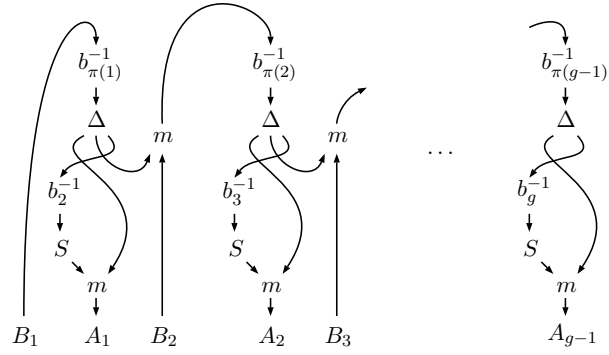
\begin{figure}[H]
\centering
\psscalebox{0.8 0.8} % Change this value to rescale the drawing.
{
\begin{pspicture}(0,-2.865869)(9.81,2.865869)
\rput[bl](0.0,-2.865869){$B_1$}
\rput[bl](1.2,-2.865869){$A_1$}
\rput[bl](2.3,-2.865869){$B_2$}
\psline[linecolor=black, linewidth=0.026, arrowsize=0.05291667cm 2.0,arrowlength=1.4,arrowinset=0.0]{->}(1.4,-2.065869)(1.4,-2.365869)
\psline[linecolor=black, linewidth=0.026, arrowsize=0.05291667cm 2.0,arrowlength=1.4,arrowinset=0.0]{<-}(2.5,0.33413085)(2.5,-2.365869)
\rput[bl](0.7,-1.3658692){$S$}
\psbezier[linecolor=black, linewidth=0.026, arrowsize=0.05291667cm 2.0,arrowlength=1.4,arrowinset=0.0]{->}(0.2,-2.365869)(0.2,2.8341308)(1.4,2.8341308)(1.4,2.134130859375)
\rput[bl](1.25,-1.9658692){$m$}
\rput[bl](1.25,0.7841309){$\Delta$}
\psline[linecolor=black, linewidth=0.026, arrowsize=0.05291667cm 2.0,arrowlength=1.4,arrowinset=0.0]{->}(1.4,1.4341309)(1.4,1.1341308)
\rput[bl](2.35,0.5341309){$m$}
\psbezier[linecolor=black, linewidth=0.026, arrowsize=0.05291667cm 2.0,arrowlength=1.4,arrowinset=0.0]{->}(1.4,0.63413084)(1.4,0.13413087)(2.1,-0.16586915)(2.3,0.334130859375)
\rput[bl](1.1,1.5341308){$b_{\pi(1)}^{-1}$}
\psline[linecolor=black, linewidth=0.026, arrowsize=0.05291667cm 2.0,arrowlength=1.4,arrowinset=0.0]{->}(1.0,-1.4658692)(1.2,-1.6658691)
\rput[bl](0.6,-0.51586914){$b_2^{-1}$}
\psbezier[linecolor=black, linewidth=0.026, arrowsize=0.05291667cm 2.0,arrowlength=1.4,arrowinset=0.0]{->}(1.6,0.63413084)(2.0,0.19663087)(1.1,0.33413085)(0.8,-0.065869140625)
\psbezier[linecolor=black, linewidth=0.026, arrowsize=0.05291667cm 2.0,arrowlength=1.4,arrowinset=0.0]{->}(1.2,0.63413084)(0.4,0.23413086)(2.6,-0.16586915)(1.6,-1.665869140625)
\psline[linecolor=black, linewidth=0.026, arrowsize=0.05291667cm 2.0,arrowlength=1.4,arrowinset=0.0]{->}(0.8,-0.6658691)(0.8,-0.9658691)
\rput[bl](4.1,-2.865869){$A_2$}
\rput[bl](5.2,-2.865869){$B_3$}
\psline[linecolor=black, linewidth=0.026, arrowsize=0.05291667cm 2.0,arrowlength=1.4,arrowinset=0.0]{->}(4.3,-2.065869)(4.3,-2.365869)
\psline[linecolor=black, linewidth=0.026, arrowsize=0.05291667cm 2.0,arrowlength=1.4,arrowinset=0.0]{<-}(5.4,0.33413085)(5.4,-2.365869)
\rput[bl](3.6,-1.3658692){$S$}
\psbezier[linecolor=black, linewidth=0.026, arrowsize=0.05291667cm 2.0,arrowlength=1.4,arrowinset=0.0]{->}(2.5,0.8341309)(2.4,3.634131)(4.3,2.8341308)(4.3,2.134130859375)
\rput[bl](4.15,-1.9658692){$m$}
\rput[bl](4.15,0.7841309){$\Delta$}
\psline[linecolor=black, linewidth=0.026, arrowsize=0.05291667cm 2.0,arrowlength=1.4,arrowinset=0.0]{->}(4.3,1.4341309)(4.3,1.1341308)
\rput[bl](5.25,0.5341309){$m$}
\psbezier[linecolor=black, linewidth=0.026, arrowsize=0.05291667cm 2.0,arrowlength=1.4,arrowinset=0.0]{->}(4.3,0.63413084)(4.3,0.13413087)(5.0,-0.16586915)(5.2,0.334130859375)
\rput[bl](4.0,1.5341308){$b_{\pi(2)}^{-1}$}
\psline[linecolor=black, linewidth=0.026, arrowsize=0.05291667cm 2.0,arrowlength=1.4,arrowinset=0.0]{->}(3.9,-1.4658692)(4.1,-1.6658691)
\rput[bl](3.5,-0.51586914){$b_3^{-1}$}
\psbezier[linecolor=black, linewidth=0.026, arrowsize=0.05291667cm 2.0,arrowlength=1.4,arrowinset=0.0]{->}(4.5,0.63413084)(4.9,0.19663087)(4.0,0.33413085)(3.7,-0.065869140625)
\psbezier[linecolor=black, linewidth=0.026, arrowsize=0.05291667cm 2.0,arrowlength=1.4,arrowinset=0.0]{->}(4.1,0.63413084)(3.4,0.23413086)(5.5,-0.16586915)(4.5,-1.665869140625)
\psline[linecolor=black, linewidth=0.026, arrowsize=0.05291667cm 2.0,arrowlength=1.4,arrowinset=0.0]{->}(3.7,-0.6658691)(3.7,-0.9658691)
\psbezier[linecolor=black, linewidth=0.026, arrowsize=0.05291667cm 2.0,arrowlength=1.4,arrowinset=0.0]{->}(5.4,0.8341309)(5.4,0.93413085)(5.5,1.2341309)(5.9,1.434130859375)
\rput[bl](8.9,-2.865869){$A_{g-1}$}
\psline[linecolor=black, linewidth=0.026, arrowsize=0.05291667cm 2.0,arrowlength=1.4,arrowinset=0.0]{->}(9.2,-2.065869)(9.2,-2.365869)
\rput[bl](8.5,-1.3658692){$S$}
\rput[bl](9.05,-1.9658692){$m$}
\rput[bl](9.05,0.7841309){$\Delta$}
\psline[linecolor=black, linewidth=0.026, arrowsize=0.05291667cm 2.0,arrowlength=1.4,arrowinset=0.0]{->}(9.2,1.4341309)(9.2,1.1341308)
\rput[bl](8.8,1.5341308){$b_{\pi(g-1)}^{-1}$}
\psline[linecolor=black, linewidth=0.026, arrowsize=0.05291667cm 2.0,arrowlength=1.4,arrowinset=0.0]{->}(8.8,-1.4658692)(9.0,-1.6658691)
\rput[bl](8.4,-0.51586914){$b_g^{-1}$}
\psbezier[linecolor=black, linewidth=0.026, arrowsize=0.05291667cm 2.0,arrowlength=1.4,arrowinset=0.0]{->}(9.4,0.63413084)(9.8,0.19663087)(8.9,0.33413085)(8.6,-0.065869140625)
\psbezier[linecolor=black, linewidth=0.026, arrowsize=0.05291667cm 2.0,arrowlength=1.4,arrowinset=0.0]{->}(9.0,0.63413084)(8.3,0.23413086)(10.4,-0.16586915)(9.4,-1.665869140625)
\psline[linecolor=black, linewidth=0.026, arrowsize=0.05291667cm 2.0,arrowlength=1.4,arrowinset=0.0]{->}(8.6,-0.6658691)(8.6,-0.9658691)
\psbezier[linecolor=black, linewidth=0.026, arrowsize=0.05291667cm 2.0,arrowlength=1.4,arrowinset=0.0]{->}(8.6,2.434131)(8.8,2.5341308)(9.2,2.634131)(9.2,2.134130859375)
\rput[bl](6.9,0.33413085){$\dots$}
\end{pspicture}
}
\end{figure}
To write this only using tensors in $H$ (and traces) we flip all the arrows oriented downwards. %Note that the $\Aut(H)$-action on $H^*$ is $b(p)=p\circ b^{-1}$ so that any automorphism on a $H^*$ tensor get inversed when we flip. 
We get the following tensor:

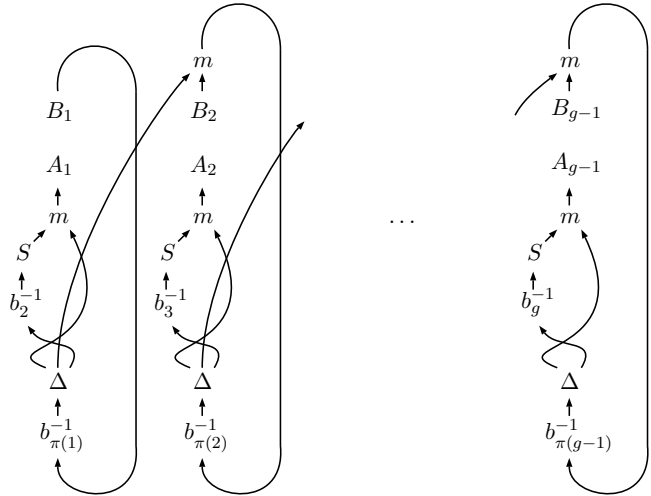
\begin{figure}[H]
\centering
\psscalebox{0.8 0.8} % Change this value to rescale the drawing.
{
\begin{pspicture}(0,-4.0912876)(10.645431,4.0912876)
\rput[bl](0.6,1.2212878){$A_1$}
\psline[linecolor=black, linewidth=0.026, arrowsize=0.05291667cm 2.0,arrowlength=1.4,arrowinset=0.0]{->}(0.8,0.72128785)(0.8,1.0212878)
\rput[bl](0.1,-0.17871216){$S$}
\rput[bl](0.65,0.42128783){$m$}
\rput[bl](0.65,-2.3287122){$\Delta$}
\psline[linecolor=black, linewidth=0.026, arrowsize=0.05291667cm 2.0,arrowlength=1.4,arrowinset=0.0]{->}(0.8,-2.7787123)(0.8,-2.478712)
\rput[bl](0.5,-3.3787122){$b_{\pi(1)}^{-1}$}
\psline[linecolor=black, linewidth=0.026, arrowsize=0.05291667cm 2.0,arrowlength=1.4,arrowinset=0.0]{->}(0.4,0.121287845)(0.6,0.32128784)
\rput[bl](0.0,-1.1287122){$b_2^{-1}$}
\psbezier[linecolor=black, linewidth=0.026, arrowsize=0.05291667cm 2.0,arrowlength=1.4,arrowinset=0.0]{->}(1.0,-1.9787122)(1.3,-1.4787122)(0.6,-1.6787121)(0.3,-1.278712158203125)
\psbezier[linecolor=black, linewidth=0.026, arrowsize=0.05291667cm 2.0,arrowlength=1.4,arrowinset=0.0]{->}(0.6,-1.9787122)(-0.4,-1.5787121)(2.0,-1.7787122)(1.0,0.321287841796875)
\psline[linecolor=black, linewidth=0.026, arrowsize=0.05291667cm 2.0,arrowlength=1.4,arrowinset=0.0]{->}(0.2,-0.6787121)(0.2,-0.37871215)
\rput[bl](0.6,2.1212878){$B_1$}
\psbezier[linecolor=black, linewidth=0.026, arrowsize=0.05291667cm 2.0,arrowlength=1.4,arrowinset=0.0](0.8,2.6212878)(0.7,3.6212878)(2.1,3.6212878)(2.1,2.621287841796875)
\psline[linecolor=black, linewidth=0.026, arrowsize=0.05291667cm 2.0,arrowlength=1.4,arrowinset=0.0](2.1,2.6212878)(2.1,-3.2787123)
\psbezier[linecolor=black, linewidth=0.026, arrowsize=0.05291667cm 2.0,arrowlength=1.4,arrowinset=0.0]{->}(2.1,-3.2787123)(2.2,-4.378712)(0.8,-4.1787124)(0.8,-3.478712158203125)
\rput[bl](3.0,1.2212878){$A_2$}
\psline[linecolor=black, linewidth=0.026, arrowsize=0.05291667cm 2.0,arrowlength=1.4,arrowinset=0.0]{->}(3.2,0.72128785)(3.2,1.0212878)
\rput[bl](2.5,-0.17871216){$S$}
\rput[bl](3.05,0.42128783){$m$}
\rput[bl](3.05,-2.3287122){$\Delta$}
\psline[linecolor=black, linewidth=0.026, arrowsize=0.05291667cm 2.0,arrowlength=1.4,arrowinset=0.0]{->}(3.2,-2.7787123)(3.2,-2.478712)
\rput[bl](2.9,-3.3787122){$b_{\pi(2)}^{-1}$}
\psline[linecolor=black, linewidth=0.026, arrowsize=0.05291667cm 2.0,arrowlength=1.4,arrowinset=0.0]{->}(2.8,0.121287845)(3.0,0.32128784)
\rput[bl](2.4,-1.1287122){$b_3^{-1}$}
\psbezier[linecolor=black, linewidth=0.026, arrowsize=0.05291667cm 2.0,arrowlength=1.4,arrowinset=0.0]{->}(3.4,-1.9787122)(3.7,-1.4787122)(3.0,-1.6787121)(2.7,-1.278712158203125)
\psbezier[linecolor=black, linewidth=0.026, arrowsize=0.05291667cm 2.0,arrowlength=1.4,arrowinset=0.0]{->}(3.0,-1.9787122)(2.0,-1.5787121)(4.4,-1.7787122)(3.4,0.321287841796875)
\psline[linecolor=black, linewidth=0.026, arrowsize=0.05291667cm 2.0,arrowlength=1.4,arrowinset=0.0]{->}(2.6,-0.6787121)(2.6,-0.37871215)
\rput[bl](3.0,2.1212878){$B_2$}
\psbezier[linecolor=black, linewidth=0.026, arrowsize=0.05291667cm 2.0,arrowlength=1.4,arrowinset=0.0](3.2,3.3212879)(3.1,4.3212876)(4.5,4.3212876)(4.5,3.321287841796875)
\psline[linecolor=black, linewidth=0.026, arrowsize=0.05291667cm 2.0,arrowlength=1.4,arrowinset=0.0](4.5,3.3212879)(4.5,-3.2787123)
\psbezier[linecolor=black, linewidth=0.026, arrowsize=0.05291667cm 2.0,arrowlength=1.4,arrowinset=0.0]{->}(4.5,-3.2787123)(4.6,-4.378712)(3.2,-4.1787124)(3.2,-3.478712158203125)
\rput[bl](3.05,3.021288){$m$}
\psline[linecolor=black, linewidth=0.026, arrowsize=0.05291667cm 2.0,arrowlength=1.4,arrowinset=0.0]{->}(3.2,2.6212878)(3.2,2.9212878)
\psbezier[linecolor=black, linewidth=0.026, arrowsize=0.05291667cm 2.0,arrowlength=1.4,arrowinset=0.0]{->}(0.8,-1.9787122)(0.8,0.021287842)(2.6,2.521288)(3.0,2.921287841796875)
\rput[bl](9.0,1.2212878){$A_{g-1}$}
\psline[linecolor=black, linewidth=0.026, arrowsize=0.05291667cm 2.0,arrowlength=1.4,arrowinset=0.0]{->}(9.3,0.72128785)(9.3,1.0212878)
\rput[bl](8.6,-0.17871216){$S$}
\rput[bl](9.15,0.42128783){$m$}
\rput[bl](9.15,-2.3287122){$\Delta$}
\psline[linecolor=black, linewidth=0.026, arrowsize=0.05291667cm 2.0,arrowlength=1.4,arrowinset=0.0]{->}(9.3,-2.7787123)(9.3,-2.478712)
\rput[bl](8.9,-3.3787122){$b_{\pi(g-1)}^{-1}$}
\psline[linecolor=black, linewidth=0.026, arrowsize=0.05291667cm 2.0,arrowlength=1.4,arrowinset=0.0]{->}(8.9,0.121287845)(9.1,0.32128784)
\rput[bl](8.5,-1.1287122){$b_g^{-1}$}
\psbezier[linecolor=black, linewidth=0.026, arrowsize=0.05291667cm 2.0,arrowlength=1.4,arrowinset=0.0]{->}(9.5,-1.9787122)(9.8,-1.4787122)(9.1,-1.6787121)(8.8,-1.278712158203125)
\psbezier[linecolor=black, linewidth=0.026, arrowsize=0.05291667cm 2.0,arrowlength=1.4,arrowinset=0.0]{->}(9.1,-1.9787122)(8.1,-1.5787121)(10.5,-1.7787122)(9.5,0.321287841796875)
\psline[linecolor=black, linewidth=0.026, arrowsize=0.05291667cm 2.0,arrowlength=1.4,arrowinset=0.0]{->}(8.7,-0.6787121)(8.7,-0.37871215)
\rput[bl](9.0,2.1212878){$B_{g-1}$}
\psbezier[linecolor=black, linewidth=0.026, arrowsize=0.05291667cm 2.0,arrowlength=1.4,arrowinset=0.0](9.3,3.3212879)(9.2,4.3212876)(10.6,4.3212876)(10.6,3.321287841796875)
\psline[linecolor=black, linewidth=0.026, arrowsize=0.05291667cm 2.0,arrowlength=1.4,arrowinset=0.0](10.6,3.3212879)(10.6,-3.2787123)
\psbezier[linecolor=black, linewidth=0.026, arrowsize=0.05291667cm 2.0,arrowlength=1.4,arrowinset=0.0]{->}(10.6,-3.2787123)(10.7,-4.378712)(9.3,-4.1787124)(9.3,-3.478712158203125)
\rput[bl](9.15,3.021288){$m$}
\psline[linecolor=black, linewidth=0.026, arrowsize=0.05291667cm 2.0,arrowlength=1.4,arrowinset=0.0]{->}(9.3,2.6212878)(9.3,2.9212878)
\psbezier[linecolor=black, linewidth=0.026, arrowsize=0.05291667cm 2.0,arrowlength=1.4,arrowinset=0.0]{->}(8.4,2.2212877)(8.5,2.4212878)(8.7,2.6212878)(9.1,2.921287841796875)
\psbezier[linecolor=black, linewidth=0.026, arrowsize=0.05291667cm 2.0,arrowlength=1.4,arrowinset=0.0]{->}(3.2,-1.9787122)(3.2,-0.30524278)(4.590909,1.7865939)(4.9,2.121287841796875)
\rput[bl](6.3,0.42128783){$\dots$}
\end{pspicture}
}
\end{figure}

Note that there are various traces in the above picture. Using that $A_i$ has the form $h^*\circ b_{\pi(i)}$ by (\ref{eq: twisted R-matrix}), replacing each of the above traces by a cointegral-integral pair using Radford's trace formula (see Subsection \ref{subs: Hopf integrals}) and using that $\la$ is a trace (which follows from $S^2=\id_H$ and $\La$ being 2-sided, see e.g. \cite[Lemma 3.9]{Kup2}), we see that the above tensor is the same as

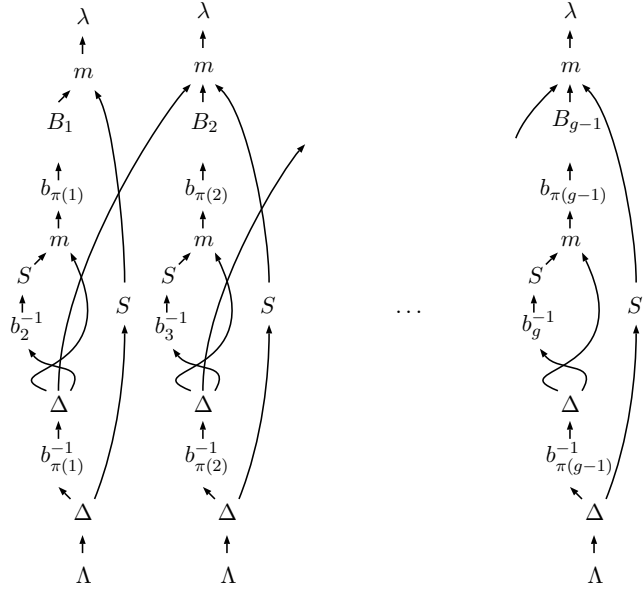
\begin{figure}[H]
\centering
\psscalebox{0.8 0.8} % Change this value to rescale the drawing.
{
\begin{pspicture}(0,-4.82)(10.46,4.82)
\rput[bl](0.5,1.48){$b_{\pi(1)}$}
\psline[linecolor=black, linewidth=0.026, arrowsize=0.05291667cm 2.0,arrowlength=1.4,arrowinset=0.0]{->}(0.8,1.08)(0.8,1.38)
\rput[bl](0.1,0.18){$S$}
\rput[bl](0.65,0.78){$m$}
\rput[bl](0.65,-1.97){$\Delta$}
\psline[linecolor=black, linewidth=0.026, arrowsize=0.05291667cm 2.0,arrowlength=1.4,arrowinset=0.0]{->}(0.8,-2.42)(0.8,-2.12)
\rput[bl](0.5,-3.02){$b_{\pi(1)}^{-1}$}
\psline[linecolor=black, linewidth=0.026, arrowsize=0.05291667cm 2.0,arrowlength=1.4,arrowinset=0.0]{->}(0.4,0.48)(0.6,0.68)
\rput[bl](0.0,-0.77){$b_2^{-1}$}
\psbezier[linecolor=black, linewidth=0.026, arrowsize=0.05291667cm 2.0,arrowlength=1.4,arrowinset=0.0]{->}(1.0,-1.62)(1.3,-1.12)(0.6,-1.32)(0.3,-0.92)
\psbezier[linecolor=black, linewidth=0.026, arrowsize=0.05291667cm 2.0,arrowlength=1.4,arrowinset=0.0]{->}(0.6,-1.62)(-0.4,-1.22)(2.0,-1.42)(1.0,0.68)
\psline[linecolor=black, linewidth=0.026, arrowsize=0.05291667cm 2.0,arrowlength=1.4,arrowinset=0.0]{->}(0.2,-0.32)(0.2,-0.02)
\rput[bl](0.6,2.68){$B_1$}
\psline[linecolor=black, linewidth=0.026, arrowsize=0.05291667cm 2.0,arrowlength=1.4,arrowinset=0.0]{->}(3.2,1.08)(3.2,1.38)
\rput[bl](2.5,0.18){$S$}
\rput[bl](3.05,0.78){$m$}
\rput[bl](3.05,-1.97){$\Delta$}
\psline[linecolor=black, linewidth=0.026, arrowsize=0.05291667cm 2.0,arrowlength=1.4,arrowinset=0.0]{->}(3.2,-2.42)(3.2,-2.12)
\rput[bl](2.9,-3.02){$b_{\pi(2)}^{-1}$}
\psline[linecolor=black, linewidth=0.026, arrowsize=0.05291667cm 2.0,arrowlength=1.4,arrowinset=0.0]{->}(2.8,0.48)(3.0,0.68)
\rput[bl](2.4,-0.77){$b_3^{-1}$}
\psbezier[linecolor=black, linewidth=0.026, arrowsize=0.05291667cm 2.0,arrowlength=1.4,arrowinset=0.0]{->}(3.4,-1.62)(3.7,-1.12)(3.0,-1.32)(2.7,-0.92)
\psbezier[linecolor=black, linewidth=0.026, arrowsize=0.05291667cm 2.0,arrowlength=1.4,arrowinset=0.0]{->}(3.0,-1.62)(2.0,-1.22)(4.4,-1.42)(3.4,0.68)
\psline[linecolor=black, linewidth=0.026, arrowsize=0.05291667cm 2.0,arrowlength=1.4,arrowinset=0.0]{->}(2.6,-0.32)(2.6,-0.02)
\rput[bl](3.0,2.68){$B_2$}
\rput[bl](3.05,3.68){$m$}
\psline[linecolor=black, linewidth=0.026, arrowsize=0.05291667cm 2.0,arrowlength=1.4,arrowinset=0.0]{->}(3.2,3.18)(3.2,3.48)
\psbezier[linecolor=black, linewidth=0.026, arrowsize=0.05291667cm 2.0,arrowlength=1.4,arrowinset=0.0]{->}(0.8,-1.62)(0.8,0.46163264)(2.6,3.0636735)(3.0,3.48)
\psline[linecolor=black, linewidth=0.026, arrowsize=0.05291667cm 2.0,arrowlength=1.4,arrowinset=0.0]{->}(9.3,1.08)(9.3,1.38)
\rput[bl](8.6,0.18){$S$}
\rput[bl](9.15,0.78){$m$}
\rput[bl](9.15,-1.97){$\Delta$}
\psline[linecolor=black, linewidth=0.026, arrowsize=0.05291667cm 2.0,arrowlength=1.4,arrowinset=0.0]{->}(9.3,-2.42)(9.3,-2.12)
\rput[bl](8.9,-3.02){$b_{\pi(g-1)}^{-1}$}
\psline[linecolor=black, linewidth=0.026, arrowsize=0.05291667cm 2.0,arrowlength=1.4,arrowinset=0.0]{->}(8.9,0.48)(9.1,0.68)
\rput[bl](8.5,-0.77){$b_g^{-1}$}
\psbezier[linecolor=black, linewidth=0.026, arrowsize=0.05291667cm 2.0,arrowlength=1.4,arrowinset=0.0]{->}(9.5,-1.62)(9.8,-1.12)(9.1,-1.32)(8.8,-0.92)
\psbezier[linecolor=black, linewidth=0.026, arrowsize=0.05291667cm 2.0,arrowlength=1.4,arrowinset=0.0]{->}(9.1,-1.62)(8.1,-1.22)(10.5,-1.42)(9.5,0.68)
\psline[linecolor=black, linewidth=0.026, arrowsize=0.05291667cm 2.0,arrowlength=1.4,arrowinset=0.0]{->}(8.7,-0.32)(8.7,-0.02)
\rput[bl](9.0,2.68){$B_{g-1}$}
\rput[bl](9.15,3.68){$m$}
\psline[linecolor=black, linewidth=0.026, arrowsize=0.05291667cm 2.0,arrowlength=1.4,arrowinset=0.0]{->}(9.3,3.18)(9.3,3.48)
\psbezier[linecolor=black, linewidth=0.026, arrowsize=0.05291667cm 2.0,arrowlength=1.4,arrowinset=0.0]{->}(8.4,2.58)(8.5,2.837143)(8.7,3.0942857)(9.1,3.48)
\psbezier[linecolor=black, linewidth=0.026, arrowsize=0.05291667cm 2.0,arrowlength=1.4,arrowinset=0.0]{->}(3.2,-1.62)(3.2,0.053469386)(4.590909,2.145306)(4.9,2.48)
\rput[bl](6.4,-0.32){$\dots$}
\rput[bl](1.05,-3.77){$\Delta$}
\psline[linecolor=black, linewidth=0.026, arrowsize=0.05291667cm 2.0,arrowlength=1.4,arrowinset=0.0]{->}(1.0,-3.42)(0.8,-3.22)
\psline[linecolor=black, linewidth=0.026, arrowsize=0.05291667cm 2.0,arrowlength=1.4,arrowinset=0.0]{->}(1.2,-4.32)(1.2,-4.02)
\psline[linecolor=black, linewidth=0.026, arrowsize=0.05291667cm 2.0,arrowlength=1.4,arrowinset=0.0]{->}(0.8,3.18)(1.0,3.38)
\rput[bl](1.05,3.58){$m$}
\psbezier[linecolor=black, linewidth=0.026, arrowsize=0.05291667cm 2.0,arrowlength=1.4,arrowinset=0.0]{->}(1.4,-3.42)(1.8,-2.42)(1.9,-1.32)(1.9,-0.52)
\rput[bl](1.75,-0.32){$S$}
\psbezier[linecolor=black, linewidth=0.026, arrowsize=0.05291667cm 2.0,arrowlength=1.4,arrowinset=0.0]{->}(1.9,0.18)(1.9,1.1090323)(1.7,2.8638709)(1.4,3.38)
\psline[linecolor=black, linewidth=0.026, arrowsize=0.05291667cm 2.0,arrowlength=1.4,arrowinset=0.0]{->}(1.2,3.98)(1.2,4.28)
\rput[bl](1.1,4.48){$\lambda$}
\rput[bl](1.09,-4.82){$\Lambda$}
\psline[linecolor=black, linewidth=0.026, arrowsize=0.05291667cm 2.0,arrowlength=1.4,arrowinset=0.0]{->}(3.2,4.08)(3.2,4.38)
\rput[bl](3.1,4.58){$\lambda$}
\rput[bl](3.45,-3.77){$\Delta$}
\psline[linecolor=black, linewidth=0.026, arrowsize=0.05291667cm 2.0,arrowlength=1.4,arrowinset=0.0]{->}(3.4,-3.42)(3.2,-3.22)
\psline[linecolor=black, linewidth=0.026, arrowsize=0.05291667cm 2.0,arrowlength=1.4,arrowinset=0.0]{->}(3.6,-4.32)(3.6,-4.02)
\psbezier[linecolor=black, linewidth=0.026, arrowsize=0.05291667cm 2.0,arrowlength=1.4,arrowinset=0.0]{->}(3.8,-3.42)(4.2,-2.42)(4.3,-1.32)(4.3,-0.52)
\rput[bl](3.49,-4.82){$\Lambda$}
\rput[bl](4.15,-0.32){$S$}
\psbezier[linecolor=black, linewidth=0.026, arrowsize=0.05291667cm 2.0,arrowlength=1.4,arrowinset=0.0]{->}(4.3,0.18)(4.3,1.08)(3.94,2.98)(3.4,3.48)
\psline[linecolor=black, linewidth=0.026, arrowsize=0.05291667cm 2.0,arrowlength=1.4,arrowinset=0.0]{->}(9.3,4.08)(9.3,4.38)
\rput[bl](9.2,4.58){$\lambda$}
\psbezier[linecolor=black, linewidth=0.026, arrowsize=0.05291667cm 2.0,arrowlength=1.4,arrowinset=0.0]{->}(10.4,0.18)(10.4,1.1380645)(10.04,2.947742)(9.5,3.48)
\psbezier[linecolor=black, linewidth=0.026, arrowsize=0.05291667cm 2.0,arrowlength=1.4,arrowinset=0.0]{->}(9.9,-3.42)(10.3,-2.42)(10.4,-1.32)(10.4,-0.52)
\rput[bl](10.25,-0.32){$S$}
\rput[bl](9.55,-3.77){$\Delta$}
\psline[linecolor=black, linewidth=0.026, arrowsize=0.05291667cm 2.0,arrowlength=1.4,arrowinset=0.0]{->}(9.5,-3.42)(9.3,-3.22)
\psline[linecolor=black, linewidth=0.026, arrowsize=0.05291667cm 2.0,arrowlength=1.4,arrowinset=0.0]{->}(9.7,-4.32)(9.7,-4.02)
\rput[bl](9.59,-4.82){$\Lambda$}
\rput[bl](2.9,1.48){$b_{\pi(2)}$}
\rput[bl](8.8,1.48){$b_{\pi(g-1)}$}
\psline[linecolor=black, linewidth=0.026, arrowsize=0.05291667cm 2.0,arrowlength=1.4,arrowinset=0.0]{->}(0.8,1.88)(0.8,2.18)
\psline[linecolor=black, linewidth=0.026, arrowsize=0.05291667cm 2.0,arrowlength=1.4,arrowinset=0.0]{->}(3.2,1.88)(3.2,2.18)
\psline[linecolor=black, linewidth=0.026, arrowsize=0.05291667cm 2.0,arrowlength=1.4,arrowinset=0.0]{->}(9.3,1.88)(9.3,2.18)
\end{pspicture}
}
\end{figure}

%Note that $A'_i$ is an iterated coproduct (with certain automorphisms in its output legs). The last step is to use the bialgebra axiom to bring all the multiplication tensors to the top of the above tensor and all coproducts at the bottom.
\noindent Note that if $H$ is a Hopf algebra in $\SVect$, then the above expression has to be multiplied by $\s=(-1)^{|\La|(g-1)}.$
\medskip

We claim that the above tensor is $\s \cdot Z_H^{\rho}(\S,\aa,\bb^e)$ where $(\S,\aa,\bb^e)$ is the (extended) Heegaard diagram of Subsection \ref{subs: HDs}, with the orientations and basepoints specified there. To see this, first note that each $A_i$-output is joined to some input of $B_{\pi(i)}$ and the product appearing right below $A_i$ is simply adding a product to the corresponding input in $B_{\pi(i)}$. Since $B_{\pi(i)}$ is itself a product, this implies that the above expression is in the form coproduct, followed by antipodes, followed by products. Thus, it has the same form as $Z_H^{\rho}(\S,\aa,\bb^e)$ and now we have to check that each coproduct/product has the same outputs/inputs as intersection points of the corresponding $\a$ or $\b$ curve and that the $b$'s twisting the tensors come from the Fox derivatives as in (\ref{eq: Fox of Wirtinger}). Indeed, for $i=1,\dots,g-2$ the $i$-th tensor on the bottom of the above tensor is
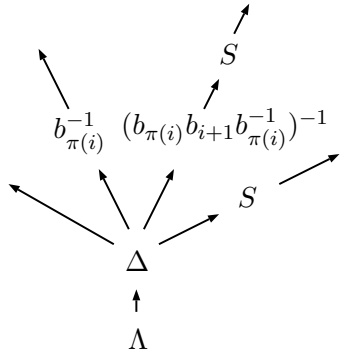
\begin{figure}[H]
\psscalebox{1.0 1.0} % Change this value to rescale the drawing.
{
\begin{pspicture}(0,-2.262859)(4.2548466,2.262859)
\rput[bl](2.7291284,1.5371408){$S$}
\rput[bl](1.4291284,0.3871408){$(b_{\pi(i)}b_{i+1}b_{\pi(i)}^{-1})^{-1}$}
\rput[bl](1.4791284,-1.2128592){$\Delta$}
\psline[linecolor=black, linewidth=0.026, arrowsize=0.05291667cm 2.0,arrowlength=1.4,arrowinset=0.0]{->}(1.7291284,-0.6628592)(2.1291285,0.13714081)
\psline[linecolor=black, linewidth=0.026, arrowsize=0.05291667cm 2.0,arrowlength=1.4,arrowinset=0.0]{->}(1.6291285,-1.7628592)(1.6291285,-1.4628592)
\rput[bl](2.9791284,-0.3628592){$S$}
\rput[bl](1.5191284,-2.262859){$\Lambda$}
\psline[linecolor=black, linewidth=0.026, arrowsize=0.05291667cm 2.0,arrowlength=1.4,arrowinset=0.0]{->}(1.9291284,-0.8628592)(2.7291284,-0.46285918)
\psline[linecolor=black, linewidth=0.026, arrowsize=0.05291667cm 2.0,arrowlength=1.4,arrowinset=0.0]{->}(3.5291283,-0.06285919)(4.3291283,0.3371408)
\psline[linecolor=black, linewidth=0.026, arrowsize=0.05291667cm 2.0,arrowlength=1.4,arrowinset=0.0]{->}(2.5291283,0.9371408)(2.7291284,1.3371408)
\psline[linecolor=black, linewidth=0.026, arrowsize=0.05291667cm 2.0,arrowlength=1.4,arrowinset=0.0]{->}(1.5291284,-0.6628592)(1.1291285,0.13714081)
\psline[linecolor=black, linewidth=0.026, arrowsize=0.05291667cm 2.0,arrowlength=1.4,arrowinset=0.0]{->}(1.3291284,-0.8628592)(-0.070871584,-0.06285919)
\rput[bl](0.52912843,0.3371408){$b_{\pi(i)}^{-1}$}
\psline[linecolor=black, linewidth=0.026, arrowsize=0.05291667cm 2.0,arrowlength=1.4,arrowinset=0.0]{->}(0.7291284,0.9371408)(0.3291284,1.7371408)
\psline[linecolor=black, linewidth=0.026, arrowsize=0.05291667cm 2.0,arrowlength=1.4,arrowinset=0.0]{->}(2.9291284,1.9371408)(3.1291285,2.3371408)
\end{pspicture}
}
\end{figure}

\noindent which is exactly the tensor corresponding to $\a_i$ as in Subsection \ref{subs: Tensors from HD's} where the $w_i$ are as in (\ref{eq: wx's positive case}). The two rightmost legs have an antipode, because the first two crossings of $\a_i$ are negative (assuming that all crossings of the diagram are positive, as we did above). Note that we show four legs coming out of the above tensor, but there may be less if those legs correspond to intersection points of $\a_i\cap \b_g$, this is because we applied $\e$ to the output leg corresponding to $\b_g$. The $(g-1)$-th coproduct above is only a triple coproduct (or less), this is because $\a_{g-1}$ intersects the arc $\b_g$. Similarly, it is easy to see that the $i$-th product on the top of the above tensor corresponds to the curve $\b_i$. Note that the tensor for $\b_1$ is slightly different, this is because there is not an $\a$ enclosing the top underarc of the diagram. This shows our claim under the assumption that all crossings of the diagram are positive. Whenever the crossing at $a_i$ is negative, $a_i$ has a black bead $A_i=S(h^i)$ followed by a white bead $\v_{b_{\pi(i)}^{-1}}$. This follows from formula (\ref{eq: twisted inverse R-matrix}) and by sliding the white bead through the white bead. Then, after proceeding as above and using that $S$ is an algebra anti-automorphism, the $i$-th coproduct in the last tensor would instead be

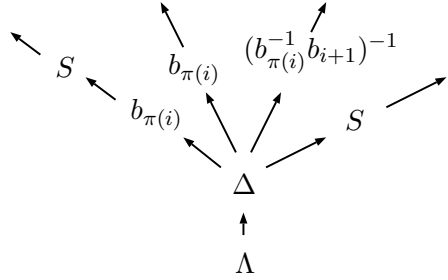
\begin{figure}[H]
\begin{pspicture}(0,-1.7628592)(5.6632714,1.7628592)
\rput[bl](3.0375538,0.8871408){$(b_{\pi(i)}^{-1}b_{i+1})^{-1}$}
\rput[bl](2.8875537,-0.7128592){$\Delta$}
\psline[linecolor=black, linewidth=0.026, arrowsize=0.05291667cm 2.0,arrowlength=1.4,arrowinset=0.0]{->}(3.1375537,-0.16285919)(3.5375538,0.6371408)
\psline[linecolor=black, linewidth=0.026, arrowsize=0.05291667cm 2.0,arrowlength=1.4,arrowinset=0.0]{->}(3.0375538,-1.2628592)(3.0375538,-0.9628592)
\rput[bl](4.3875537,0.13714081){$S$}
\rput[bl](2.9275537,-1.7628592){$\Lambda$}
\psline[linecolor=black, linewidth=0.026, arrowsize=0.05291667cm 2.0,arrowlength=1.4,arrowinset=0.0]{->}(3.3375537,-0.3628592)(4.1375537,0.03714081)
\psline[linecolor=black, linewidth=0.026, arrowsize=0.05291667cm 2.0,arrowlength=1.4,arrowinset=0.0]{->}(4.937554,0.43714082)(5.7375536,0.8371408)
\psline[linecolor=black, linewidth=0.026, arrowsize=0.05291667cm 2.0,arrowlength=1.4,arrowinset=0.0]{->}(3.9375536,1.4371408)(4.1375537,1.8371408)
\psline[linecolor=black, linewidth=0.026, arrowsize=0.05291667cm 2.0,arrowlength=1.4,arrowinset=0.0]{->}(2.9375536,-0.16285919)(2.5375538,0.6371408)
\psline[linecolor=black, linewidth=0.026, arrowsize=0.05291667cm 2.0,arrowlength=1.4,arrowinset=0.0]{->}(2.7375536,-0.3628592)(2.2375536,0.03714081)
\rput[bl](2.0375538,0.73714083){$b_{\pi(i)}$}
\psline[linecolor=black, linewidth=0.026, arrowsize=0.05291667cm 2.0,arrowlength=1.4,arrowinset=0.0]{->}(2.2375536,1.2371408)(1.9375538,1.8371408)
\rput[bl](1.5375537,0.13714081){$b_{\pi(i)}$}
\psline[linecolor=black, linewidth=0.026, arrowsize=0.05291667cm 2.0,arrowlength=1.4,arrowinset=0.0]{->}(1.3375537,0.5371408)(0.9375537,0.8371408)
\rput[bl](0.5375537,0.8371408){$S$}
\psline[linecolor=black, linewidth=0.026, arrowsize=0.05291667cm 2.0,arrowlength=1.4,arrowinset=0.0]{->}(0.3375537,1.1371408)(-0.06244629,1.4371408)
\end{pspicture}
\end{figure}
\noindent which is again the tensor defined in Subsection \ref{subs: Tensors from HD's} by (\ref{eq: wx's negative case}). Therefore, we have shown that 
\begin{equation*}
\e_{D(H)}(\rt_{\uDH}^{\rho}(T))=\e_{D(H)}(z'_D)=\rH(\rho(m))^{-r/2}(-1)^{|\La|(g-1)}Z_H^{\rho}(\S,\aa,\bb^e)
\end{equation*}
so that
\begin{equation*}
P_H^{\rho}(K)=\rH(\rho(m))^{\frac{w(T)-r}{2}}(-1)^{|\La|(g-1)}Z_H^{\rho}(\S,\aa,\bb^e)\dot{=}I_H^{\rho}(M,\c)
\end{equation*}
where $\dot{=}$ means equality up to the above indeterminacy (since $w(T)-r\in 2\Z$).% But since $rt_{\uDH}^{\rhoc}$ is symmetric in $t$ and $I$
\medskip

\end{proof}

%When $|\La|=1$, the sign indeterminacy comes from the orientations and order of $\aa$ and $\bb$. In the case $M=S^3\sm K$, we have $\det(\b_j\cdot \a_i)\neq 0$ so we can remove the sign indeterminacy by multiplying $Z_H^{\rho}(\S,\aa,\bb^e)$ by $\s_0^{|\La|}$ where $\s_0$ is the sign of $\det(\b_j\cdot\a_i)$. Note that with our orientations on $\aa$ and $\bb$, the matrix $(\b_j\cdot\a_i)$ is upper triangular with $-1$'s on the diagonal so $\det(\b_j\cdot\a_i)=(-1)^{g-1}$. 

Note that if $\rho:\pi_1(S^3\sm K)\to \Ker(\rH)$, then
\begin{equation}
\label{eq: thm no indeterminacy}
P_H^{\rho}(K)=(\pm 1)^{|\La|} I_H^{\rho}(M,\c).
\end{equation}
If $H$ is $\Z$-graded and we consider $\rhoc$ instead, then there is still a $(\pm 1)^{|\La|} t^{k|\La|}$ indeterminacy.
\medskip

Now let $\rho:\pi_1(S^3\sm K)\to SL(n,\C)$ be a homomorphism. If $H=\La(\C^n)$ is an exterior algebra, then $\Aut(H)=GL(n,\C)$, $\rH$ is the determinant and $\Ker(\rH)=SL(n,\C)$.

\begin{corollary}
\label{cor: twisted Alex is twisted RT}
The $SL(n,\C)$-twisted Reidemeister torsion of the complement of $K$ is recovered as a Reshetikhin-Turaev invariant from a twisted Drinfeld double of an exterior algebra $\La(\C^n)$ by
\begin{align*}
P_{\La(\C^n)}^{\rho}(K)=(\pm 1)^n \tau^{\rho}(S^3\sm K,m).
\end{align*}
The twisted Alexander polynomial $\De^{\rho}_K(t)$ of $K$ is obtained as follows:
\begin{align*}
P_{\La(\C^n)}^{\rho}(K,t)\dot{=}\tau^{\rhoc}(S^3\sm K,m)\dot{=}\det(t\rho(m)-I_n)\frac{\De^{\rho}_{K}(t)}{\De^{\rho}_{K,0}(t)}
\end{align*}
where $\dot{=}$ is equality of to multiplication by $(\pm 1)^n t^{kn}, k\in\Z$ and $\De^{\rho}_{K,0}(t)$ is the $0$-th twisted Alexander polynomial of $K$.
\end{corollary}
\begin{proof}
The first assertion follows from (\ref{eq: thm no indeterminacy}) together with \cite[Theorem 2]{LN:twisted}. Note that, as explain in Remark \ref{remark: conventions on IH}, our $I_H^{\rho}$ is the $I_H^{\rho^{-t}}$ of \cite{LN:twisted} so with our conventions we get $I_{\La(\C^n)}^{\rho}=\tau^{\rho}$. %Note that with the present conventions, the argument of \cite[Theorem 2]{LN:twisted} gives $$I^{\rho}_{\La(\C^n)}(M,\c)=\det\left(\s\left(\frac{\ov{\a_j}}{\p\b^*_i}\right)_{i,j=1,\dots,g-1}\right)$$ where $\s:\Z[\pi]\to\Z[\pi]$ is defined by $\s(g)=g^{-1}, g\in F$. This is exactly $\tau^{\rho}(M,\c)$ (and not the torsion $\tau^{\rho^{-t}}$ at the inverse-transpose as in \cite{LN:twisted}). 
The second assertion follows from our theorem at $\rhoc$ and standard theorems of Reidemeister torsion, see \cite{LN:twisted}.
\end{proof}

%MISSING RMK: When applying $\e$, the $B_i$ bead of $A_0$ and the $A_j$ bead of $B_g$ dissapear as well.

\bibliographystyle{amsplain}
\bibliography{/Users/daniel/Desktop/Daniel/TEX/bib/referencesabr}

\end{document}